\documentclass[reqno,11pt,letterpaper]{amsart}

\usepackage{amsmath,amssymb,amsthm,mathrsfs,bbm}
\usepackage{accents}
\usepackage{arydshln}
\usepackage{upgreek}
\usepackage{slashed}
\usepackage{xifthen}
\usepackage{graphicx}
\usepackage{longtable}
\usepackage[inline]{enumitem}
\usepackage{lipsum}

\usepackage{xcolor}
\definecolor{winered}{rgb}{0.7,0,0}
\definecolor{lessblue}{rgb}{0,0,0.7}

\usepackage[pdftex,colorlinks=true,linkcolor=winered,citecolor=lessblue,breaklinks=true,bookmarksopen=true]{hyperref}

\makeatletter
\newcommand{\myitem}[3]{\item[#2]\def\@currentlabel{#3}\label{#1}}
\makeatother

\usepackage{titletoc}

\makeatletter

\def\@tocline#1#2#3#4#5#6#7{
\begingroup
  \par
    \parindent\z@ \leftskip#3 \relax \advance\leftskip\@tempdima\relax
                  \rightskip\@pnumwidth plus 4em \parfillskip-\@pnumwidth
    \ifcase #1 
       \vskip 0.6em \hskip 0em 
       \or
       \or \hskip 0em 
       \or \hskip 1em 
    \fi%
    #6
    \nobreak\relax{\leavevmode\leaders\hbox{\,.}\hfill}
    \hbox to\@pnumwidth {\@tocpagenum{#7}}
  \par
\endgroup
}

 \def\l@section{\@tocline{0}{0pt}{0pc}{}{}}

\renewcommand{\tocsection}[3]{%
  \indentlabel{\@ifnotempty{#2}{ 
    \ignorespaces\bfseries{#2. #3}}}
  \indentlabel{\@ifempty{#2}{\ignorespaces\bfseries{#3}}{}} 
    \vspace{1.5pt}}

\renewcommand{\tocsubsection}[3]{%
  \indentlabel{\@ifnotempty{#2}{
    \ignorespaces#2. #3}}
  \indentlabel{\@ifempty{#2}{\ignorespaces #3}{}}
    \vspace{1.5pt}}

\renewcommand{\tocsubsubsection}[3]{%
  \indentlabel{\@ifnotempty{#2}{
    \ignorespaces#2. #3}}
  \indentlabel{\@ifempty{#2}{\ignorespaces #3}{}}
    \vspace{1.5pt}}

\makeatother

\makeatletter
\def\@nomenstarted{0}
\newlength{\@nomenoldtabcolsep}

\newcommand{\nomenstart}
  {%
    \def\@nomenstarted{1}%
    \setlength{\@nomenoldtabcolsep}{\tabcolsep}%
    \setlength{\tabcolsep}{3.5pt}%
    \begin{longtable}{p{0.11\textwidth} p{0.86\textwidth}}
  }

\newcommand{\nomenitem}[2]{%
    \ifcase\@nomenstarted%
      \or 
      \or \\ 
    \fi%
    #1\,{\leavevmode\leaders\hbox{\,.}\hfill} & #2%
    \def\@nomenstarted{2}%
  }%
\newcommand{\nomenend}
  {\\%
      \end{longtable}%
      \setlength{\tabcolsep}{\@nomenoldtabcolsep}%
      \def\@nomenstarted{0}%
  }
\makeatother

\makeatletter
\newcommand{\vast}{\bBigg@{4}}
\newcommand{\Vast}{\bBigg@{5}}
\newcommand{\VAST}{\bBigg@{6}}
\makeatother

\allowdisplaybreaks

\numberwithin{equation}{section}
\numberwithin{figure}{section}
\newtheorem{thm}{Theorem}[section]

\newtheorem{prop}[thm]{Proposition}
\newtheorem{lemma}[thm]{Lemma}
\newtheorem{cor}[thm]{Corollary}
\newtheorem*{thm*}{Theorem}
\newtheorem*{prop*}{Proposition}
\newtheorem*{cor*}{Corollary}
\newtheorem*{conj*}{Conjecture}

\theoremstyle{definition}
\newtheorem{definition}[thm]{Definition}

\theoremstyle{remark}
\newtheorem{rmk}[thm]{Remark}

\newcommand{\mc}{\mathcal}
\newcommand{\cA}{\mc A}

\newcommand{\cC}{\mc C}

\newcommand{\cE}{\mc E}
\newcommand{\cF}{\mc F}

\newcommand{\cH}{\mc H}

\newcommand{\cK}{\mc K}

\newcommand{\cO}{\mc O}

\newcommand{\cV}{\mc V}

\newcommand{\cY}{\mc Y}

\newcommand{\ms}{\mathscr}

\newcommand{\scri}{\ms I}

\newcommand{\sS}{\ms S}

\newcommand{\C}{\mathbb{C}}
\newcommand{\N}{\mathbb{N}}
\newcommand{\R}{\mathbb{R}}

\newcommand{\Sph}{\mathbb{S}}

\newcommand{\bfa}{\mathbf{a}}

\newcommand{\fm}{\mathfrak{m}}

\newcommand{\ft}{\mathfrak{t}}

\newcommand{\slg}{\slashed{g}{}}

\newcommand{\slDelta}{\slashed{\Delta}{}}

\newcommand{\vol}{\operatorname{vol}}

\renewcommand{\Re}{\operatorname{Re}}
\renewcommand{\Im}{\operatorname{Im}}

\newcommand{\mathspan}{\operatorname{span}}
\newcommand{\supp}{\operatorname{supp}}

\newcommand{\eps}{\epsilon}

\newcommand{\la}{\langle}

\newcommand{\ol}{\overline}
\newcommand{\pa}{\partial}
\newcommand{\ra}{\rangle}

\newcommand{\wh}{\widehat}
\newcommand{\wt}{\widetilde}
\newcommand{\xra}{\xrightarrow}

\newcommand{\pfstep}[1]{$\bullet$\ \underline{\textit{#1}}}
\newcommand{\pfsubstep}[1]{$-$\ \textit{#1}}

\newcommand{\bop}{{\mathrm{b}}}

\newcommand{\bface}{{\mathrm{bf}}}

\newcommand{\tface}{{\mathrm{tf}}}
\newcommand{\zface}{{\mathrm{zf}}}

\newcommand{\cp}{{\mathrm{c}}}

\newcommand{\scl}{{\mathrm{sc}}}

\newcommand{\Diff}{\mathrm{Diff}}

\newcommand{\Vb}{\cV_\bop}
\newcommand{\Diffb}{\Diff_\bop}

\newcommand{\Tsc}{{}^{\scl}T}

\newcommand{\half}{{\tfrac{1}{2}}}

\newcommand{\loc}{{\mathrm{loc}}}
\newcommand{\CI}{\cC^\infty}

\newcommand{\CIc}{\cC^\infty_\cp}

\newcommand{\Hb}{H_{\bop}}

\newcommand{\Hbext}{\bar H_{\bop}}

\newcommand{\Hbsupp}{\dot H_{\bop}}

\newcommand{\Hbh}{H_{\bop,h}}
\newcommand{\Hbhext}{\bar H_{\bop,h}}

\newcommand{\bhm}{\fm}
\newcommand{\bha}{\bfa}

\newcommand{\openbigpmatrix}[1]
  {%
    \def\@bigpmatrixsize{#1}%
    \addtolength{\arraycolsep}{-#1}%
    \begin{pmatrix}%
  }
\newcommand{\closebigpmatrix}
  {%
    \end{pmatrix}%
    \addtolength{\arraycolsep}{\@bigpmatrixsize}%
  }

\newlength{\enummargin}

\newcommand{\usref}[1]{{\upshape\ref{#1}}}

\DeclareGraphicsExtensions{.mps}

\makeatletter
\newcommand{\fakephantomsection}{%
  \Hy@MakeCurrentHref{\@currenvir.\the\Hy@linkcounter}
  \Hy@raisedlink{\hyper@anchorstart{\@currentHref}\hyper@anchorend}%
}
\makeatother

\makeatletter
\newcommand*{\fwbw}[1]{\expandafter\@fwbw\csname c@#1\endcsname}
\newcommand*{\@fwbw}[1]{\ifcase #1 \or {\rm fw}\or {\rm bw}\fi}
\AddEnumerateCounter{\fwbw}{\@fwbw}
\makeatother

\begin{document}

\title[Price's law]{A sharp version of Price's law for wave decay on asymptotically flat spacetimes}

\author{Peter Hintz}
\address{Department of Mathematics, Massachusetts Institute of Technology, Cambridge, Massachusetts 02139-4307, USA}
\email{phintz@mit.edu}

\subjclass[2010]{Primary 58J50, Secondary 83C57, 35L05, 35C20}

\begin{abstract}
  We prove Price's law with an explicit leading order term for solutions $\phi(t,x)$ of the scalar wave equation on a class of stationary asymptotically flat $(3+1)$-dimensional spacetimes including subextremal Kerr black holes. Our precise asymptotics in the full forward causal cone imply in particular that $\phi(t,x)=c t^{-3}+\cO(t^{-4+})$ for bounded $|x|$, where $c\in\C$ is an explicit constant. This decay also holds along the event horizon on Kerr spacetimes and thus renders a result by Luk--Sbierski on the linear scalar instability of the Cauchy horizon unconditional. We moreover prove inverse quadratic decay of the radiation field, with explicit leading order term. We establish analogous results for scattering by stationary potentials with inverse cubic spatial decay. On the Schwarzschild spacetime, we prove pointwise $t^{-2 l-3}$ decay for waves with angular frequency at least $l$, and $t^{-2 l-4}$ decay for waves which are in addition initially static. This definitively settles Price's law for linear scalar waves in full generality.
  
  The heart of the proof is the analysis of the resolvent at low energies. Rather than constructing its Schwartz kernel explicitly, we proceed more directly using the geometric microlocal approach to the limiting absorption principle pioneered by Melrose and recently extended to the zero energy limit by Vasy.
\end{abstract}

\date{\today. Original version: April 3, 2020.}

\maketitle

\section{Introduction}
\label{SI}

The Schwarzschild spacetime \cite{SchwarzschildPaper} with mass $\bhm>0$ is a spherically symmetric solution of the Einstein vacuum equation given by
\begin{equation}
\label{EqISchw}
  g = -\Bigl(1-\frac{2\bhm}{r}\Bigr)d t^2 + \Bigl(1-\frac{2\bhm}{r}\Bigr)^{-1}d r^2 + r^2\slg
\end{equation}
on $\R_t\times(2\bhm,\infty)_r\times\Sph^2$, where $\slg$ is the standard metric on $\Sph^2$. To describe our main result in a simple setting, we consider the initial value problem
\begin{equation}
\label{EqIIVP}
  \Box_g\phi = -|g|^{-1/2}\pa_\mu(|g|^{1/2}g^{\mu\nu}\pa_\nu\phi)=0,\quad
  \phi(0,x)=\phi_0(x),\quad
  \pa_t\phi(0,x)=\phi_1(x),
\end{equation}
with compactly supported and smooth initial data $\phi_0,\phi_1\in\CIc((2\bhm,\infty)\times\Sph^2)$.

\begin{thm}[Price's law on the Schwarzschild spacetime]
\label{ThmI}
  Let $\phi$ denote the solution of equation~\eqref{EqIIVP}. Fix a compact subset $K\Subset(2\bhm,\infty)\times\Sph^2\subset\R^3_x$.
  \begin{enumerate}
  \item\label{ItILOT} There exists a constant $c\in\C$ so that $\phi(t,x)$ decays according to
  \begin{equation}
  \label{EqILOT}
    |\phi(t,x)-c t^{-3}| \leq C_\eps t^{-4+\eps},\quad x\in K,
  \end{equation}
  for any $\eps>0$. Derivatives of $\phi-c t^{-3}$ along any finite number of the vector fields $t\pa_t$ and $\pa_x$ satisfy the same estimate (with different $C_\eps$). Explicitly, $c$ is given by
  \begin{equation}
  \label{EqILOTConst}
    c = -\frac{2\bhm}{\pi}\iiint_{r>2\bhm} \Bigl(1-\frac{2\bhm}{r}\Bigr)^{-1}\phi_1(r,\theta,\varphi)\,r^2\sin\theta\,d r\,d\theta\,d\varphi.
  \end{equation}
  \item\label{ItIPrice} If $l\in\N_0$ and $\phi_0,\phi_1$ are supported in angular frequencies $\geq l$ (meaning that for all $r$, $\phi_j(r,-)\in\CI(\Sph^2)$ is orthogonal to the eigenspaces of $\Delta_\slg$ with eigenvalues $k(k+1)$ for $k=0,\ldots,l-1$), then
  \begin{equation}
  \label{EqIPrice}
    |\phi(t,x)|\leq C t^{-2 l-3},\quad x\in K,
  \end{equation}
  and the same decay holds for derivatives of $\phi$ along $t\pa_t$ and $\pa_x$. This decay rate is generically sharp.
  \item\label{ItIStatic} In both cases, if $\phi$ is initially static, i.e.\ $\phi_1\equiv 0$, then the decay rate of $\phi$ is faster by one power of $t^{-1}$.
  \end{enumerate}
\end{thm}

We describe a more general result momentarily. Price \cite{PriceLawI,PriceLawII}, as clarified by Price and Burko \cite{PriceBurkoLaw}, conjectured these decay rates in the 1970s. Pointwise $t^{-3}$ decay was proved by Donninger--Schlag--Soffer \cite{DonningerSchlagSofferSchwarzschild} for Schwarzschild spacetimes and by Tataru \cite{TataruDecayAsympFlat} on a general class of stationary asymptotically flat spacetimes which includes Schwarzschild and subextremal Kerr spacetimes \cite{KerrKerr}. Parts~\eqref{ItIPrice}--\eqref{ItIStatic} (see Corollary~\ref{CorPF}) constitute the definitive resolution of Price's conjecture for linear scalar waves; they improve on the pointwise $t^{-2 l-2}$ decay ($t^{-2 l-3}$ for initially static perturbations) established by Donninger--Schlag--Soffer \cite{DonningerSchlagSofferPrice} by one power of $t^{-1}$; in fact, we control the infinite sum over all spherical harmonic modes with frequency at least $l$, rather than merely individual modes. (See also \cite{LeaverSchwarzschild} for a heuristic description of the full time evolution.) Angelopoulos--Aretakis--Gajic \cite{AngelopoulosAretakisGajicLate} gave the first rigorous derivation of the leading order term in~\eqref{EqILOTConst} on a class of spherically symmetric, stationary, and asymptotically flat spacetimes including Schwarzschild and subextremal Reissner--Nordstr\"om spacetimes.

Theorem~\ref{ThmI}\eqref{ItILOT} is a consequence of a partial expansion of the resolvent $\wh{\Box_g}(\sigma)^{-1}$ at $\sigma=0$. Using a novel systematic and, to a large degree, algorithmic method, we show, roughly speaking, that the strongest singularity of $\wh{\Box_g}(\sigma)^{-1}$, acting on inputs with compact support (or more generally satisfying almost-sharp decay assumptions), is $\sigma^2\log(\sigma+i 0)$, and we compute its coefficient explicitly; see~\S\ref{SsIP} and Theorem~\ref{ThmPRes} for further details. The study of the low energy resolvent has a long history, starting with the work by Jensen--Kato \cite{JensenKatoResolvent} on Euclidean space. Recent works describe qualitative \cite{BonyHaefnerResolvent,VasyWunschMorawetz} and quantitative bounds \cite{RodnianskiTaoResolvent} as well as Hahn-meromorphic properties \cite{MullerStrohmaierResolvent} of the resolvent, and the connection between the low energy resolvent behavior and the cohomology of the spatial manifold \cite{StrohmaierWatersHodge}. Here, we adopt Vasy's perspective \cite{VasyLowEnergyLag,VasyLAPLag} and obtain the resolvent expansion in a direct manner, rather than by adapting the Schwartz kernel constructions of Guillarmou--Hassell and Sikora \cite{GuillarmouHassellResI,GuillarmouHassellResII,GuillarmouHassellSikoraResIII} (discussed further below).

\subsection{Sharp asymptotics on subextremal Kerr spacetimes}

In order to describe radiation falling into the black hole or escaping to infinity, it is convenient to introduce a new time coordinate $t_*$ whose level sets are transversal to the future event horizon and to future null infinity. Indeed, one can choose $t_*$ to be roughly equal to $t+r_*$ near the event horizon $r=2\bhm$ and $t-r_*$ for large $r$, where $r_*=r+2\bhm\log(r-2\bhm)$ is the Regge--Wheeler tortoise coordinate. The Schwarzschild metric $g$, expressed using $t_*$ instead of $t$, can then be extended smoothly across the event horizon, and is a stationary Lorentzian metric on\footnote{The sphere $r=\bhm$ inside the black hole where we stop keeping track of waves is chosen arbitrarily.}
\[
  \R_{t_*}\times X^\circ,\qquad X^\circ = [\bhm,\infty)_r\times\Sph^2\subset\R^3.
\]
See Figure~\ref{FigI2}. The family of subextremal Kerr metrics, described in~\S\ref{SK}, generalizes the Schwarzschild metric and describes rotating black holes with angular momentum $\bha\in(-\bhm,\bhm)$ and event horizon at $r=r_{\bhm,\bha}:=\bhm+\sqrt{\bhm^2-\bha^2}>\bhm$.

\begin{thm}[Price's law with leading order term on subextremal Kerr spacetimes]
\label{ThmI2}
  Let $g$ be a subextremal Kerr metric. Consider compactly supported initial data $\phi_0,\phi_1\in\CIc(X^\circ)$. Then, for a constant $c$, explicitly computable in terms of $\phi_0,\phi_1$, the solution $\phi$ of the initial value problem $\Box_g\phi=0$, $\phi|_{t_*=0}=\phi_0$, $\pa_{t_*}\phi|_{t_*=0}=\phi_1$, decays according to\footnote{The factor $\frac{t_*}{t_*+r}$ vanishes simply at null infinity but is positive in timelike cones $r<(1-\delta)t$, $\delta>0$.}
  \begin{equation}
  \label{EqI2Phi}
    \Bigl| \phi - c\frac{t_*+r}{t_*^2(t_*+2 r)^2} \Bigr| \leq C_\eps t_*^{-4+\eps}\frac{t_*}{t_*+r}.
  \end{equation}
  In particular, the radiation field\footnote{Some authors define the radiation field as the $t_*$-derivative of $F$. We follow the original \cite{FriedlanderRadiationOrig}.} $F(t_*,\omega) := \lim_{r\to\infty} r\phi(t_*,r,\omega)$ has a leading order term with remainder,
  \begin{equation}
  \label{EqI2Rad}
    |F - \tfrac14 c t_*^{-2}| \leq C_\eps t_*^{-3+\eps}.
  \end{equation}
  The decay rates in~\eqref{EqI2Phi} and \eqref{EqI2Rad} hold for all derivatives of $\phi-c\frac{t_*+r}{t_*^2(t_*+2 r)^2}$ and $F-\tfrac14 c t_*^{-2}$ along any finite number of the vector fields $t_*\pa_{t_*}$, rotation vector fields on $\Sph^2$, and $r\pa_r$ in case of~\eqref{EqI2Phi}.
\end{thm}

See~Theorem~\ref{ThmKW}. Figure~\ref{FigI2} shows the Penrose diagram and a resolution (blow-up) well-adapted to the description of global asymptotics (described in detail in Definition~\ref{DefPWComp}).

\begin{figure}[!ht]
\centering
\includegraphics{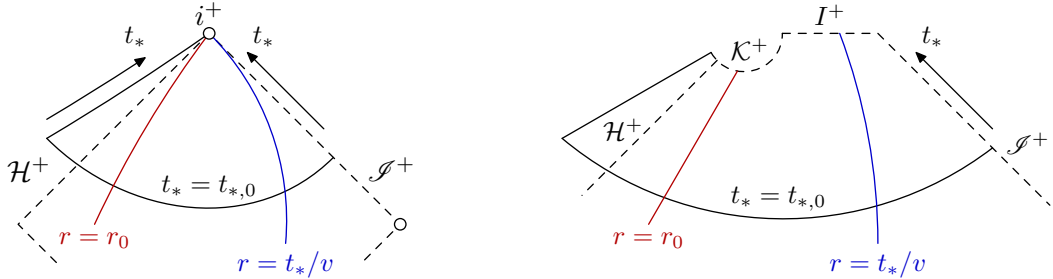}
\caption{\textit{On the left:} Penrose diagram of a Schwarzschild or subextremal Kerr spacetime, with level sets of $t_*$ (black), $r$ (red), and $t_*/r$ (blue) indicated. Also shown are $\scri^+$ (future null infinity), $\cH^+$ (the event horizon), and $i^+$ (future timelike infinity). A hypersurface $t_*/r=v\in(0,\infty)$ is a timelike cone asymptotic to the cone $r=\frac{1}{1+v}t$. \textit{On the right:} Resolution of the Penrose diagram obtained by first blowing up $i^+$ (obtaining $\cK^+$) and then the lift of the future boundary of $\scri^+$ (obtaining $I^+$). The asymptotics $c t_*^{-3}$ govern decay at $\cK^+$, while the profile~\eqref{EqI2Phi} gives asymptotics at $I^+$ (which match $c t_*^{-3}$ at $\cK^+\cap I^+$).}
\label{FigI2}
\end{figure}

\begin{rmk}
  For initial data supported away from the event horizon, a simple explicit expression for the constant $c$ using Boyer--Lindquist coordinates is given in Corollary~\ref{CorKWExpl}.
\end{rmk}

\begin{rmk}
  In any region $r_{\bhm,\bha}<r_0<r<(1-\delta)t$ away from the event horizon and away from null infinity, we can replace $t_*$ in~\eqref{EqI2Phi} by $t-r$ up to $t^{-4+\eps}$ errors and thus obtain
  \[
    \Bigl|\phi - c\frac{t}{(t^2-r^2)^2}\Bigr| \lesssim t^{-4+\eps},\qquad r_0<r<(1-\delta)t.
  \]
  We note that $t/(t^2-r^2)^2$ is an exact solution of the wave equation of Minkowski space. For more details, see Remark~\ref{RmkKWBL}.
\end{rmk}

\begin{rmk}
  The constant $C_\eps$ in~\eqref{EqI2Phi} is bounded by a universal constant times $\|\phi_0\|_{H^N}+\|\phi_1\|_{H^{N-1}}$ for some large $N$ as long as $\phi_0,\phi_1$ have support in a fixed compact subset of $X^\circ$. In order to simplify the exposition, we do not keep track of the number of derivatives or the precise decay assumptions (except for forcing problems, see Theorems~\ref{ThmPW}, \ref{ThmPWGlobal}, \ref{ThmKW}). The interested reader can find a concrete value for $N$ by carefully revisiting our arguments.
\end{rmk}

\begin{rmk}
\label{RmkIPriceL}
  In the context of part~\eqref{ItIPrice} of Theorem~\ref{ThmI}, we prove $t_*^{-l-3}$ decay of $\phi$ in future timelike cones, and $t_*^{-l-2}$ decay of the radiation field of $\phi$; for initially static perturbations, the decay rates are improved by $1$. See Theorem~\ref{ThmPF}. Generalizations of such $l$-dependent decay rates to Kerr spacetimes have been discussed in the physics literature \cite{GleiserPricePullinKerrTails,BurkoKhannaKerrTails}.
\end{rmk}

The constant $c$ in Theorem~\ref{ThmI2} only vanishes on a codimension $1$ subspace of initial data. Thus, the restriction of $|\pa_{t_*}\phi|^2$ to the event horizon of the black hole generically obeys a pointwise \emph{lower} (and upper) bound of $t_*^{-6}$. This proves Conjecture~1.9 in the paper \cite{LukSbierskiKerr} by Luk--Sbierski and thus implies that generic smooth and compactly supported Cauchy data on subextremal Kerr spacetimes give rise to solutions of the scalar wave equation for which the nondegenerate energy on any spacelike hypersurface transversal to the Cauchy horizon is infinite (see \cite[Conjecture~1.7]{LukSbierskiKerr}). Indeed, the upper and lower bounds in assumptions (i) and (ii) of their main theorem, \cite[Theorem~3.2]{LukSbierskiKerr}, hold for $q=5$, $\delta=0$.

The asymptotic behavior~\eqref{EqI2Phi} holds more generally on any stationary and asymptotically flat (with mass $\bhm\in\R$) spacetime, a notion we introduce in~\S\ref{SA}. Roughly speaking, these are spacetimes whose metrics have a $2\bhm/r$ long range term just like the Schwarzschild metric $g_\bhm$, plus lower order perturbations which decay at a rate of at least $r^{-2}$. We need to assume the absence of zero energy bound states or resonances (smooth stationary solutions of $\Box_g\phi=0$ with $|\phi|\lesssim r^{-1}$ for large $r$), the absence of nontrivial mode solutions which are purely oscillatory or exponentially growing as $t_*\to\infty$, and high energy estimates for the resolvent; in concrete situations, the latter can typically be proved easily using microlocal methods. See Definitions~\ref{DefAF} and \ref{DefASAssm} and the results in \S\ref{SsPW}.

\begin{rmk}
\label{RmkIDyn}
  Waves on dynamical spacetimes which merely \emph{settle down} to a stationary spacetime can not be described in one fell swoop using spectral methods. Rather, as demonstrated on asymptotically de~Sitter \cite{HintzVasySemilinear,HintzQuasilinearDS} and Kerr--de~Sitter spacetimes \cite{HintzVasyQuasilinearKdS}, in particular in the proof of the nonlinear stability of slowly rotating Kerr--de~Sitter black holes \cite{HintzVasyKdSStability}, the analysis of the stationary wave equation (in these settings based on \cite{SaBarretoZworskiResonances,BonyHaefnerResolvent,MelroseSaBarretoVasySdS,DyatlovQNM,VasyMicroKerrdS,HintzPsdoInner}) is one step in a two-step analysis. Namely, microlocal methods allow the control of high frequencies of waves on dynamical spacetimes, while their decay is controlled using precise decay results on the stationary model spacetime (typically with a loss of regularity); together, this controls waves up to \emph{compact} error terms (on a scale of weighted Sobolev spaces) which can then be dropped in perturbative regimes. Details of this approach on asymptotically flat spacetimes will be given in future work.
\end{rmk}

In order to put Theorems~\ref{ThmI} and \ref{ThmI2} into context, we recall that Angelopoulos, Aretakis, and Gajic are pursuing a program aimed at a detailed asymptotic description of waves on spherically symmetric spacetimes, including both subextremal and extremal black hole spacetimes. In particular, in \cite{AngelopoulosAretakisGajicVF}, they prove almost sharp inverse polynomial decay on spherically symmetric, stationary, asymptotically flat spacetimes using energy estimate and vector field methods. In the aforementioned companion paper \cite{AngelopoulosAretakisGajicLate}, they give the first rigorous proof of a $t_*^{-3}$ leading order term in compact spatial regions, a $t_*^{-2}$ leading order term for the radiation field (confirming predictions of Gundlach--Price--Pullin \cite{GundlachPricePullin}), and the asymptotic profile in the full forward light cone; key to their arguments are certain conservation laws at null infinity. Our results on Kerr (or more general) spacetimes remove the assumption that the underlying spacetime be spherically symmetric; as we shall discuss in~\S\ref{SsIV} below, we also allow the coupling of scalar waves to stationary potentials. On the other hand, unlike \cite{AngelopoulosAretakisGajicLate}, we do not keep careful track of the number of derivatives used. The subsequent work \cite{AngelopoulosAretakisGajicLog} goes one step further and computes the first \emph{subleading} $t_*^{-3}\log t_*$ term of the radiation field for spherically symmetric waves, confirming heuristic arguments by G\'omez--Winicour--Schmidt \cite{GomezWinicourSchmidtTails}. These leading and subleading terms are the first two terms of a (conjectural) full polyhomogeneous expansion of linear scalar waves $\phi$ on Kerr (or more general) spacetimes.

\begin{rmk}
\label{RmkIPhg}
  In~\S\ref{SssPWGlobal}, we define a compactification of $[0,\infty)_{t_*}\times X^\circ$ to a manifold with corners on which we conjecture $\phi$ to be polyhomogeneous; see Figures~\ref{FigI2} and \ref{FigPWCpt}.
\end{rmk}

On asymptotically Minkowski spacetimes, Baskin--Vasy--Wunsch \cite{BaskinVasyWunschRadMink,BaskinVasyWunschRadMink2} show the polyhomogeneity of scalar waves on a resolution of the radial compactification of $\R^4$ at the boundaries at infinity of the future and past light cones. The spacetimes under consideration are required to be well-behaved with respect to the dilation action $(t,x)\mapsto(\lambda t,\lambda x)$; in particular, stationary perturbations are not permitted. Baskin--Marzuola \cite{BaskinMarzuolaCone} (see also \cite{BaskinMarzuolaComp}) extended these results to allow for conic singularities of the metric on a cross section of the dilation action. This is directly related to the profile appearing in~\eqref{EqI2Phi}: in the terminology of \cite{BaskinVasyWunschRadMink,BaskinVasyWunschRadMink2,BaskinMarzuolaCone}, this profile is, under suitable identifications, a resonant state of exact hyperbolic space with a conic singularity at $r=0$; and indeed $I^+$ is equal to the blow-up of the `north cap', denoted $C_+$ in the references, at the `north pole'.

Guillarmou--Hassell and Sikora \cite{GuillarmouHassellResI,GuillarmouHassellResII,GuillarmouHassellSikoraResIII} give a complete description of the Schwartz kernel of the low energy resolvent $(-\sigma^2+\Delta_g+V)^{-1}$ for potential scattering on asymptotically conic (or flat) spaces as a polyhomogeneous distribution on a suitable resolved space (which includes $(0,1)_\sigma\times X^\circ\times X^\circ$ as an open dense submanifold). Via the inverse Fourier transform, this (together with bounds for bounded and high energies) gives full polyhomogeneous expansions of linear waves on a compactification of the spacetime. Their setup does not directly apply to Schwarzschild or Kerr spacetimes but, in concert with \cite{HassellVasyResolvent} does cover wave equations on Riemannian manifolds whose metrics, in dimension $3$, have a long range mass term $2\bhm/r$ of the same type as the Schwarzschild metric.

The proofs in \cite{DonningerSchlagSofferPrice} of the first $l$-dependent pointwise decay rates $t^{-2 l-2}$ in the context of Theorem~\ref{ThmI}, as well as the subsequent \cite{DonningerSchlagSofferSchwarzschild}, control the spectral measure for low frequencies using separation of variables techniques. (See Finster--Kamran--Smoller--Yau \cite{FinsterKamranSmollerYauKerr} for a similar approach on Kerr spacetimes.) Tataru~\cite{TataruDecayAsympFlat} proves $t^{-3}$ decay in large generality on a class of asymptotically flat and stationary spacetimes under the assumption that local energy decay estimates hold; these estimates hold on subextremal Kerr spacetimes, as discussed below (see also \cite{MetcalfeSterbenzTataru}). The metric asymptotics assumed in \cite{TataruDecayAsympFlat} are quite weak: Tataru allows even the long range perturbations to be merely conormal, in contrast to our $2\bhm/r$ leading order term which, however, is key for getting the \emph{leading order term} rather than merely a $\cO(t_*^{-3})$ \emph{upper bound}. (Our assumptions on \emph{short} range perturbations in Definition~\ref{DefAF} can easily be relaxed to conormality, see Remark~\ref{RmkAConormal}.) His method allows for the coupling to stationary potentials with $r^{-3}$ decay; these fit into our framework as well, as discussed in~\S\ref{SsIV} below. Metcalfe--Tataru--Tohaneanu \cite{MetcalfeTataruTohaneanuPriceNonstationary} subsequently established Price's law on nonstationary spacetimes with suitable decay towards stationarity. Unlike \cite{TataruDecayAsympFlat} and the present paper, the proof in \cite{MetcalfeTataruTohaneanuPriceNonstationary} does not make use of the Fourier transform in time, but rather combines local energy decay with the explicit solution of the constant coefficient d'Alembertian. The same authors also prove $t^{-4}$ decay for the Maxwell equation on Schwarzschild spacetimes \cite{MetcalfeTataruTohaneanuMaxwellSchwarzschild}. In this case, there is a zero resonance, which gives rise to the stationary Coulomb solution (and is dealt with in an ad hoc manner). On the spectral side, this corresponds to a first order pole of the resolvent, the sharp analysis of which is beyond the scope of the present paper; see \cite{HaefnerHintzVasyKerr} for weaker results in a related context.

There is a large amount of literature on wave decay on perturbations of Minkowski space; besides the above references, we mention in particular the work by Christodoulou and Klainerman \cite{KlainermanGlobal,ChristodoulouGlobalSolutionsSmallData,ChristodoulouKlainermanStability}, Lindblad--Rodnianski \cite{LindbladRodnianskiGlobalStability}, and references therein.

Boundedness of linear waves on the Schwarzschild was first proved by Wald and Kay--Wald \cite{WaldSchwarzschild,KayWaldSchwarzschild}. A robust approach based on carefully chosen vector field multipliers and commutators was pioneered by Dafermos--Rodnianski \cite{DafermosRodnianskiRedShift}, with many subsequent improvements, see e.g.\ \cite{LukSchwarzschild,MoschidisRp}. Previously, Blue--Soffer \cite{BlueSofferSchwarzschildDecay,BlueSofferSchwarzschildSpin2} proved local decay estimates using Morawetz estimates generalizing \cite{MorawetzExteriorDecay}. Dafermos--Rodnianski \cite{DafermosRodnianskiPrice} proved sharp $t^{-3}$ decay in a nonlinear, albeit spherically symmetric setting.

Local energy decay estimates and pointwise decay of linear scalar waves on Kerr spacetimes were first obtained for very small angular momenta by Andersson--Blue \cite{AnderssonBlueHiddenKerr}, Dafermos--Rodnianski \cite{DafermosRodnianskiKerrDecaySmall}, and Tataru--Tohaneanu \cite{TataruTohaneanuKerrLocalEnergy}, and established in the full subextremal range by Dafermos--Rodnianski--Shlapentokh-Rothman \cite{DafermosRodnianskiShlapentokhRothmanDecay}; the spectral theoretic input is the mode stability proved by Whiting~\cite{WhitingKerrModeStability} and Shlapentokh-Rothman \cite{ShlapentokhRothmanModeStability}. Strichartz estimates were proved in \cite{MarzuolaMetcalfeTataruTohaneanuStrichartz,TohaneanuKerrStrichartz}. See Aretakis \cite{AretakisExtremalKerr} for the extremal case. Results for semilinear and quasilinear equations were proved by Luk \cite{LukKerrNonlinear}, Lindblad--Tohaneanu \cite{LindbladTohaneanuSchwarzschildQuasi,LindbladTohaneanuKerrQuasi}, and for the Einstein equation by Klainerman--Szeftel \cite{KlainermanSzeftelPolarized}. For further results on tensor-valued waves on Schwarzschild, we refer the reader to \cite{FinsterKamranSmollerYauDiracKerrNewman,SterbenzTataruMaxwellSchwarzschild,BlueMaxwellSchwarzschild,DafermosHolzegelRodnianskiSchwarzschildStability,JohnsonSchwarzschild,HungSchwarzschildOdd,HungSchwarzschildEven,PasqualottoMaxwell}; for subextremal Kerr spacetimes, see \cite{AnderssonBlueMaxwellKerr,DafermosHolzegelRodnianskiTeukolsky,AnderssonBackdahlBlueMaKerr,HaefnerHintzVasyKerr}.

\subsection{Sharp asymptotics for wave equations with stationary potentials}
\label{SsIV}

On subextremal Kerr spacetimes (or generalizations as in~\S\ref{SA}), we can couple scalar waves to stationary complex-valued potentials $V$ with $r^{-3}$ decay at infinity under spectral conditions on $\Box_g+V$ as before (absence of bound states and nontrivial nondecaying mode solutions; high energy estimates). The asymptotics~\eqref{EqI2Phi} continue to hold in every cone $\delta t_*<r<(1-\delta)t_*$, $\delta>0$; however, the asymptotic behavior in compact spatial sets is modified: one has
\[
  |\phi(t_*,x) - u_{(0)}t_*^{-3}| \lesssim t_*^{-4+\eps}
\]
where $u_{(0)}$ is an `extended bound state': $u_{(0)}$ is the unique stationary solution of $(\Box_g+V)u_{(0)}=0$ which for large $r$ is equal to a constant $c$ plus $\cO(r^{-1+\eps})$ corrections. Here, $c$ is equal to the $L^2$ inner product of a linear combination of the initial data with an `extended dual bound state' $u_{(0)}^*$ which solves $(\Box_g+V)^*u_{(0)}^*=0$. We illustrate this on Minkowski space $\R_t\times\R^3_x$ with metric $-d t^2+d x^2$; even in this setting, the result appears to be new:

\begin{thm}[Sharp asymptotics for wave equations with stationary potentials in a simple special case]
\label{ThmIV}
  Let $V\in\CI(\R^3)$ be a potential which in $r>1$ is of the form $V(r,\omega)=r^{-3}W(r^{-1},\omega)$, where $W(\rho,\omega)\in\CI([0,1)\times\Sph^2)$. Define
  \[
    \bar V_0:=(4\pi)^{-1}\int_{\Sph^2}W(0,\omega)\,d\omega.
  \]
  Suppose that the resolvent\footnote{Here, we use the nonnegative Laplacian $\Delta_{\R^3}\geq 0$.} $(-\sigma^2+\Delta_{\R^3}+V)^{-1}\colon L^2(\R^3)\to L^2(\R^3)$ extends analytically from $\Im\sigma\gg 1$ to $\Im\sigma>0$, and is continuous down to $\R_\sigma$ as a map $\CIc(\R^3)\to\CI(\R^3)$. Define $u_{(0)}\in\CI(\R^3)$ as the (unique) solution of $(\Delta_{\R^3}+V)u_{(0)}=0$ such that $u_{(0)}\to 1$ as $r\to\infty$, and denote by $u_{(0)}^*=\ol{u_{(0)}}$ the corresponding solution of $(\Delta_{\R^3}+\bar V)u_{(0)}^*=0$. Given smooth, compactly supported initial data $\phi_0,\phi_1\in\CIc(\R^3)$, let $\phi(t,x)$ denote the solution of the initial value problem
  \begin{equation}
  \label{EqIVEqn}
    (-D_t^2+\Delta_{\R^3}+V)\phi=0,\quad
    \phi(0,x)=\phi_0(x),\quad
    \pa_t\phi(0,x)=\phi_1(x).
  \end{equation}
  Then, for $x$ restricted to any fixed compact subset of $\R^3$, we have
  \[
    |\phi(t,x) - c u_{(0)}t^{-3}| \leq C_\eps t^{-4+\eps},\qquad
    c = -\frac{\bar V_0}{\pi}\la\phi_1,u_{(0)}^*\ra_{L^2(\R^3)}.
  \]
\end{thm}

See \S\ref{SsPV} for the general result, and the discussion following~\eqref{EqPVLongRange} for the proof of Theorem~\ref{ThmIV}; see Remark~\ref{RmkKWV} for the extension to Kerr spacetimes, and Remark~\ref{RmkPVConormal} regarding relaxed regularity requirements. The existence of a leading order term, and also its explicit form, can in principle also be obtained using the methods of \cite{GuillarmouHassellSikoraResIII} upon supplementing the reference with high energy resolvent estimates. The asymptotic behavior of solutions of~\eqref{EqIVEqn} for \emph{compactly supported} $V$ is drastically different (resonance expansions, exponential decay); for a detailed discussion, we refer to \cite[Chapter~3]{DyatlovZworskiBook} and references therein.

\subsection{Method of proof; outlook}
\label{SsIP}

We work almost entirely on the spectral side and solve forward problems for forced waves,
\begin{equation}
\label{EqIPForc}
  \Box_g \phi = f \in \CIc(\R_{t_*}\times X^\circ),
\end{equation}
by means of the Fourier transform, given by $\hat\phi(\sigma,x)=\int e^{i\sigma t_*}\phi(t_*,x)\,d t_*$. Thus,
\begin{equation}
\label{EqIPFT}
  \phi(t_*,x) = \frac{1}{2\pi} \int_\R e^{-i\sigma t_*}\wh{\Box_g}(\sigma)^{-1}\hat f(\sigma,x)\,d\sigma,
\end{equation}
where $\wh{\Box_g}(\sigma)$ is defined in terms of $\Box_g$ by replacing all $\pa_{t_*}$ derivatives by multiplication by $-i\sigma$. The integral is well-defined and produces the forward solution of~\eqref{EqIPForc}; we refer the reader to the discussion in \cite[\S1.1]{HaefnerHintzVasyKerr} for details, and only briefly recall the main aspects here. Roughly speaking, if one integrates over the contour $\Im\sigma=C\gg 1$, one does obtain the forward solution by the Paley--Wiener theorem. One can then shift the contour down to the real axis using the mode stability assumption and the absence of zero energy resonances; high energy estimates (polynomial bounds on $\wh{\Box_g}(\sigma)^{-1}$ acting on suitable Sobolev spaces) justify the contour shifting. As mentioned before, mode stability is known on subextremal Kerr spacetimes \cite{WhitingKerrModeStability,ShlapentokhRothmanModeStability}; high energy estimates on the other hand are known to hold using semiclassical microlocal techniques, combining radial point estimates at the event horizons \cite{VasyMicroKerrdS} (see also \cite{ZworskiRevisitVasy} and \cite[Appendix~E]{DyatlovZworskiBook} for streamlined presentations), estimates at the normally hyperbolic trapped set \cite{WunschZworskiNormHypResolvent,DyatlovSpectralGaps,DyatlovWaveAsymptotics}, and radial point (microlocal Mourre) estimates at spatial infinity \cite{MelroseEuclideanSpectralTheory,VasyLAPLag}. See\footnote{The assumption of very small angular momenta, which is used in the precise low energy resolvent analysis of the reference, is \emph{not} used in the proof of high energy estimates.} \cite[Theorem~4.3]{HaefnerHintzVasyKerr}.

As in~\cite{HaefnerHintzVasyKerr}, the main task is thus to control the regularity of $\wh{\Box_g}(\sigma)^{-1}$ at $\sigma=0$; higher regularity means faster temporal decay of $\phi$. A key aspect of our analysis is that we use the Fourier transform in the coordinate $t_*$ (whose level sets are transversal to null infinity), rather than the `usual' time coordinate $t$ (which is not as well-suited to scattering theoretic considerations in the context of wave equations) as for example in \cite{BonyHaefnerResolvent,VasyWunschMorawetz,VasyLowEnergy}. The advantage of this point of view was pointed out by Vasy \cite{VasyLAPLag,VasyLowEnergyLag}. Namely, the limiting resolvent $\wh{\Box_g}(\sigma)^{-1}$ for nonzero real $\sigma$ produces outgoing solutions; working on the Minkowski spacetime and setting $t_*=t-r$ for concreteness, this corresponds to solutions with time dependence $e^{-i\sigma t}$ and leading order spatial dependence $r^{-1}e^{i\sigma r}$, thus an overall $e^{-i\sigma t_*}r^{-1}$; therefore, the `outgoing' condition for the output of $\wh{\Box_g}(\sigma)^{-1}$ in~\eqref{EqIPFT} means \emph{nonoscillatory, $\sigma$-independent} $r^{-1}$ behavior at infinity. In fact, the output is \emph{conormal}, i.e.\ has $r^{-1}$ decay upon repeated application of $r\pa_r$ and rotation vector fields.

As shown in \cite{VasyLowEnergyLag}, one can then analyze the low energy resolvent $\wh{\Box_g}(\sigma)^{-1}$ uniformly as $\sigma\to 0$ on spaces $\cA^\alpha$ of conormal functions with suitable decay rates $\alpha$ at infinity, corresponding to $r^{-\alpha}$ decay. Besides needing to allow outgoing $r^{-1}$ asymptotics in the target space, the decay rate $\alpha$ of the target space need to be chosen to ensure the invertibility also of the zero energy operator. The model is the Euclidean Laplacian $\Delta_{\R^3}$, which is invertible with domain given by functions decaying faster than $r^0$ (avoiding the nullspace given by constants) and less than $r^{-1}$ (since $\Delta_{\R^3}^{-1}$ with Schwartz kernel $(4\pi|x-y|)^{-1}$ typically does produce $r^{-1}$ asymptotics). Overall, one expects invertibility
\begin{equation}
\label{EqIPBox0}
  \wh{\Box_g}(0)^{-1} \colon \cA^{2+\alpha} \to \cA^\alpha,\quad \alpha\in(0,1),
\end{equation}
and thus uniform bounds for $\wh{\Box_g}(\sigma)^{-1}\colon\cA^{2+\alpha}\to\cA^\alpha$ for small $\sigma$.

We can now illustrate the basic mechanism underlying the proof of our main theorems. On the Schwarzschild spacetime, $\wh{\Box_g}(\sigma)$ is, to the relevant precision, equal to
\[
  L(\sigma) = -2 i\sigma r^{-1}(r\pa_r+1) + \bigl(\Delta_{\R^3} + 2\bhm r^{-3}(r\pa_r)^2\bigr)
\]
for large $r$. In terms of weights, the first term typically only gains one order of decay (mapping $r^\alpha$ to $r^{\alpha-1}$), the second term (which is $L(0)$) two. Let us now formally expand the resolvent near zero frequency by writing
\begin{align*}
  L(\sigma)^{-1}f &= L(0)^{-1}f + (L(\sigma)^{-1}-L(0)^{-1})f \\
    &= u_0 + \sigma L(\sigma)^{-1}f_1,
\end{align*}
where we set
\[
  u_0 = L(0)^{-1}f,\quad
  f_1 = -\sigma^{-1}\bigl(L(\sigma)-L(0)\bigr)u_0.
\]
The first term, $u_0$, is $\sigma$-independent. In the second term, we gain a factor of $\sigma$ (thus suggesting this is a more regular term); however, $L(0)^{-1}$ loses two orders of decay, while $\sigma^{-1}(L(\sigma)-L(0))$ only gains back one order, thus $f_1$ has one order of decay less than $f$. That is, \emph{the $\sigma$-gain comes at the cost of losing one order of decay} in the argument of $L(\sigma)^{-1}$.

One would like to iterate
\begin{equation}
\label{EqIPfIt}
  u_k = L(0)^{-1}f_k,\quad
  f_{k+1}=-\sigma^{-1}\bigl(L(\sigma)-L(0)\bigr)u_k,
\end{equation}
as often as possible while remaining in the invertible range~\eqref{EqIPBox0}; the resolvent expansion is then $u_0+\sigma u_1+\cdots+\sigma^k u_k+\cdots$. However, one cannot continue the iteration once $f_k$ only has $r^{-2}$ decay or less: the correction term
\begin{equation}
\label{EqIPfk}
  \sigma^k L(\sigma)^{-1}f_k,\qquad f_k\gtrsim r^{-2},
\end{equation}
in the expansion is then typically no longer uniformly bounded as $\sigma\to 0$. In fact, we show that if $f_k$ has borderline $r^{-2}$ decay, then $L(\sigma)^{-1}r^{-2}$ has a logarithmic singularity at $\sigma=0$ with explicit coefficient, hence~\eqref{EqIPfk} is $\sigma^k\log(\sigma+i 0)$. Upon taking the inverse Fourier transform in~\eqref{EqIPFT}, this singular term gives rise to a $t_*^{-k-1}$ leading order term\footnote{The absence of $\log t_*$ factor here is a consequence of the calculations~\eqref{EqPWFT}--\eqref{EqPWi0}.} of the solution of the wave equation $\Box_g\phi=f$. On the $(3+1)$-dimensional stationary and asymptotically flat spacetimes under consideration here, we obtain a borderline term~\eqref{EqIPfk} for $k=2$, giving the desired $t_*^{-3}$ decay; see the beginning of~\S\ref{SsPR} for a brief sketch of why it is indeed $f_2$ that has borderline $r^{-2}$ decay, and how the mass term $\bhm$ enters.

The precise analysis of a borderline term~\eqref{EqIPfk} is accomplished by geometric microlocal means: one constructs an approximate solution of $L(\sigma)\tilde u=f_k$ on a resolution $X_{\rm res}^+$ of the total space $[0,1)_\sigma\times([0,1)_\rho\times\Sph^2)$, $\rho=r^{-1}$, obtained by blowing up $\sigma=\rho=0$ (thus separating the regimes $\sigma/\rho=\sigma r\gtrsim 1$ and $\sigma r\lesssim 1$); the resolved space $X_{\rm res}^+$ already played a prominent role in~\cite{VasyLowEnergyLag}. The model problem on the front face is the spectral family at frequency $1$ of a rescaled \emph{exact} Minkowski space and can be analyzed in detail; the desired approximate solution is shown to have a $\log(\sigma/\rho)$ singularity (see Lemma~\ref{LemmaAModel}). To find the true solution $\tilde u$, one merely needs to apply $L(\sigma)^{-1}$ to the remaining error which has a logarithmic singularity in $\sigma$ but better than $r^{-2}$ spatial decay. (Overall, the coefficient of the logarithmic term of $\tilde u$ is an element in $\ker L(0)$ of size $1$ for large $r$, thus equal to $1$ for wave equations, and equal to $u_{(0)}$ as in Theorem~\ref{ThmIV} in the presence of potentials; see Proposition~\ref{PropARhoSq}.) The leading order term in the full forward light cone asserted in~\eqref{EqI2Phi} drops out of this construction as well, via the inverse Fourier transform (in $\sigma/\rho$) of the solution of the model problem (see equation~\eqref{EqPWOscInt} in the proof of Theorem~\ref{ThmPWGlobal}).

The proof of the full Price law~\eqref{EqIPrice} in~\S\ref{SPF} proceeds along the same lines; the higher regularity is mainly due to the fact that the invertibility~\eqref{EqIPBox0} holds for the larger range $\alpha\in(-l,l+1)$ of weights when restricting to angular frequency $l\in\N$, which allows for more iterations~\eqref{EqIPfIt}.

\begin{rmk}
  The regularity and pointwise estimates for $\phi$ in Theorems~\ref{ThmI} and \ref{ThmI2} are consequences of the conormal regularity at $\sigma=0$ of the resolvent (i.e.\ regularity of $L(\sigma)^{-1}f$ with respect to repeated application of $\sigma\pa_\sigma$) with values in appropriate conormal spatial function spaces.
\end{rmk}

Potential future extensions and applications of the methods developed here include:
\begin{enumerate}
\item a new proof of Morgan's results \cite{MorganDecay} on decay for stationary spacetimes asymptotic to Minkowski space at an inverse polynomial rate using the same approach (albeit possibly requiring more derivatives on the metric and the initial data);
\item a proof of the polyhomogeneity of $\phi$ on a compactification of the spacetime mentioned in Remark~\ref{RmkIPhg}. We expect that this can be done using same iteration~\eqref{EqIPfIt} upon keeping track of polyhomogeneous expansions of all the $u_k$, $f_k$, and extending the analysis of borderline (or worse) terms~\eqref{EqIPfk} so as to keep track of the polyhomogeneous expansion of the solution of the model problem, as well as of the remaining error term. The main ingredient is the analysis of $L(0)^{-1}$ on inputs which are polyhomogeneous on the resolved space $X_{\rm res}^+$;
\item an analysis of the effect of the angular momentum parameter $\bha$ of Kerr spacetimes. Here, $\bha$ enters at one order lower (in terms of $r$-decay) than the mass parameter $\bhm$, and destroys spherical symmetry, thus leading to the coupling of spherical harmonics when computing further terms (i.e.\ beyond what we do in the present paper) in the resolvent expansion. It would be interesting to see how the presence of nonzero angular momentum affects the full asymptotic expansion, in particular, whether there are extra logarithmic terms which are not present for $\bha=0$;
\item sharp asymptotics for equations with zero energy resonances or bound states. This requires significantly more work, as the resolvent now has strong singularities at $\sigma=0$; see \cite{HaefnerHintzVasyKerr}. Examples include Maxwell's equation or the equations of linearized gravity on Kerr spacetimes.
\end{enumerate}

\subsection{Outline of the paper}

\begin{itemize}
\item In~\S\ref{SA}, we describe the geometry (\S\ref{SsAM}) and spectral theory (\S\ref{SsAS}) of the class of stationary and asymptotically flat spacetime under investigation. We give a detailed account of the regularity and mapping properties of the low energy resolvent (\S\S\ref{SssASLo}--\ref{SssASLoLarge}) as required for the precise analysis of the iteration~\eqref{EqIPfIt}.
\item In~\S\ref{SP}, we prove the main result giving the low energy resolvent expansion (\S\ref{SsPR}) and use it to prove Price's law with leading order term (\S\ref{SsPW}). The modifications required for stationary potentials and Theorem~\ref{ThmIV} are described in~\S\ref{SsPV}.
\item In~\S\ref{SK}, subextremal Kerr spacetimes are placed into our general framework.
\item Finally, in~\S\ref{SPF}, we prove the full Price law stated in parts~\eqref{ItIPrice}--\eqref{ItIStatic} of Theorem~\ref{ThmI}.
\end{itemize}

\subsection*{Acknowledgments}

This project arose out of an ongoing collaboration with Andr\'as Vasy, and I would like to thank him for many valuable insights. I am very grateful to Jared Wunsch for an in-depth discussion and many detailed comments and suggestions. Part of this research was conducted during the period I served as a Clay Research Fellow. This material is based upon work supported by the National Science Foundation under Grant No.\ DMS-1440140 while I was in residence at the Mathematical Sciences Research Institute in Berkeley, California, during the Fall 2019 semester.

\section{Asymptotically flat spacetimes}
\label{SA}

\subsection{Metrics and wave operators}
\label{SsAM}

The model for the large scale behavior of the spacetimes we have in mind here is the Schwarzschild spacetime: given the mass $\bhm>0$, it has the metric
\begin{equation}
\label{EqAMSchw}
  g_\bhm = -\Bigl(1-\frac{2\bhm}{r}\Bigr)d t^2 + \Bigl(1-\frac{2\bhm}{r}\Bigr)^{-1}d r^2 + r^2\slg,
\end{equation}
where $\slg=d\theta^2+\sin^2\theta\,d\varphi^2$ is the standard metric on $\Sph^2$. Denote by $r_*=r+2\bhm\log(r-2\bhm)$ the tortoise coordinate, and put $t_*=t-r_*$; then
\[
  g_\bhm = -\Bigl(1-\frac{2\bhm}{r}\Bigr)d t_*^2 - 2 d t_* d r + r^2\slg,\quad
  g_\bhm^{-1} = -2\pa_{t_*}\pa_r + \Bigl(1-\frac{2\bhm}{r}\Bigr)\pa_r^2 + r^{-2}\slg^{-1}.
\]
The spacetimes we consider here are `short range' perturbations of this. To capture the asymptotics in a compact fashion, we define:

\begin{definition}
\label{DefACompact}
  The compactified spatial manifold $X$ is the radial compactification
  \[
    X := \ol{\R^3}
  \]
  of $\R^3$, defined as $(\R^3\sqcup([0,\infty)_\rho\times\Sph^2))/\sim$ where $\sim$ identifies points $r\omega\in\R^3$ for $r>0$, $\omega\in\Sph^2$ with $(\rho,\omega)$, $\rho=r^{-1}$.
\end{definition}

Thus, smooth functions on $X$ are precisely those smooth functions on $\R^3$ which in $r>1$ are smooth functions of $r^{-1}$ and the spherical variables. Near $\pa X=\rho^{-1}(0)$, we shall work in the collar neighborhood $[0,\eps)_\rho\times\Sph^2$.

\begin{definition}
\label{DefATsc}
  The \emph{scattering tangent bundle}
  \[
    \Tsc X \to X
  \]
  is the unique vector bundle for which the space of smooth sections consists of all smooth vector fields $V$ on $X^\circ$ which for $r>1$ are of the form $V=a\pa_r+\sum_{j=1}^3 b_j\rho\Omega_j$, where $a,b_j\in\CI(X)$, and $\Omega_1,\Omega_2,\Omega_3\in\cV(\Sph^2)$ are the rotation vector fields.\footnote{The point is that for each $p\in\Sph^2$, the tangent space $T_p\Sph^2$ is spanned by $\{\Omega_j(p)\colon j=1,2,3\}$.} In local coordinates $(\theta,\varphi)$ on $\Sph^2$, this means $V=a\pa_r+r^{-1}\tilde b_1\pa_\theta+r^{-1}\tilde b_2\pa_\varphi$ with $\tilde b_j\in\CI(X)$.
\end{definition}

One can check that the coordinate vector fields $\pa_{x^1},\pa_{x^2},\pa_{x^3}$ on $\R^3$ form a basis of $\Tsc X$ \emph{down to} $\pa X$. For example, the restriction of $g_\bhm^{-1}$ on $S^2 T^*X$, $(1-2\bhm\rho)\pa_r^2+r^{-2}(\pa_\theta^2+\sin^{-2}\theta\,\pa_\varphi^2)$, lies in $\CI(X;S^2\,\Tsc X)$ upon cutting it off to a neighborhood of $\rho=0$.

\begin{definition}
\label{DefAF}
  We call a smooth Lorentzian\footnote{Our signature convention is ${-}{+}{+}{+}$.} metric $g$ on $M^\circ=\R_{t_*}\times X^\circ$ \emph{stationary and asymptotically flat (with mass $\bhm\in\R$)} if $\pa_{t_*}$ is a Killing vector field, and if moreover
  \begin{enumerate}
  \item\label{ItAFdt} $d t_*$ is everywhere future timelike, i.e.\ $g^{0 0}<0$;
  \item\label{ItAFAsy} the coefficients of the dual metric
    \[
      g^{-1} = g^{0 0}\pa_{t_*}^2 + 2\pa_{t_*}\otimes_s g^{0 X} + g^{X X}
    \]
    satisfy
    \begin{align*}
      g^{0 0}&\in\rho^2\CI(X), \\
      g^{0 X}&\in-\pa_r+\rho^2\CI(X;\Tsc X), \\
      g^{X X}&\in (1-2\bhm\rho)\pa_r^2 + r^{-2}\slg^{-1} + \rho^2\CI(X;S^2\,\Tsc X).
    \end{align*}
  \end{enumerate}
\end{definition}

\begin{rmk}
  The function $t_*$ defined previously on the Schwarzschild spacetime does not satisfy~\eqref{ItAFdt}, but a small modification does; see~\S\ref{SK}. Even then, this definition \emph{excludes} Schwarzschild and Kerr metrics due to the existence of horizons. Since the low energy behavior of the resolvent is only sensitive to large end of the spacetime however, the adaptations to deal with Kerr are small and will be discussed in~\S\ref{SK}.
\end{rmk}

\begin{rmk}
\label{RmkAConormal}
  It suffices to assume that $g^{0 0}$ and the $\rho^2\CI$ error terms of $g^{0 X},g^{X X}$ are merely conormal, i.e.\ of class $\cA^2(X)$ in the notation of Definition~\ref{DefASSpaces}, in order for all arguments in this paper to apply unchanged. One can further relax their decay to $\cA^{1+\beta}(X)$ for $\beta>0$, though this \emph{does} affect the spatial decay rate and the $\sigma$-regularity of various terms in the resolvent expansion.
\end{rmk}

The wave operator $\Box_g=-|g|^{-1/2}\pa_\mu(|g|^{1/2}g^{\mu\nu}\pa_\nu) \in \Diff^2(M^\circ)$ of such a metric $g$ is invariant under time translations. We compute its form in the following terms:

\begin{definition}
\label{DefADiffb}
  The space $\Vb(X)$ of \emph{b-vector fields} on $X$ consists of all smooth vector fields $V$ on $X$ which are tangent to $\pa X$. For $r>1$, this means $V=a\rho\pa_\rho+\sum_{j=1}^3 b_j\Omega_j$ with $a,b_j\in\CI(X)$. For $m\in\N$, the space $\Diffb^m(X)$ of $m$-th order b-differential operators consists of all finite sums of up to $m$-fold products of b-vector fields. Finally, $\rho^\ell\Diffb^m(X)=\{\rho^\ell A\colon A\in\Diffb^m(X)\}$.
\end{definition}

\begin{lemma}
\label{LemmaABox}
  The wave operator $\Box_g$ of an admissible metric is given by
  \[
    \Box_g = -2\rho\pa_{t_*}Q + \wh{\Box_g}(0) - g^{0 0}\pa_{t_*}^2,
  \]
  where $Q\in\Diffb^1(X)$ and $\wh{\Box_g}(0)\in\rho^2\Diffb^2(X)$; near $\pa X$, they are of the form
  \begin{alignat}{2}
    Q &= Q_0 + \tilde Q, &\quad \tilde Q\in\rho^2\Diffb^1(X), \nonumber\\
  \label{EqABoxPieces0}
    \rho^{-2}\wh{\Box_g}(0) &= L_0 + \rho L_1 + \tilde L, &\quad \tilde L\in\rho^2\Diffb^2(X),
  \end{alignat}
  where the dilation-invariant (in $\rho$) operators $Q_0,L_0,L_1$ are given by
  \begin{equation}
  \label{EqABoxPieces}
  \begin{split}
    Q_0 &= \rho\pa_\rho-1, \\
    L_0 &= -(\rho\pa_\rho)^2+\rho\pa_\rho+\slDelta, \\
    L_1 &= 2\bhm(\rho\pa_\rho)^2,
  \end{split}
  \end{equation}
  where $\slDelta=-(\sin\theta)^{-1}\pa_\theta\sin\theta\,\pa_\theta-(\sin\theta)^{-2}\pa_\varphi^2$ is the (nonnegative) spherical Laplacian.
\end{lemma}
\begin{proof}
  The coefficients of second order derivatives are of course equal to (minus) the coefficients of the dual metric function; noting that $\pa_r=-\rho^2\pa_\rho$, this verifies the first order term of $Q_0$ and the second order terms of $L_0,L_1$. To compute the lower order terms, note that in polar coordinates $(\theta,\varphi)$ on $\Sph^2$ and up to an overall sign, we have
  \[
    |g|^{\pm 1/2} \in \rho^{\mp 2}(\sin\theta)^{\pm 1}(1+\rho^2\CI).
  \]
  Thus, up to $\pa_{t_*}\circ\rho^3\Diffb^1$ error terms (captured by $\rho\tilde Q$), the $t_*$-$X$-cross terms are given by $-\pa_{t_*}(-\pa_r) + \rho^2\cdot\rho^2\pa_\rho\bigl(\rho^{-2}(-\pa_{t_*})\bigr)$, which upon using $[\rho\pa_\rho,\rho^{-2}]=-2\rho^{-2}$ gives~\eqref{EqABoxPieces}.

  For the zero energy operator, the $\pa_r^2$ term of the dual metric gives, modulo $\rho^2\Diffb^2$ and using that $\rho^2\pa_\rho^2=(\rho\pa_\rho)^2-\rho\pa_\rho$,
  \[
    -\rho^2\rho^2\pa_\rho\bigl[\rho^{-2}(1-2\bhm\rho)\rho^2\pa_\rho\bigr] = \rho^2\bigl[-(\rho\pa_\rho)^2+\rho\pa_\rho + 2\bhm\rho(\rho\pa_\rho)^2 \bigr]
  \]
  which gives the terms in $L_j$ involving $\pa_\rho$. The $\pa_\theta^2$ term gives $-\rho^2(\sin\theta)^{-1}\pa_\theta\bigl(\sin\theta\,\pa_\theta)$; and the $\pa_{\varphi_*}^2$ coefficient finally gives ${-}(\sin\theta)^{-2}\rho^2\pa_{\varphi_*}^2$. The $\rho^2\CI$ error terms in $g^{X X}$ contribute to the $\rho^2\tilde L\in\rho^2\cdot\rho^2\Diffb^2$ error terms of $\wh{\Box_g}(0)$.
\end{proof}

\subsection{Spectral theory}
\label{SsAS}

We fix a stationary and asymptotically flat metric $g$ with mass $\bhm$. We denote the spectral family of $\Box_g$ by
\begin{equation}
\label{EqASBoxFam}
  \wh\Box(\sigma) := e^{i\sigma t_*}\Box e^{-i\sigma t_*} = 2 i\sigma\rho Q + \wh\Box(0) + \sigma^2 g^{0 0} \in \rho\Diffb^2(X),\qquad
  \Box\equiv\Box_g.
\end{equation}
We equip $X$ with the volume density
\begin{equation}
\label{EqASDensity}
  |d g_X|
\end{equation}
defined via $|d g|=|d t_*||d g_X|$, where $|d g|$ is the volume density of $g$. (Thus, in local coordinates on $X^\circ$, $|d g_X|=|\det g(t_*,x)|^{1/2}|d x|$, with the determinant independent of $t_*$.) We write $L^2(X):=L^2(X;|d g_X|)$; formal $L^2$ adjoints of differential operators on $X$ shall always be with respect to this $L^2$ space.

\begin{definition}
\label{DefASSpaces}
  On $X$ as in Definition~\ref{DefACompact}, we define the following function spaces.
  \begin{enumerate}
  \item For $s\in\N_0$ and $\ell\in\R$, we define the \emph{weighted b-Sobolev space} $\Hb^{s,\ell}(X)=\rho^\ell\Hb^s(X)$ for $\ell=s=0$ by $\Hb^0(X)=L^2(X)$, while $\Hb^s(X)$ for $s\in\N$ consists of all $u\in L^2(X)$ such that $A u\in L^2(X)$ for all $A\in\Diffb^s(X)$. The space $\Hb^{s,\ell}(X)$ is a Hilbert space with norm
  \[
    \|u\|_{\Hb^{s,\ell}(X)}^2 := \sum_{j=0}^s \sum_k \|A_{j k}u\|_{L^2(X)}^2,
  \]
  where the inner sum is over a finite set of operators $A_{j k}$ which span $\Diffb^j(X)$ over $\CI(X)$.\footnote{In the concrete setting at hand, we can take the operators $A_{j k}$ to be all up to $j$-fold compositions of the b-vector fields $\pa_{x^i}$ and $x^i\pa_{x^j}$ for $i,j=1,2,3$, since these vector fields span $\Vb(X)$ over $\CI(X)$.} The spaces $\Hb^s(X)$ for $s\in\R$ are defined by duality and interpolation.
  \item For $s,\ell\in\R$ and $h>0$, the \emph{semiclassical Sobolev space} $\Hbh^{s,\ell}(X)=\rho^\ell\Hbh^s(X)$ is equal to $\Hb^{s,\ell}(X)$ as a vector space, but with norm given by
  \[
    \|u\|_{\Hbh^{s,\ell}(X)}^2 := \sum_{j=0}^s \sum_k \|h^j A_{j k}u\|_{L^2(X)}^2.
  \]
  That is, each b-derivative comes with an extra factor of $h$.
  \item For $\alpha\in\R$, we define the \emph{conormal space} $\cA^\alpha(X)=\rho^\alpha\cA^0(X)$ to consist of all $u\in\rho^\alpha L^\infty(X)$ (i.e.\ $\rho^{-\alpha}u\in L^\infty(X)$) so that $A u\in\rho^\alpha L^\infty(X)$ for all $A\in\Diffb(X)$ (b-differential operators of any order).
  \end{enumerate}
\end{definition}

Sobolev embedding implies the inclusions
\begin{equation}
\label{EqASobolevEmb}
  \Hb^{\infty,\ell}(X) \subset \cA^{\ell+3/2}(X) \subset \Hb^{\infty,\ell-}(X);
\end{equation}
the `$+\tfrac32$' is due to $r^2\,d r|d\slg|=(r^{3/2})^2\frac{d r}{r}|d\slg|$ being a weighted b-density. Here, we define
\[
  \Hb^{s,\ell-}(X) := \bigcup_{\eps>0} \Hb^{s,\ell-\eps}(X),\quad
  \cA^{\alpha-}(X) := \bigcup_{\eps>0} \cA^{\alpha-\eps}(X).
\]
From~\eqref{EqASBoxFam}, we see that $\wh\Box(\sigma)\colon\cA^{1+\alpha}(X)\to\cA^\alpha(X)$ for all $\sigma\in\C$, with the zero energy operator being special in that $\wh\Box(0)\colon\cA^{2+\alpha}(X)\to\cA^\alpha(X)$.

\begin{definition}
\label{DefASAssm}
  The metric $g$ is \emph{spectrally admissible} if the following conditions are satisfied:
  \begin{enumerate}
  \item\label{ItASAssmBound} (Absence of bound states.) The nullspace of $\wh\Box(0)$ on $\cA^1(X)$ is trivial.
  \item\label{ItASAssmMS} (Mode stability.) The nullspace of $\wh\Box(\sigma)$ on $\cA^1(X)$ is trivial for all $\sigma\in\C$, $\Im\sigma\geq 0$.
  \item\label{ItASAssmHi} (High energy estimates.) There exists $\delta\in\R$ such that for $s\in\R$, $\ell<-\half$, $s+\ell>-\half$, there exists $C>0$ such that the estimate
  \begin{equation}
  \label{EqASAssmHi}
    \|u\|_{H_{\bop,|\sigma|^{-1}}^{s,\ell}(X)} \leq C|\sigma|^{-1+\delta}\|\wh\Box(\sigma)u\|_{H_{\bop,|\sigma|^{-1}}^{s,\ell+1}(X)},\qquad \Im\sigma\in[0,1),\ |\Re\sigma|\geq C,
  \end{equation}
  holds for all $u$ for which the norms on both sides are finite.
  \end{enumerate}
\end{definition}

Typically, the estimate~\eqref{EqASAssmHi} follows from assumptions on the dynamics of the null-geodesic flow on $(M^\circ,g)$: if there is no trapping, one can take $\delta=0$; if there is normally hyperbolic trapping, one needs to take $\delta>0$ though it can be arbitrarily small.

\subsubsection{Resolvent regularity at nonzero frequencies}

We briefly discuss the behavior of $\wh\Box(\sigma)^{-1}$ for $\sigma$ away from $0$. For any fixed $0<R_0<R_1$, and $s,\ell$ with $\ell<-\half$, $s+\ell>-\half$, assumption~\eqref{ItASAssmMS} implies the quantitative estimate
\begin{equation}
\label{EqASAssmMed}
  \|u\|_{\Hb^{s,\ell}(X)} \leq C\|\wh\Box(\sigma)u\|_{\Hb^{s,\ell+1}(X)},\quad
  \Im\sigma\geq 0,\ R_0<|\sigma|<R_1
\end{equation}
for a constant $C$ depending on $R_0,R_1,s,\ell$, see \cite[Theorem~1.1]{VasyLAPLag}. We then record:

\begin{lemma}
\label{LemmaASHiReg}
  Let $m\in\N_0$, then $\pa_\sigma^m\wh\Box(\sigma)^{-1}\colon\Hb^{s,\ell+1}(X)\to\Hb^{s-m,\ell}$ is bounded for $\sigma,s,\ell$ as in~\eqref{EqASAssmMed}. Similarly, the operator
  \[
    \pa_\sigma^m\wh\Box(\sigma)^{-1}\colon H_{\bop,|\sigma|^{-1}}^{s,\ell+1}(X)\to |\sigma|^{-1+(m+1)\delta} H_{\bop,|\sigma|^{-1}}^{s-m,\ell}(X)
  \]
  is uniformly bounded for $s,\ell,\sigma$ as in~\eqref{EqASAssmHi}.\footnote{By this, we mean that $\|\pa_\sigma^m\wh\Box(\sigma)^{-1}f\|_{H_{\bop,|\sigma|^{-1}}^{s-m,\ell}(X)}\leq C|\sigma|^{-1+(m+1)\delta}\|f\|_{H_{\bop,|\sigma|^{-1}}^{s,\ell+1}(X)}$.}
\end{lemma}
\begin{proof}
  The argument is identical to \cite[Proposition~12.10]{HaefnerHintzVasyKerr}. In brief, using $\pa_\sigma\wh\Box(\sigma)\in\rho\Diffb^1+\sigma\rho^2\CI$, one sees that $\pa_\sigma\wh\Box(\sigma)^{-1}=-\wh\Box(\sigma)^{-1}\circ\pa_\sigma\wh\Box(\sigma)\circ\wh\Box(\sigma)^{-1}$ maps
  \[
    H_{\bop,|\sigma|^{-1}}^{s,\ell+1} \xra{\wh\Box(\sigma)^{-1}} |\sigma|^{-1+\delta}H_{\bop,|\sigma|^{-1}}^{s,\ell} \xra{\pa_\sigma\wh\Box(\sigma)} |\sigma|^\delta H_{\bop,|\sigma|^{-1}}^{s-1,\ell+1} \xra{\wh\Box(\sigma)^{-1}} |\sigma|^{-1+2\delta}H_{\bop,|\sigma|^{-1}}^{s-1,\ell}.
  \]
  An inductive argument proves the lemma.
\end{proof}

\subsubsection{Mapping properties of the low energy resolvent}
\label{SssASLo}

Of primary interest for us is the low energy behavior of $\wh\Box(\sigma)^{-1}$. We recall from \cite[Theorem~1.1]{VasyLowEnergyLag}:

\begin{thm}
\label{ThmALo}
  Under assumption~\eqref{ItASAssmBound} of Definition~\usref{DefASAssm}, and for $s,\ell,\nu\in\R$ with $\ell<-\half$, $s+\ell>-\half$, $\ell-\nu\in(-\tfrac32,-\half)$,\footnote{The conditions arise from \begin{enumerate*} \item allowing the outgoing behavior of $u$, which means $r^{-1}e^{i\sigma (r-t)}$ behavior on spacetime for $r\gg 1$, and thus $r^{-1}$ on the spectral side, with $r^{-1}\in\Hb^{\infty,\ell}$ precisely for $\ell<-\half$; \item enforcing the absence of ingoing spherical waves $r^{-1}e^{i\sigma(-r-t)}$ on spacetime, thus $r^{-1}e^{-2 i\sigma r}$ on the spectral side, which is accomplished by requiring a decay order $s+\ell>-\half$ for nonzero oscillations at $r=\infty$; \item working in a range of weights on which $\wh\Box(0)$ is invertible, which gives the range of weights $(-\tfrac32,-\half)$ relative to $L^2(X)$.\end{enumerate*}} the bound
  \[
    \|(\rho+|\sigma|)^\nu u\|_{\Hb^{s,\ell}(X)} \leq C\|(\rho+|\sigma|)^{\nu-1}\wh\Box(\sigma)u\|_{\Hb^{s,\ell+1}(X)}
  \]
  holds for $\Im\sigma\geq 0$ with $|\sigma|\leq\sigma_0\ll 1$.
\end{thm}

As a consequence of the Sobolev embedding~\eqref{EqASobolevEmb}, we have
\begin{equation}
\label{EqAZeroMapping}
  \wh\Box(0)^{-1} \colon \cA^{2+\alpha}(X) \to \cA^{\alpha-}(X),\quad \alpha\in(0,1).
\end{equation}
In order to capture the output of the resolvent $\wh\Box(\sigma)^{-1}$ precisely near $\rho=\sigma=0$, we work on a resolved space:

\begin{definition}
\label{DefAResolved}
  The \emph{resolved space} (for positive frequencies) $X^+_{\rm res}$ is the blow-up
  \[
    X^+_{\rm res} := \bigl[ X \times [0,1)_\sigma; \pa X\times\{0\} \bigr].
  \]
  Denote the blow-down map by $\beta\colon X^+_{\rm res}\to [0,1)\times X$. The boundary hypersurfaces are denoted as follows:
  \begin{itemize}
  \item $\tface$: the front face;
  \item $\bface$ (`b-face'): the lift of $[0,1)\times\pa X$, i.e.\ the closure of $\beta^{-1}((0,1)\times\pa X)$;
  \item $\zface$ (`zero face'): the lift of $\{0\}\times X$, i.e.\ the closure of $\beta^{-1}(\{0\}\times X^\circ)$.
  \end{itemize}
\end{definition}

See Figure~\ref{FigARes}. In $\rho<1$, the functions
\[
  \rho_\bface := \frac{\rho}{\rho+\sigma},\quad
  \rho_\tface := \rho+\sigma,\quad
  \rho_\zface := \frac{\sigma}{\rho+\sigma}
\]
are smooth defining functions of the respective boundary hypersurfaces. Away from $\zface$, it is more convenient to work with the local defining functions $\hat\rho=\rho/\sigma=(\sigma r)^{-1}$ and $\sigma$, and away from $\bface$ one can take $\rho=r^{-1}$ and $\hat r=\sigma/\rho=\sigma r$. Thus, $\tface$ captures the transition from the regime $\sigma r\lesssim 1$ to $\sigma r\gtrsim 1$. (This is related to \cite[\S\S4--5]{DonningerSchlagSofferPrice}.)

\begin{figure}[!ht]
\centering
\includegraphics{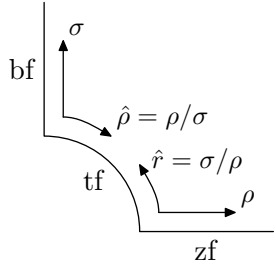}
\caption{The resolved space $X_{\rm res}^+$, together with useful local coordinates.}
\label{FigARes}
\end{figure}

On the manifold with corners $X^+_{\rm res}$, we consider conormal spaces
\[
  \cA^{\alpha,\beta,\gamma}(X^+_{\rm res})=\rho_\bface^\alpha\rho_\tface^\beta\rho_\zface^\gamma\cA^{0,0,0}(X^+_{\rm res}),
\]
with $\cA^{0,0,0}(X^+_{\rm res})$ consisting of all locally bounded functions that remain such upon application of any finite number of vector fields tangent to all boundary hypersurfaces of $X^+_{\rm res}$. We also need more precise function spaces capturing partial polyhomogeneous expansions. Recall that an \emph{index set} is a subset $\cE\subset\C\times\N_0$ such that the number of $(z,k)\in\cE$ with $\Re z<C$ is finite for any fixed $C\in\R$, and so that $(z,k)\in\cE$ implies $(z+1,k)\in\cE$ and, if $k\geq 1$, $(z,k-1)\in\cE$. We let $\inf\Re\cE$ denote the smallest value of $\Re z$ among all $(z,k)\in\cE$.

\begin{definition}
\fakephantomsection\label{DefAPhg}
  \begin{enumerate}
  \item Let $\cE$ be an index set, $\alpha_0:=\inf\Re\cE$, and $\alpha\in\R$; put $\beta=\min(\alpha_0,\alpha)$. Then the space $\cA^{(\cE,\alpha)}(X)\subset\cA^{\beta-}(X)$ consists of all $u$ which are smooth in $X^\circ$ and which near $\pa X$ have a partial expansion
  \[
    u - \sum_{\genfrac{}{}{0pt}{}{(z,k)\in\cE}{\Re z\leq\alpha}} u_{{z,k}}(\omega)\rho^z(\log\rho)^k \in \cA^\alpha(X)
  \]
  for some $u_{{z,k}}\in\CI(\pa X)$.
  \item Let $\cE_\bface,\cE_\tface,\cE_\zface$ denote three index sets, $\alpha_{0,\bullet}=\inf\Re\cE_\bullet$, and $\alpha_\bullet\in\R$ for $\bullet=\bface,\tface,\zface$. Put $\beta_\bullet=\min(\alpha_{0,\bullet},\alpha_\bullet)$. Then
  \[
    \cA^{(\cE_\bface,\alpha_\bface),(\cE_\tface,\alpha_\tface),(\cE_\zface,\alpha_\zface)}(X^+_{\rm res})
  \]
  consists of all $u\in\cA^{\beta_\bface-,\beta_\tface-,\beta_\zface-}(X^+_{\rm res})$ which have partial expansions with conormal remainders at all boundary hypersurfaces. That is, in a collar neighborhood $[0,\eps)_{\rho_\zface}\times\zface$ of the zero face $\zface\cong X$, there exist $u_{\zface,(z,k)}\in\cA^{(\cE_\tface,\alpha_\tface)}(\zface)$ such that
  \[
    u - \sum_{\genfrac{}{}{0pt}{}{(z,k)\in\cE_\zface}{\Re z\leq\alpha_\zface}} u_{\zface,(z,k)}(x) \rho_\zface^z(\log\rho_\zface)^k \in \cA^{\beta_\tface-,\alpha_\zface}([0,\eps)\times\zface),
  \]
  where the exponents on the right are, in this order, the weights at $[0,\eps)\times\pa\zface$ and $\{0\}\times\zface$. Likewise, $u$ has partial expansions at the remaining two boundary hypersurfaces $\tface,\bface$.
  \item Partially polyhomogeneous spaces such as $\cA^{\alpha_\bface,\alpha_\tface,(\cE_\zface,\alpha_\zface)}(X^+_{\rm res})$ have partial expansions only at the boundary hypersurfaces at which an index set is given.
  \end{enumerate}
\end{definition}

A typical index set is
\[
  (z_0,k_0) := \{ (z,k)\in\C\times\N_0 \colon k\leq k_0,\ z-z_0\in\N_0 \};
\]
for instance, $\cA^{(z_0,0)}(X)=\rho^{z_0}\CI(X)$ and $\cA^{(0,1)}(X)=\CI(X)+(\log\rho)\CI(X)$. The regularity of the low energy resolvent is then as follows; this is similar to \cite[Propositions~12.4 and 12.12]{HaefnerHintzVasyKerr}, though here we do not keep track of the number of derivatives used.

\begin{prop}
\label{PropALoSmall}
  Let $\alpha\in(0,1)$ and $f\in\cA^{2+\alpha}(X)$. Then\footnote{The order at $\bface$ can be improved to $1-$ by taking $\ell<-\half$ close to $-\half$ in the application of Theorem~\ref{ThmALo} in the proof below, though we do not need this precision here.}
  \begin{equation}
  \label{EqALoResRegInput}
    \wh\Box(\sigma)^{-1}f \in \cA^{\alpha-,\alpha-,((0,0),\alpha-)}(X^+_{\rm res}).
  \end{equation}
  For $\sigma$-dependent inputs $f\in\cA^\beta([0,1)_\sigma;\cA^{2+\alpha}(X))$ with $\beta\in\R$, we have
  \begin{equation}
  \label{EqALoResConInput}
    \wh\Box(\sigma)^{-1}f \in \cA^{\alpha-,\alpha+\beta-,\beta}(X^+_{\rm res}).
  \end{equation}
\end{prop}
\begin{proof}
  By Theorem~\ref{ThmALo} and using the Sobolev embedding~\eqref{EqASobolevEmb}, $u:=\wh\Box(\sigma)^{-1}f$ is bounded in $\sigma$ with values in $\cA^{\alpha-}(X)$; the conormality of $u$ at $\sigma=0$ is a consequence of $\sigma\pa_\sigma\wh\Box(\sigma)^{-1}=-\wh\Box(\sigma)^{-1}\circ\sigma\pa_\sigma\wh\Box(\sigma)\circ\wh\Box(\sigma)^{-1}$, which maps
  \begin{equation}
  \label{EqALoResConormal}
    \cA^{2+\alpha}(X) \xra{\wh\Box(\sigma)^{-1}} \cA^{\alpha-}(X) \xra{\sigma\pa_\sigma\wh\Box(\sigma)} |\sigma|\cA^{\alpha+1-}(X) \xra{\wh\Box(\sigma)^{-1}} \cA^{\alpha-}(X);
  \end{equation}
  in the final mapping step, we use Theorem~\ref{ThmALo} with $\ell=-3/2+\alpha-\delta$, $\nu=0$, with $0<\delta\leq\alpha$ arbitrary, and estimate $(\rho+|\sigma|)^{-1}\leq|\sigma|^{-1}$. Higher order derivatives along $\sigma\pa_\sigma$ are handled iteratively. Thus,
  \begin{equation}
  \label{EqALoResCon1}
    u\in\cA^0([0,1)_\sigma;\cA^{\alpha-}(X))\subset\cA^{\alpha-,\alpha-,0}(X^+_{\rm res}).
  \end{equation}

  To improve regularity at $\zface$, we apply Theorem~\ref{ThmALo} in the same fashion, but now estimate $(\rho+|\sigma|)^{-1}\leq|\sigma|^{-1+\alpha-\delta}\rho^{-\alpha+\delta}$ with $0<\delta<\alpha$ to see that $\sigma\pa_\sigma\wh\Box(\sigma)^{-1}$ maps
  \[
    \cA^{2+\alpha}(X) \xra{\wh\Box(\sigma)^{-1}} \cA^{\alpha-}(X) \xra{\sigma\pa_\sigma\wh\Box(\sigma)} |\sigma|\cA^{\alpha+1-}(X) \xra{\wh\Box(\sigma)^{-1}} |\sigma|^{\alpha-\delta}\cA^{\delta-}(X).
  \]
  Applying further $\sigma\pa_\sigma$ derivatives using~\eqref{EqALoResConormal} shows that
  \[
    \sigma\pa_\sigma u\in\cA^{\alpha-\delta}([0,1)_\sigma;\cA^{\delta-}(X))\subset\cA^{\delta-,\alpha-,\alpha-\delta}(X^+_{\rm res}).
  \]

  Inverting the regular singular ODE $\sigma\pa_\sigma u\in\rho_\zface^{\alpha-\delta}\cA^{\alpha-}(\zface)$ near $\zface$ gives the desired leading order term at $\rho_\zface=0$, i.e.\ $u\in\cA^{\alpha-,(0,0)+\alpha-}([0,\eps)_{\rho_\zface}\times\zface)$ near $\zface$. Combining this, using a partition of unity, with~\eqref{EqALoResCon1} proves~\eqref{EqALoResRegInput}.

  The claim~\eqref{EqALoResConInput} follows directly from $\wh\Box(\sigma)^{-1}f\in\cA^\beta([0,1);\cA^{\alpha-}(X))$; the latter is proved by a simple adaptation of~\eqref{EqALoResConormal}.
\end{proof}

\begin{rmk}
\label{RmkALoRes}
  Upon taking the inverse Fourier transform in $\sigma$, this already suffices to show $t_*^{-1-\alpha}$ decay of forward solutions of $\Box\phi\in\CIc(\R_{t_*};\cA^{2+\alpha}(X))$.
\end{rmk}

For inputs living on the resolved space, we record:
\begin{lemma}
\label{LemmaALoResolvedInput}
  Let $\alpha\in(0,1)$ and $f\in\cA^{2+\alpha,2+\alpha,\alpha}(X^+_{\rm res})$. Then
  \begin{equation}
  \label{EqALoResInput}
    \wh\Box(\sigma)^{-1}f \in \cA^{\alpha-,\alpha-,\alpha-}(X^+_{\rm res}).
  \end{equation}
  For the $\sigma$-independent operator $\wh\Box(0)^{-1}$, we have $\wh\Box(0)^{-1}f\in\cA^{\alpha-,\alpha-,\alpha-}(X^+_{\rm res})$.
\end{lemma}
\begin{proof}
  Since $f\in\cA^0([0,1)_\sigma;\cA^{2+\alpha}(X))\cap\cA^{\alpha-\delta}([0,1)_\sigma;\cA^{2+\delta}(X))$ for all $0<\delta<\alpha$, we have
  \begin{align*}
    \wh\Box(\sigma)^{-1}f &\in \cA^0([0,1);\cA^{\alpha-}(X)) \cap \cA^{\alpha-\delta}([0,1);\cA^{\delta-}(X)) \\
      &\subset \cA^{\alpha-,\alpha-,0}(X^+_{\rm res}) \cap \cA^{\delta-,\alpha-,\alpha-\delta}(X^+_{\rm res}),
  \end{align*}
  proving~\eqref{EqALoResInput}. These arguments apply verbatim also to $\wh\Box(0)^{-1}f$ in view of~\eqref{EqAZeroMapping}.
\end{proof}

\subsubsection{Action of the resolvent on large inputs}
\label{SssASLoLarge}

We first record a simple estimate for less decaying inputs which will be used to estimate error terms in resolvent expansions later on:
\begin{lemma}
\label{LemmaALoLargeInput}
  Let $\alpha\in(0,1)$. Suppose $f\in\cA^{2-\alpha}(X)$. Then
  \[
    \wh\Box(\sigma)^{-1}f \in \cA^{1-\alpha-,-\alpha-,-\alpha-}.
  \]
  The same conclusion holds if, more generally, $f\in\cA^0([0,1),\cA^{2-\alpha}(X))$.
\end{lemma}
\begin{proof}
  By Theorem~\ref{ThmALo} with $l=-\half-\alpha$ and $\nu=0$, which implies that $\wh\Box(\sigma)^{-1}\colon\cA^{2-\alpha}(X)\to\cA^{1-\alpha-}(X)$ is bounded by $|\sigma|^{-1}$, we have $\wh\Box(\sigma)^{-1}f\in\cA^{-1}([0,1);\cA^{1-\alpha}(X))$. (Conormal regularity in $\sigma$ is proved as usual.)
  
  On the other hand, applying Theorem~\ref{ThmALo} with $l=-\tfrac32-\delta$ and $\nu=-2\delta$ for small $0<\delta<1-\alpha$ allows us to estimate
  \begin{align*}
    \|u\|_{\Hb^{s,-3/2-\delta}} &\lesssim \|(\rho+|\sigma|)^{-2\delta}u\|_{\Hb^{s,-3/2-\delta}} \lesssim \|(\rho+|\sigma|)^{-1-2\delta}\wh\Box(\sigma)u\|_{\Hb^{s,-1/2-\delta}} \\
      &\lesssim |\sigma|^{-\alpha-\delta}\|\wh\Box(\sigma)u\|_{\Hb^{s,1/2-\alpha}},
  \end{align*}
  hence (upon increasing $\alpha$ by any small positive amount) $|\sigma|^{\alpha+\delta+}\wh\Box(\sigma)^{-1}f$ is bounded in $\cA^{-\delta}(X)$. We thus obtain $\wh\Box(\sigma)^{-1}f\in\cA^{-\alpha-}([0,1);\cA^{-0}(X))$, giving the improvement at $\zface$ and proving the lemma.
\end{proof}

\begin{rmk}
  Following the general strategy outlined in \S\ref{SsIP}, this lemma can also be proved more systematically by solving a model problem involving $\wt\Box(1)$ and applying the standard resolvent to the remaining error term which has better decay.
\end{rmk}

The key technical result for obtaining the precise nature of the first singular term of the resolvent concerns the regularity of $\wh\Box(\sigma)^{-1}$ acting on borderline $\rho^2\CI(X)$ input, see Proposition~\ref{PropARhoSq} below. To set this up, we first show:

\begin{lemma}
\label{LemmaAExtState}
  We have
  \[
    \cA^0(X)\cap\ker\wh\Box(0) = \C u_{(0)},\qquad
    \cA^0(X)\cap\ker\wh\Box(0)^* = \C u_{(0)}^*
  \]
  for states $u_{(0)},u_{(0)}^*\in\cA^{((0,0),1-)}(X)$ which are uniquely determined by their leading order behavior $u_{(0)},u_{(0)}^*\in 1+\cA^{1-}(X)$.
\end{lemma}

In the present setting we simply have
\[
  u_{(0)} = 1,\qquad
  u_{(0)}^* = 1.
\]
However, we keep the notation more general in order for our derivation of Price's law to apply unchanged to more general situations such as Kerr or wave equations with potential, see \S\S\ref{SsPV} and \ref{SK}. Correspondingly, the proof of this lemma will only use the structures which are present in these more general situations.

\begin{proof}[Proof of Lemma~\usref{LemmaAExtState}]
  First, we prove the existence of $u_{(0)}$. Let $\chi_\pa\in\CI(X)$ denote a cutoff to a neighborhood of $\pa X$. Since the normal operator of $\rho^{-2}\wh\Box(0)$, given as $L_0=-(\rho\pa_\rho)^2+\rho\pa_\rho+\slDelta$ by Lemma~\ref{LemmaABox}, annihilates constants, we have $e:=-\wh\Box(0)(\chi_\pa)\in\rho^3\CI(X)$, which is one order of improvement relative to the usual mapping property $\wh\Box(0)\colon\CI(X)\to\rho^2\CI(X)$ of $\wh\Box(0)$. But since $\wh\Box(0)\colon\cA^{1-}(X)\to\cA^{3-}(X)$ is surjective, there exists a unique $\tilde u\in\cA^{1-}(X)$ with $\wh\Box(0)\tilde u=e$. We can then put $u_{(0)}=\chi_\pa+\tilde u$.

  For any other extended zero energy state $\tilde u_{(0)}\in 1+\cA^{1-}(X)$, we have $u_{(0)}-\tilde u_{(0)}\in\cA^{1-}(X)\cap\ker\wh\Box(0)=\{0\}$ since $\wh\Box(0)$ is injective on $\cA^{1-}(X)$; this gives uniqueness.

  The arguments for $u_{(0)}^*$ are completely analogous.
\end{proof}

The analysis of $\wh\Box(\sigma)^{-1}f$, $f\in\rho^2\CI(X)$, proceeds by constructing an approximate solution of $\wh\Box(\sigma)u=f$ near $\rho=\sigma=0$ explicitly, and then correcting it to a true solution using $\wh\Box(\sigma)^{-1}$ acting on a function space with more decay. The relevant model problem already prominently featured in \cite[\S5]{VasyLowEnergyLag} in the context of the proof of Theorem~\ref{ThmALo}.

\begin{definition}
\label{DefAModel}
  Let $\hat\rho:=\rho/\sigma$. The \emph{model operator} at $\rho=\sigma=0$ is
  \[
    \wt\Box(1) = 2 i\hat\rho(\hat\rho\pa_{\hat\rho}-1)+\hat\rho^2\bigl(-(\hat\rho\pa_{\hat\rho})^2+\hat\rho\pa_{\hat\rho}+\slDelta\bigr) \in \Diff^2((0,\infty)_{\hat\rho}\times\Sph^2).
  \]
\end{definition}

Letting $\hat r=\hat\rho^{-1}=\sigma/\rho=\sigma r$, we recognize this as the spectral family of the wave operator on Minkowski space $\R_{\hat t_*}\times(0,\infty)_{\hat r}\times\Sph^2$ with metric $-d\hat t_*-2 d\hat t_*\,d\hat r+\hat r^2\slg$ at frequency $1$, though we regard the `origin' $\hat r=0$ as a (singular) conic point.

We regard $\wt\Box(1)$ as a differential operator on $\tface$; note that $\tface\setminus(\tface\cap\zface)=[0,\infty)_{\hat\rho}\times\pa X$, and $\hat\rho$ (which is smooth on $X^+_{\rm res}\setminus\zface$) is a defining function of $\bface$ away from $\zface$. Changing variables, we see that
\[
  \wt\Box(1) = -2 i\hat r^{-1}(\hat r\pa_{\hat r}+1) + \hat r^{-2}\bigl(-(\hat r\pa_{\hat r})^2-\hat r\pa_{\hat r}+\slDelta\bigr).
\]
Thus, $\wt\Box(1)\in\rho_\bface\rho_\zface^{-2}\Diffb^2(\tface)$. The importance of $\wt\Box(1)$ in relation to $\wh\Box(\sigma)$ stems from the following calculation:

\begin{lemma}
\label{LemmaAResolvedOp}
  The operator $\wh\Box(\sigma)$, as a second order differential operator on $X^+_{\rm res}$, is a b-differential operator of class $\wh\Box(\sigma)\in\rho_\bop\rho_\tface^2\Diffb^2(X^+_{\rm res})$. Its b-normal operators are:
  \begin{itemize}
  \item $2 i\sigma^2\hat\rho(\hat\rho\pa_{\hat\rho}-1)$ at $\bface$, i.e.\ $\wh\Box(\sigma)$ differs from this by an element of $\rho_\bop^2\rho_\tface^2\Diffb^2$;
  \item $\sigma^2\wt\Box(1)$ at $\tface$, i.e.\ $\wh\Box(\sigma)-\sigma^2\wt\Box(1)\in\rho_\bop\rho_\tface^3\Diffb^2$;
  \item $\wh\Box(0)$ at $\zface$, i.e.\ $\wh\Box(\sigma)-\wh\Box(0)\in\rho_\bop\rho_\tface^2\rho_\zface\Diffb^2$.
  \end{itemize}
  In fact, $\wh\Box(\sigma)-\sigma^2\wt\Box(1)\in\rho_\bop^2\rho_\tface^3\Diffb^2$.
\end{lemma}
\begin{proof}
  Working near $\bface\subset X^+_{\rm res}$ with coordinates $\sigma\geq 0$, $\hat\rho\in[0,\infty)$, and a factor of $\pa X$, consider the form~\eqref{EqASBoxFam} of $\wh\Box(\sigma)$ in the notation of Lemma~\ref{LemmaABox}: the terms $Q_0$ and $L_0$ give rise to $\wt\Box(1)$. On the other hand, elements of $\sigma^l\rho^k\Diffb^m(X)$ lift along the stretched projection $X^+_{\rm res}\to X$ to elements of $\rho_\bface^k\rho_\tface^{l+k}\rho_\zface^l\Diffb^m(X^+_{\rm res})$; hence $\tilde Q$ and $\rho L_1+\tilde L$ (as well as $g^{0 0}\sigma^2$) lift to b-differential operators on $X^+_{\rm res}$ which vanish cubically at $\tface$ and quadratically at $\bface$.

  Near $\zface\subset X^+_{\rm res}$ on the other hand, all terms with a factor of $\sigma$ vanish at $\zface$, hence the b-normal operator at $\zface$ is $\wh\Box(0)$ as claimed.
\end{proof}

The limiting absorption principle for $\wt\Box(1)$ is stated in terms of the function spaces
\[
  \Hb^{s,l,\nu}(\tface)=\rho_\bface^l\rho_\zface^\nu\Hb^s(\tface),\quad \cA^{\beta,\gamma}(\tface)=\rho_\bface^\beta\rho_\zface^\gamma\cA^{0,0}(\tface).
\]
The b-Sobolev space is defined as usual, with the convention that $\Hb^0(\tface)=L^2(\tface;\hat r^2\,|d\hat r\,d\slg|)$ is the Euclidean $L^2$ space. The following result is proved in \cite[Proposition~5.4]{VasyLowEnergyLag}:

\begin{thm}
\label{ThmALAP}
  For $s\in\R$, $l<-\half$, $s+l>-\half$, and $\nu\in(\half,\tfrac32)$, the operator
  \[
    \tilde\Box(1) \colon \bigl\{ u\in\Hb^{s,l,\nu}(\tface) \colon \tilde\Box(1)u\in\Hb^{s,l+1,\nu-2}(\tface) \bigr\} \to \Hb^{s,l+1,\nu-2}
  \]
  is invertible. In particular, $\tilde\Box(1) \colon \cA^{\beta,\gamma}(\tface) \to \cA^{\beta+1,\gamma-2}(\tface)$ is an isomorphism for $\beta<1$ and $\gamma\in(-1,0)$.
\end{thm}

We are now prepared to study the model problem $\sigma^2\wt\Box(1)\tilde u=\rho^2 f$:

\begin{lemma}
\label{LemmaAModel}
  Let $\tilde f\in\CI(\pa X)$. The unique solution $\tilde u\in\cA^{1-,0-}(\tface)$ of
  \[
    \wt\Box(1)\tilde u=\hat\rho^2\tilde f \in \cA^{2,-2}(\tface)
  \]
  lies in the space $\cA^{1-,((0,1)+1-)}(\tface)$; the leading term at $\zface$ is ${-}\frac{1}{4\pi}(\int_{\pa X}\tilde f\,|d\slg|)\log\hat r$.
\end{lemma}
\begin{proof}
  We only need to analyze $\tilde u$ at $\tface\cap\zface$, i.e.\ near $\hat r=0$. Let $\psi=\psi(\hat r)\in\CIc([0,\half))$ be identically $1$ near $\hat r=0$, and put $v=\psi\tilde u$, then, on $[0,1)_{\hat r}\times\pa X$,
  \[
    \wt\Box(1)v = \hat r^{-2}\psi\tilde f + [\wt\Box(1),\psi]\tilde u \in \hat r^{-2}\CI.
  \]
  The b-normal operator of $\wt\Box(1)$ is $\hat r^{-2}L_0$ in the notation of~\eqref{EqABoxPieces}, hence $v\in\cA^{0-}$ solves
  \begin{equation}
  \label{EqAModelZf}
    L_0 v = h,\quad h\in\psi\tilde f+\cA^{1-},\qquad L_0=-(\hat r\pa_{\hat r})^2-\hat r\pa_{\hat r}+\slDelta.
  \end{equation}
  We analyze this using a typical b-normal operator and contour shifting argument using the Mellin transform in $\hat r$, defined by
  \[
    \hat v(\xi) := \int_0^\infty \hat r^{-i\xi}v(\hat r)\frac{d\hat r}{\hat r}
  \]
  (dropping the dependence on the variables in $\pa X$ from the notation); note that $\wh{\hat r\pa_{\hat r}v}(\xi)=i\xi\hat v(\xi)$. Thus, equation~\eqref{EqAModelZf} becomes
  \[
    \wh{L_0}(\xi)\hat v(\xi) = \hat h(\xi),\qquad
    \wh{L_0}(\xi) := -(i\xi)^2-i\xi+\slDelta;
  \]
  at this point, we only know that $\hat v(\xi)$ is holomorphic in $\Im\xi>0$ and satisfies estimates
  \[
    \|\hat v(\xi)\|_{H^s(\pa X)} \leq C_{s N\eps}\la\xi\ra^{-N},\qquad
    \eps<\Im\xi<1,
  \]
  for all $s,N\in\R$ and $\eps>0$; the same estimate holds for $\hat h(\xi)$ but in the larger region $-1<\Im\xi<1$, $|\xi|>\eps$, except $\hat h(\xi)$ has a simple pole at $\xi=0$ with residue a constant multiple of $h(0)=\tilde f$.
  
  Expand $v$ and $h$ into spherical harmonics $Y_{l m}$, $l=0,1,2,\dots$, $|m|\leq l$, with $\slDelta Y_{l m}=l(l+1)$, and let $\cY_l:=\mathspan\{Y_{l m}\colon m=-l,\dots,l\}$. Restricted to $\cY_l$, the inverse $\wh{L_0}(\xi)^{-1}|_{\cY_l}=(-(i\xi)^2-i\xi+l(l+1))^{-1}$ is meromorphic with simple poles precisely at $i\xi=-l-1,l$. Thus, writing $h=h_0+h'$ with $h'$ orthogonal to $Y_0$ (i.e.\ with vanishing spherical average), then $\wh{L_0}(\xi)^{-1}\wh{h_0}(\xi)$, resp.\ $\wh{L_0}(\xi)^{-1}\wh{h'}(\xi)$, is meromorphic in $\Im\xi>-1$ with (at most) a double, resp.\ simple pole at $\xi=0$. In the inverse Mellin transform
  \[
    v(\hat r) = \frac{1}{2\pi}\int_{\Im\xi=\eps} \hat r^{i\xi} \wh{L_0}(\xi)^{-1}\hat h(\xi)\,d\xi,
  \]
  we can then shift the contour to $\Im\xi=-1+\eps$; the residue theorem gives the expansion
  \[
    v(\hat r) = c(\log\hat r) + \tilde v + \cA^{1-}
  \]
  for some $c\in Y_0=\C$ and $\tilde v\in\CI(\pa X)$. The value of $c$ can be determined from this, or directly by noting that $L_0(\log\hat r)=-1$, hence $c=-h_0(0)=-(4\pi)^{-1}\int_{\pa X}\tilde f\,|d\slg|$.
\end{proof}

These types of arguments are frequently formulated in the opposite order: one first explicitly solves away the leading term (here the spherically symmetric part of $\tilde f$) using the $\log\hat r$ term, and then solves away the remaining error term, acting on which $L_0^{-1}$ does not produce any logarithmic singularities anymore (i.e.\ poles on the Mellin transform side).

The precise behavior of $\wh\Box(\sigma)^{-1}$ on spherically symmetric $\rho^2\CI(X)$ inputs is then:\footnote{Inputs whose leading order term at $\pa X$ has vanishing average do not produce singularities at $\zface$. Since such inputs do not arise in our application, we do not state a result for $\rho^2\CI(X)$ inputs here.}
\begin{prop}
\label{PropARhoSq}
  Let $u_{(0)}\in\cA^0(X)$ with $\wh\Box(0)u_{(0)}=0$ be as in Lemma~\usref{LemmaAExtState} above. Let $\tilde u^{(2)}:=\wt\Box(1)^{-1}(\hat\rho^2)$, as computed by Lemma~\usref{LemmaAModel}. Then
  \[
    \wh\Box(\sigma)^{-1}\rho^2 \in \cA^{1-,((0,0),1-),((0,1),1-)}(X^+_{\rm res}).
  \]
  The leading order term at $\tface$ is equal to $\tilde u^{(2)}$, and the leading order term at $\zface$ is equal to $-(\log\tfrac{\sigma}{\rho})u_{(0)}$.
\end{prop}
\begin{proof}
  For $\chi_\pa\in\CI(X)$ identically $1$ near $\pa X$, we write
  \begin{equation}
  \label{EqARhoSqErr}
    \wh\Box(\sigma)^{-1}\rho^2 = \chi_\pa\tilde u^{(2)} + \wh\Box(\sigma)^{-1}\tilde e,\qquad
    \tilde e := \rho^2 - \wh\Box(\sigma)(\chi_\pa\tilde u^{(2)}).
  \end{equation}
  Note here that $\chi_\pa\tilde u^{(2)}\in\cA^{1-,(0,0),(0,1)+1-}(X^+_{\rm res})$. The key point here is that $\tilde e$ has an extra order of spatial decay; indeed, $\tilde e \in \cA^{3-,3,0-}(X^+_{\rm res})$, improving over $\rho^2\in\cA^{2,2,0}(X^+_{\rm res})$ at the expense of a singularity at $\zface$. More precisely, $\tilde e$ is polyhomogeneous on $X^+_{\rm res}$: using Lemma~\ref{LemmaAResolvedOp},
  \[
    \tilde e\in\cA^{3-,3-,(0,1)+1-}(X^+_{\rm res}),
  \]
  where the $\log\hat r$ leading term at $\zface$ is given by $-(\log\hat r)\wh\Box(0)(-\chi_\pa)=(\log\hat r)\wh\Box(0)(\chi_\pa)$. We then write
  \begin{equation}
  \label{EqARhoSqErrExp}
    \wh\Box(\sigma)^{-1}\tilde e = \tilde v - \sigma\wh\Box(\sigma)^{-1}\bigl[\sigma^{-1}(\wh\Box(\sigma)-\wh\Box(0))\bigr]\tilde v, \qquad
    \tilde v:=\wh\Box(0)^{-1}\tilde e.
  \end{equation}

  Now, since $(\sigma\pa_\sigma)^2\tilde e\in\cA^{3-,3-,1-}(X^+_{\rm res})$, Lemma~\ref{LemmaALoResolvedInput} implies
  \begin{equation}
  \label{EqARhoSqErrVder}
    (\sigma\pa_\sigma)^2\tilde v\in\cA^{1-,1-,1-}(X^+_{\rm res}).
  \end{equation}
  On the other hand, viewing $\tilde e\in\cA^{3-,3-,0-}(X^+_{\rm res})\subset\cA^{0-}([0,1);\cA^{3-}(X))$, Proposition~\ref{PropALoSmall} gives $\tilde v\in\cA^{0-}([0,1);\cA^{1-}(X))\subset\cA^{1-,1-,0-}(X^+_{\rm res})$. Integrating~\eqref{EqARhoSqErrVder} thus implies
  \[
    \tilde v \in \cA^{1-,1-,(0,1)+1-}(X^+_{\rm res}).
  \]
  Using Lemma~\ref{LemmaAResolvedOp}, the logarithmic term of $\tilde v$ at $\zface$ is $(\log\hat r)\wh\Box(0)^{-1}\bigl(\wh\Box(0)(\chi_\pa)\bigr)$, with $\wh\Box(0)^{-1}$ mapping into $\cA^{1-}(X)$; but the unique $\tilde v_{0,1}\in\cA^{1-}(X)$ satisfying the equation $\wh\Box(0)\tilde v_{0,1}=\wh\Box(0)(\chi_\pa)$ is $\tilde v_{0,1}=\chi_\pa-u_{(0)}$. Plugging this into~\eqref{EqARhoSqErr}, the total logarithmic term of $\wh\Box(\sigma)^{-1}\rho^2$ at $\zface$ is therefore
  \begin{equation}
  \label{EqARhoSqLog}
    (\log\hat r)(-\chi_\pa) + (\log\hat r)(\chi_\pa-u_{(0)}) = -(\log\hat r)u_{(0)},
  \end{equation}
  the first term coming from $\chi_\pa\tilde u^{(2)}$, the second from $(\log\hat r)\tilde v_{0,1}$.

  It remains to prove that the second term in~\eqref{EqARhoSqErrExp} merely contributes an error term in $\cA^{1-,1-,1-}(X_{\rm res}^+)$. But by~\eqref{EqASBoxFam} and Lemma~\ref{LemmaABox}, the operator $\sigma^{-1}(\wh\Box(\sigma)-\wh\Box(0))$ maps $\tilde v\in\cA^{1-,1-,0-}(X^+_{\rm res})$ into $\cA^{2-,2-,0-}(X^+_{\rm res})\subset|\sigma|^{-\delta}\cA^0([0,1),\cA^{2-\delta-})$ for any $\delta>0$; by Lemma~\ref{LemmaALoLargeInput}, this in turn mapped by $\wh\Box(\sigma)^{-1}$ into $|\sigma|^{-\delta}\cA^{1-\delta-,-\delta-,-\delta-}(X^+_{\rm res})$. Since $\delta>0$ was arbitrary, multiplication by $\sigma$ produces an element of $\cA^{1-,1-,1-}(X^+_{\rm res})$, as desired.
\end{proof}

\begin{rmk}
\label{RmkAModelExpl}
  We can explicitly compute
  \begin{equation}
  \label{EqAModelExpl}
    \tilde u^{(2)} = \wt\Box(1)^{-1}\hat r^{-2} = i\hat r^{-1}\Bigl(\frac{2\gamma+\log 4-i\pi}{4}+\int_0^\infty e^{-2 t}\log(\hat r+i t)\,d t\Bigr),
  \end{equation}
  where $\gamma$ is the Euler--Mascheroni constant; near $\hat r=0$, we have $\tilde u^{(2)}=-\log\hat r+\frac{i\pi}{2}+c+\cA^{1-}$ with $c\in\R$. This implies that in Proposition~\ref{PropARhoSq},
  \begin{equation}
  \label{EqAModelExplr0}
    \text{the imaginary part of the $\hat r^0$ coefficient of $\wh\Box(\sigma)^{-1}\rho^2$ at $\zface$ is $\frac{\pi}{2}u_{(0)}$.}
  \end{equation}
  (Indeed, the $\hat r^0$ term of $\tilde v$ in~\eqref{EqARhoSqErrExp} is $\wh\Box(0)^{-1}(\rho^2-[\wh\Box(0),\log\hat r](-\chi_\pa)-\frac{i\pi}{2}\wh\Box(0)(\chi_\pa))$, the imaginary part of which is $\frac{\pi}{2}(u_{(0)}-u_\pa)$; this gives the overall stated imaginary part of $\wh\Box(\sigma)^{-1}$ when plugged into~\eqref{EqARhoSqErr}.)

  The explicit solution~\eqref{EqAModelExpl} can be found as follows: the radial part $R$ of $\wt\Box(1)$ can be factored, $R=-\hat r^{-1}(\pa_{\hat r}+2 i)(\hat r\pa_{\hat r}+1)$. The equation $R\tilde u^{(2)}=\hat r^{-2}$ thus becomes $\pa_{\hat r}e^{2 i\hat r}v=-e^{2 i\hat r}\hat r^{-1}$ where $v=\pa_{\hat r}\hat r\tilde u^{(2)}$; this gives $v=e^{-2 i\hat r}\int_{\hat r}^\infty e^{2 i s}s^{-1}\,d s$. The constant of integration is absent to ensure the outgoing condition at $\hat r=\infty$. Indeed, deforming the integration contour to $\{\hat r+i\hat r t\colon t\in[0,\infty)\}$, gives $v=\int_0^\infty e^{-2 t\hat r}(t-i)^{-1}\,d t$; repeated integration by parts in $t$ using $(-2\hat r)^{-1}\pa_t e^{-2 t\hat r}=e^{-2 t\hat r}$ then shows that $v\in\hat\rho\CI([0,1)_{\hat\rho})$ is outgoing (meaning: conormal at $\hat r=0$). Now, $v=\int_0^\infty e^{-2 t}(t-i\hat r)^{-1}\,d t$ and $\int(t-i\hat r)^{-1}\,d\hat r=i c'+i\log(\hat r+i t)$ imply $\tilde u^{(2)} = i\hat r^{-1}(-c' + \int_0^\infty e^{-2 t}\log(\hat r+i t)\,d t)$. Requiring $\tilde u^{(2)}\in\cA^{-1+}$ near $\hat r=0$ forces $c'\in\C$ to be equal to the constant term of the integral at $\hat r=0$, giving~\eqref{EqAModelExpl}.
  
  The imaginary part of the constant term $\zeta$ of $\tilde u^{(2)}$ is equal to real part of the $\cO(\hat r)$ term of $I(\hat r):=\int_0^\infty e^{-2 t}\log(\hat r+i t)\,d t$. The proof of Lemma~\ref{LemmaAModel} gives the structure of the expansion and the coefficient of the logarithmic term in $I(\hat r)=-c'+i\hat r\log\hat r+\zeta\hat r+\cA^{2-}$. Thus, $\Re\zeta=\pa_{\hat r}\Re I(\hat r)|_{\hat r=0}=\lim_{\hat r\to 0}\int_0^\infty e^{-2 t}\frac{\hat r}{\hat r^2+t^2}\,d t=\frac{\pi}{2}$ upon substituting $t=\hat r y$.
\end{rmk}

\section{Price's law with a leading order term}
\label{SP}

We fix a stationary and asymptotically flat (with mass $\bhm$) metric $g$ (see Definition~\ref{DefAF}), which we moreover assume is spectrally admissible (see Definition~\ref{DefASAssm}). We abbreviate
\[
  \Box := \Box_g,\quad
  \la-,-\ra := \la-,-\ra_{L^2(X;|d g_X|)},
\]
where the density $|d g_X|$ is defined after~\eqref{EqASDensity}.

\subsection{Resolvent expansion}
\label{SsPR}

The key result of the paper is:

\begin{thm}
\label{ThmPRes}
  Let $f=f(x)\in\cA^{4+\alpha}(X)$. For positive frequencies, the resolvent acting on $f$ is then the form
  \begin{align*}
    &\wh\Box(\sigma)^{-1}f = u_{\rm sing}(\sigma) + u_{\rm reg}(\sigma), \\
    &\qquad u_{\rm sing}(\sigma) \in \sigma^2\cA^{1-,((0,0),1-),((0,1),1-)}(X^+_{\rm res}), \\
    &\qquad u_{\rm reg}(\sigma) \in \CI([0,1)_\sigma;\cA^{1-}(X)) + \cA^{\alpha-,2+\alpha-,((2,0),2+\alpha-)}(X^+_{\rm res}),
  \end{align*}
  where the leading terms of $\sigma^{-2}u_{\rm sing}(\sigma)$ at $\zface$ and $\tface$ are, respectively, $-(\log\tfrac{\sigma}{\rho})c_X(f)u_{(0)}$ and $c_X(f)\tilde u^{(2)}$ with $\tilde u^{(2)}=\wt\Box(1)^{-1}(\hat\rho^2)$ (see Proposition~\usref{PropARhoSq}); here,
  \[
    c_X(f) = \frac{\bhm}{\pi}\la f,u_{(0)}^*\ra.
  \]
\end{thm}

The subscript `sing' refers to the fact that $u_{\rm sing}$ captures the most singular ($\sigma^2\log\sigma$) behavior of the resolvent at $\zface$ (i.e.\ as $\sigma\to 0$ in $X^\circ$), while $u_{\rm reg}$ collects those terms in the resolvent expansion which are smooth down to $\sigma=0$ or at least more regular than $u_{\rm sing}$.

As already used in the proof of Proposition~\ref{PropARhoSq}, the strategy is to write
\begin{equation}
\label{EqPRewrite1}
\begin{split}
  \wh\Box(\sigma)^{-1}f &= \wh\Box(0)^{-1}f + (\wh\Box(\sigma)^{-1}-\wh\Box(0)^{-1})f \\
    &= \wh\Box(0)^{-1}f + \sigma \wh\Box(\sigma)^{-1}\bigl[\sigma^{-1}(\wh\Box(\sigma)-\wh\Box(0))\wh\Box(0)^{-1}f\bigr].
\end{split}
\end{equation}
In the second term, we gain a power of $\sigma$ due to $\wh\Box(\sigma)-\wh\Box(0)\in\sigma\rho\Diffb^1(X)$; however, $\wh\Box(0)^{-1}$ typically loses (at least) two orders of decay, while $\wh\Box(\sigma)-\wh\Box(0)$ typically only gains back one order, thus the $\sigma$-gain comes at the cost of reducing the decay of the argument of $\wh\Box(\sigma)^{-1}$. One can iterate the rewriting~\eqref{EqPRewrite1} while keeping track of the precise decay of the terms on which $\wh\Box(\sigma)^{-1}$ on the right in~\eqref{EqPRewrite1} acts. The outline of the proof of Theorem~\ref{ThmPRes} is then:
\begin{enumerate}
\item[(1.1)] First iteration (\S\ref{SssP1}): $u_0:=\wh\Box(0)^{-1}f=c_{(0)}\rho+\bhm c_{(0)}\rho^2+\dots$, $c_{(0)}=(4\pi)^{-1}\la f,u_{(0)}^*\ra$.
\item[(1.2)] Input of second term in~\eqref{EqPRewrite1}: $f_1:=-\sigma^{-1}(\wh\Box(\sigma)-\wh\Box(0))u_0=-2 i\bhm c_{(0)}\rho^3+\dots$
\item[(2.1)] Second iteration (\S\ref{SssP2}): $u_1:=\wh\Box(0)^{-1}f_1=2 i\bhm c_{(0)}\rho\log\rho+\dots$
\item[(2.2)] Input of next term in expansion: $f_2:=-\sigma^{-1}(\wh\Box(\sigma)-\wh\Box(0))u_1=4\bhm c_{(0)}\rho^2$.
\item[(3)] $\wh\Box(\sigma)^{-1}f_2$ is logarithmically divergent as $\sigma\to 0$ since $f_2$ (barely) fails to have sufficient decay, cf.\ Proposition~\ref{PropARhoSq}. This produces $u_{\rm sing}(\sigma)$.
\end{enumerate}

The total resolvent expansion being
\[
  \wh\Box(\sigma)^{-1}f = u_0 + \sigma u_1 + \sigma^2\wh\Box(\sigma)^{-1}f_2,
\]
step~(3) above provides the main contribution to the singular term $u_{\rm sing}(\sigma)$. We remark that the importance of the $\cO(\rho^2)$ subleading term of $u_0$ is also explained in the discussion of \cite[Proposition~6.14]{TataruDecayAsympFlat}.

\subsubsection{First iteration}
\label{SssP1}

We begin by analyzing the first term in~\eqref{EqPRewrite1} in some detail:

\begin{lemma}
\label{LemmaPL0inv4}
  For $f\in\cA^{4+\alpha}(X)$, we have
  \[
    \wh\Box(0)^{-1}f = c_{(0)}\rho + \bhm c_{(0)}\rho^2 + \rho^2 Y_{(1)} + \tilde u,
  \]
  where $c_{(0)}=(4\pi)^{-1}\la f,u_{(0)}^*\ra$ (with $u_{(0)}^*$ given by Lemma~\usref{LemmaAExtState}), $Y_{(1)}\in\cY_1=\mathspan\{Y_{1 m}\colon m=-1,0,1\}$, and $\tilde u\in\cA^{2+\alpha-}(X)$.
\end{lemma}
\begin{proof}
  We have $u:=\wh\Box(0)^{-1}f\in\cA^{1-}(X)$. But then, in the notation~\eqref{EqABoxPieces},
  \begin{equation}
  \label{EqPL0invL0}
    L_0 u = \rho^{-2}f - (\rho L_1+\tilde L)u \in \cA^{2+\alpha} + \cA^{2-} = \cA^{2-}.
  \end{equation}
  Let $\chi_\pa\in\CI(X)$ denote a cutoff which is identically $1$ near $\pa X$ and supported in a slightly larger neighborhood of $\pa X$. Then, using Lemma~\ref{LemmaABox}, $u_\pa:=\chi_\pa u\in\cA^{1-}$ satisfies an equation
  \begin{equation}
  \label{EqPL0inv4L0}
    L_0(u_\pa) = f_\pa \in \cA^{2-}.
  \end{equation}
  We use this equation to establish better decay of $u_\pa$ using a b-normal operator argument similarly to the proof of Lemma~\ref{LemmaAModel}. Passing to the Mellin transform in $\rho$, defined by $\wh{u_\pa}(\xi):=\int \rho^{-i\xi}u_\pa(\rho)\,\frac{d\rho}{\rho}$ (dropping the dependence on the spherical variables from the notation), equation~\eqref{EqPL0inv4L0} becomes
  \[
    \wh{L_0}(\xi)\wh{u_\pa}(\xi)=\wh{f_\pa}(\xi),\qquad
    \wh{L_0}(\xi):=-(i\xi)^2+i\xi+\slDelta.
  \]
  (The sign switch of $\xi$ compared to the proof of Lemma~\ref{LemmaAModel} is due to the Mellin transform there being in the variable $\hat r\sim\rho^{-1}$.) We already know that $\wh{u_\pa}(\xi)$ is holomorphic in $\Im\xi>-1$ and satisfies estimates
  \[
    \|\wh{u_\pa}(\xi)\|_{H^s(\Sph^2)} \leq C_{s N \eps}\la\xi\ra^{-N},\ -1+\eps<\Im\xi<0.
  \]
  for all $s,N\in\R$ and $\eps>0$; the same estimate holds for $\wh{f_\pa}(\xi)$ but in the larger region $-2<\Im\xi<0$. Expanding $\wh{u_\pa}(\xi)$ into spherical harmonics, so $\wh{u_\pa}(\xi)=\sum u_{m l}(\xi)Y_{m l}$ and $\wh{f_\pa}(\xi)=\sum f_{m l}(\xi)Y_{m l}$, and noting that $\wh{L_0}(\xi)|_{\cY_l}=-(i\xi)^2+i\xi+l(l+1)$ (with $\cY_l=\mathspan\{Y_{l m}\}$) is invertible for $i\xi\neq -l,l+1$, we conclude that $u_{m l}(\xi)=(\wh{L_0}(\xi)|_{\cY_l})^{-1}f_{m l}(\xi)$ is holomorphic in $\Im\xi>-1$ for all $m,l$ except possibly for $l=0$ where it has a simple pole. Therefore, in the inverse Mellin transform
  \[
    u_\pa(\rho) = \frac{1}{2\pi}\int_{\Im\xi=-1+\eps} \rho^{i\xi} \wh{L_0}(\xi)^{-1}\wh{f_\pa}(\xi)\,d\xi,
  \]
  we can shift the contour of integration through the pole at $\xi=-i$ to $\Im\xi=-2+\eps$; the residue theorem thus gives $u_\pa(\rho)\in c_{(0)}\rho+\cA^{2-}$ and therefore
  \begin{equation}
  \label{EqPL0invAsy0}
    u(\rho) = c_{(0)}\rho + \tilde u_0,\quad \tilde u_0\in\cA^{2-}.
  \end{equation}

  The constant $c_{(0)}$ can be evaluated as follows:\footnote{This is an instance of the proof of the relative index formula in \cite[\S6]{MelroseAPS}. Roughly speaking, the lack of invertibility of $\wh{L_0}(\xi)$ on $\cY_0$ for $\xi=-i$ is due to its lack of injectivity (the kernel producing the $c_{(0)}\rho$ leading order term) or, equivalently, due to its lack of surjectivity which manifests itself in the existence of a cokernel, which here is the kernel of $L_0^*$ on $\cA^0$; the constant $c_{(0)}$ then measures the failure of the right hand side $f_\pa$ in~\eqref{EqPL0inv4L0} to be orthogonal to the cokernel.} letting $\chi_\eps(\rho)=\chi(\rho/\eps)$ where $\chi\in\CI([0,\infty))$ vanishes near $0$ and is identically $1$ on $[1,\infty)$, we have
  \begin{align}
    0 &= \la u,\wh\Box(0)^*u_{(0)}^*\ra = \lim_{\eps\to 0} \la \chi_\eps u,\wh\Box(0)^*u_{(0)}^*\ra = \lim_{\eps\to 0} \la \wh\Box(0)\chi_\eps u,u_{(0)}^*\ra \nonumber\\
  \label{EqPL0invLot}
      &= \la f,u_{(0)}^*\ra + \lim_{\eps\to 0} \la [\wh\Box(0),\chi_\eps]u,u_{(0)}^*\ra.
  \end{align}
  In the second summand, note that $[\wh\Box(0),\chi_\eps]\in\rho^2\Diffb^1$ converges to $0$ strongly as an operator $\cA^\beta\to\cA^{\beta+2}$ for any $\beta\in\R$; hence $[\wh\Box(0),\chi_\eps]\tilde u_0\to 0$ in $\cA^{4-}$, and therefore $\la[\wh\Box(0),\chi_\eps]\tilde u_0,u_{(0)}^*\ra\to 0$. (Note that $|d g_X|$ is a smooth positive multiple of $r^2 |d r\,d\slg|=\rho^{-3}|\frac{d\rho}{\rho}d\slg|$ by Lemma~\ref{LemmaABox}, hence convergence to $0$ of the pointwise product of the two slots of the pairing in $\cA^{3+\delta}$ suffices for convergence of the inner product to $0$.) Similarly, all subleading terms of $\wh\Box(0)$ (i.e.\ terms in $\rho^3\Diffb^2$) do not contribute in the limit. Therefore, using the fact that the leading order term of $u_{(0)}^*$ at $\pa X$ is $1$, and using the explicit form of the term $L_0$ from~\eqref{EqABoxPieces}, the second term in~\eqref{EqPL0invLot} is equal to
  \begin{align*}
    &c_{(0)}\lim_{\eps\to 0}\la[\wh\Box(0),\chi_\eps]\rho,1\ra \\
    &\qquad = c_{(0)}\vol(\Sph^2)\lim_{\eps\to 0} \int_0^\infty \bigl[\rho^2\bigl(-\rho\pa_\rho [\rho\pa_\rho,\chi_\eps]-[\rho\pa_\rho,\chi_\eps](\rho\pa_\rho-1)\bigr)\rho\bigr)\,\rho^{-3}\frac{d\rho}{\rho} \\
      &\qquad = 4\pi c_{(0)} \int_0^\infty \bigl(-[\rho\pa_\rho,\chi_\eps]\bigr)\frac{d\rho}{\rho} \\
      &\qquad = -4\pi c_{(0)}.
  \end{align*}
  upon substituting $x=\rho/\eps$. Plugging this into~\eqref{EqPL0invLot} gives $c_{(0)}=(4\pi)^{-1}\la f,u_{(0)}^*\ra$.

  We sharpen the asymptotics of $u$ further by plugging the partial expansion~\eqref{EqPL0invAsy0} into equation~\eqref{EqPL0invL0}: using $L_0\rho\equiv 0$ and the explicit expression for $L_1$ in~\eqref{EqABoxPieces}, we obtain
  \[
    L_0\tilde u_0 = \rho^{-2}f - \rho L_1(c_{(0)}\rho) - (\rho L_1\tilde u_0 + \tilde L u) \in -2\bhm\rho^2 c_{(0)} + \cA^{2+\alpha},
  \]
  with the a priori information $\tilde u_0\in\cA^{2-}$. Localizing near $\pa X$ and using the Mellin transform as before, we now get a contribution to $\tilde u_0$ from the pole of $(\wh{L_0}(\xi)|_{\cY_1})^{-1}$ at $\xi=-2 i$, and an additional contribution from the single pole of $-2\bhm c_{(0)}\wh{\rho^2\chi_\pa}(\xi)$ at $\xi=-2 i$ (where $\wh{L_0}(\xi)$ acting on $\cY_0$ does \emph{not} have a pole). More directly, we can solve away $-2\bhm\rho^2 c_{(0)}$ by hand using $L_0(\bhm c_{(0)}\rho^2) = -2\bhm c_{(0)}\rho^2$, thus
  \[
    L_0 \tilde u_0' \in \cA^{2+\alpha},\qquad \tilde u_0':=\tilde u_0-\bhm c_{(0)}\rho^2;
  \]
  and then $\tilde u_0'=\rho^2 Y_{(1)}+\tilde u$ for some $Y_{(1)}\in\cY_1$ and $\tilde u\in\cA^{2+\alpha-}$. The proof is complete.
\end{proof}

Denote the output of Lemma~\ref{LemmaPL0inv4} by
\begin{equation}
\label{EqPu0}
  u_0 := \wh\Box(0)^{-1}f = c_{(0)}\rho+\bhm c_{(0)}\rho^2+\rho^2 Y_{(1)} + \tilde u.
\end{equation}
By~\eqref{EqASBoxFam} and in the notation of Lemma~\ref{LemmaABox}, we have
\begin{equation}
\label{EqPBoxDiff}
  {-}\sigma^{-1}\bigl(\wh\Box(\sigma)-\wh\Box(0)\bigr)= -2 i\rho(Q_0+\tilde Q) + g^{0 0}\sigma;
\end{equation}
and therefore, since $Q_0\rho=0$,
\begin{equation}
\label{EqPf1}
  f_1(\sigma) := -\sigma^{-1}\bigl(\wh\Box(\sigma)-\wh\Box(0)\bigr)u_0 = -2 i\bhm c_{(0)}\rho^3-2 i\rho^3 Y_{(1)} + \tilde f_1(\sigma)
\end{equation}
where $\tilde f_1(\sigma)\in\cA^{3+\alpha-}(X)+\sigma\cA^3(X)$.

\subsubsection{Second iteration; model problem and logarithmic singularity}
\label{SssP2}

The calculation~\eqref{EqPRewrite1} can now iterated: with $f_1(\sigma)$ given by~\eqref{EqPf1}, we have
\begin{equation}
\label{EqPExp30}
\begin{split}
  \wh\Box(\sigma)^{-1}f &= u_0 + \sigma\wh\Box(\sigma)^{-1}f_1(\sigma) \\
    &= u_0 + \sigma\wh\Box(0)^{-1}f_1(\sigma) - \sigma\wh\Box(\sigma)^{-1}(\wh\Box(\sigma)-\wh\Box(0))\wh\Box(0)^{-1}f_1(\sigma).
\end{split}
\end{equation}
Let us define
\begin{equation}
\label{EqPu1}
  u_1 := \wh\Box(0)^{-1}f_1(0)\in \cA^{1-}(X).
\end{equation}

\begin{lemma}
\label{LemmaPL0inv3}
  We have $u_1=2 i\bhm c_{(0)}\rho\log\rho+\tilde u_1$ with $\tilde u_1\in\cA^{((1,0),1+\alpha-)}(X)$.
\end{lemma}
\begin{proof}
  Writing $L_0 u_1\in\rho^{-2}f_1(0) + \cA^{2-}=-2 i\bhm c_{(0)}\rho-2 i\rho Y_{(1)}+\cA^{1+\alpha-}$ and passing to the Mellin transform, the logarithmic term of $u_1$ arises from a double pole at $\xi=-i$ due to \begin{enumerate*} \item the simple pole of $\wh{L_0}(\xi)$ acting on spherically symmetric functions and \item the simple pole of $\wh{\rho\chi_\pa}(\xi)$.\end{enumerate*} On the other hand, $\rho Y_{(1)}$ does not create a logarithmic term since $(\wh{L_0}(\xi)|_{\cY_1})^{-1}$ does not have a pole at $\xi=-i$. Concretely, we have
  \[
    L_0 \tilde u_1' \in \cA^{1+\alpha-},\qquad \tilde u_1' := u_1 - 2 i\bhm c_{(0)}\rho\log\rho + i\rho Y_{(1)}.
  \]
  By a normal operator argument as in the proof of Lemma~\ref{LemmaPL0inv4}, we conclude that $u_1'\in\rho Y_{(1)}'+\cA^{1+\alpha-}$ for some $Y_{(1)}'\in\cY_1$.
\end{proof}

We plug this into the expansion~\eqref{EqPExp30}. If we let $u_1'=\wh\Box(0)^{-1}(f_1')\in\cA^{1-}(X)$ where $f_1'=\sigma^{-1}(f_1(\sigma)-f_1(0))\in\cA^3(X)$, we have
\begin{equation}
\label{EqPExp3}
\begin{split}
  \wh\Box(\sigma)^{-1}f &= u_0 + \sigma u_1 + \sigma^2 u_1' - \sigma\wh\Box(\sigma)^{-1}\bigl(\wh\Box(\sigma)-\wh\Box(0)\bigr)u_1 \\
    &\qquad - \sigma^2\wh\Box(\sigma)^{-1}\bigl(\wh\Box(\sigma)-\wh\Box(0)\bigr)u_1'.
\end{split}
\end{equation}

For the fourth term on the right, which is the main term at this step, we compute
\begin{equation}
\label{EqPf2}
\begin{split}
  f_2(\sigma) &:= -\sigma^{-1}\bigl(\wh\Box(\sigma)-\wh\Box(0)\bigr)u_1 = f_{2,0} + \tilde f_2(\sigma), \\
  &\qquad
  f_{2,0}=4 \bhm c_{(0)}\rho^2,\quad
  \tilde f_2(\sigma) \in \cA^{2+\alpha-} + \sigma\cA^{3-},
\end{split}
\end{equation}
using~\eqref{EqPBoxDiff} again; note here that $\rho Q_0\tilde u_1\in\cA^{2+\alpha-}$ due to $Q_0\rho=0$. Therefore, we can apply Proposition~\ref{PropARhoSq} to $f_{2,0}=4\bhm c_{(0)}\rho^2$ to deduce
\begin{equation}
\label{EqPBoxInvf2}
  \wh\Box(\sigma)^{-1}f_{2,0} \in \cA^{1-,((0,0),1-),((0,1),1-)}(X^+_{\rm res}),
\end{equation}
with leading order term at $\zface$ equal to $-(\log\frac{\sigma}{\rho})4\bhm c_{(0)}u_{(0)}$, and leading order term at $\tface$ equal to $4\bhm c_{(0)}\tilde u^{(2)}$ (in the notation of the proposition); these are our main terms.

The remaining terms are error terms: Proposition~\ref{PropALoSmall} and Lemma~\ref{LemmaALoLargeInput} give, a fortiori,
\begin{gather}
\label{EqPBoxInvf20}
  \sigma^2\wh\Box(\sigma)^{-1}\tilde f_2(0) \in \sigma^2\cA^{\alpha-,\alpha-,((0,0),\alpha-)}(X^+_{\rm res}) = \cA^{\alpha-,2+\alpha-,((2,0),2+\alpha-)}(X^+_{\rm res}), \\
\label{EqPBoxtildef2}
  \sigma^2\wh\Box(\sigma)^{-1}(\tilde f_2(\sigma)-\tilde f_2(0)) \in \sigma^3\cA^{1-,0-,0-}(X^+_{\rm res}) = \cA^{1-,3-,3-}(X^+_{\rm res});
\end{gather}
in the last term in~\eqref{EqPExp3} finally, $\sigma^2\wh\Box(\sigma)^{-1}$ acts on an element of $\sigma\cA^{2-}+\sigma^2\cA^{3-}$, hence lies in the space~\eqref{EqPBoxtildef2} as well.

In summary, the expansion~\eqref{EqPExp3}, with $u_0$ and $u_1$ given by~\eqref{EqPu0} and~\eqref{EqPu1}, and using~\eqref{EqPf2}--\eqref{EqPBoxInvf2}, gives
\begin{align*}
  \wh\Box(\sigma)^{-1}f & \in \cA^{((1,0),2+\alpha-)}(X) + \sigma \cA^{((1,1),1+\alpha-)}(X) + \sigma^2\cA^{1-}(X) \\
    &\qquad + \sigma^2\cA^{1-,((0,0),1-),((0,1),1-)}(X^+_{\rm res}) + \cA^{\alpha-,2+\alpha-,((2,0),2+\alpha-)}(X^+_{\rm res}).
\end{align*}
Combining terms proves Theorem~\ref{ThmPRes}.

\subsection{Asymptotic behavior of waves}
\label{SsPW}

For simplicity, we first restrict our attention to the long time asymptotics in compact spatial sets in~\S\ref{SssPWCpt} before describing the global asymptotics in~\S\ref{SssPWGlobal}.

\subsubsection{Asymptotics in spatially compact sets}
\label{SssPWCpt}

\begin{thm}
\label{ThmPW}
  Let $\alpha\in(0,1)$, and let $f=f(t_*,x)\in\CIc(\R_{t_*};\cA^{4+\alpha}(X))$. Then the unique forward solution $\phi$ of $\Box\phi=f$ satisfies
  \begin{equation}
  \label{EqPWConst}
    |\phi(t_*,x) - c_M(f)u_{(0)}t_*^{-3}| \leq C t_*^{-3-\alpha+},\qquad
    c_M(f)=-\frac{2\bhm}{\pi}\int_{\R_{t_*}\times X} f(t_*,x)u_{(0)}^*(x)\,|d g|.
  \end{equation}
  for $x\in X$ lying in a fixed compact set $K\Subset X^\circ$. Derivatives of $\phi$ decay accordingly to
  \begin{equation}
  \label{EqPWDecay}
    |\pa_{t_*}^j\pa_x^\beta(\phi-c_M(f)u_{(0)}t_*^{-3})|\leq C_{j\beta} t_*^{-3-\alpha-j+}\qquad \forall\ j\in\N_0,\ \beta\in\N_0^3.
  \end{equation}
\end{thm}
\begin{proof}
  Using the Fourier transform $\hat\phi(\sigma,x)=\int_\R e^{i\sigma t_*}\phi(t_*,x)\,d t_*$, we express $\phi$ as
  \[
    \phi(t_*,x) = \cF^{-1}\bigl(\wh\Box(\sigma)^{-1}\hat f\bigr) = \frac{1}{2\pi} \int_\R e^{-i\sigma t_*}\wh\Box(\sigma)^{-1}\hat f(\sigma,x)\,d\sigma.
  \]
  The fact that this gives the (unique) \emph{forward} solution follows from the Paley--Wiener theorem upon deforming the integration contour to $i C+\R$ and letting $C\to\infty$; this uses the mode stability assumption~\eqref{ItASAssmMS} in Definition~\ref{DefASAssm}.
  
  Let $\chi\in\CIc((-1,1))$ be identically $1$ near $0$; we then split $\phi=\phi_0+\phi_1+\phi_2$,
  \begin{equation}
  \label{EqPWSplit}
  \begin{split}
    \wh{\phi_0}(\sigma) &:= \chi(\sigma)\wh\Box(\sigma)^{-1}\hat f(0), \\
    \wh{\phi_1}(\sigma) &:= \sigma\chi(\sigma)\wh\Box(\sigma)^{-1}\bigl(\sigma^{-1}(\hat f(\sigma)-\hat f(0))\bigr), \\
    \wh{\phi_2}(\sigma) &:= (1-\chi(\sigma))\wh\Box(\sigma)^{-1}\hat f(\sigma).
  \end{split}
  \end{equation}
  The high energy piece $\phi_2$ has strong decay: since we are not counting derivatives, we simply observe that by Lemma~\ref{LemmaASHiReg}, $\wh{\phi_2}(\sigma)$ is Schwartz in $\sigma$ with values in $\cA^{1-}(X)$. Therefore,
  \begin{equation}
  \label{EqPWu1}
    \phi_2(t_*,x)\in\sS(\R_{t_*};\cA^{1-}(X)).
  \end{equation}

  Consider next the low energy piece $\phi_0$. Since we are restricting to $\rho>0$, Theorem~\ref{ThmPRes} shows that $\phi_0$ is of class
  \begin{equation}
  \label{EqPWExpPos}
    \CI([0,1)_\sigma;\CI(X^\circ)) + \cA^{2+\alpha-}([0,1)_\sigma;\CI(X^\circ)) - c_X(f)\sigma^2\bigl((\log(\sigma)-\tfrac{i\pi}{2})u_{(0)}+ u'\bigr)
  \end{equation}
  for some \emph{real-valued} $u'\in\CI(X^\circ)$ (from the $\cO(\hat r^0)$ term of Proposition~\usref{PropARhoSq}); the constant $\frac{i\pi}{2}$ comes from Remark~\ref{RmkAModelExpl} (see also Remark~\ref{RmkPWConst} below). An inspection of the explicit expansion~\eqref{EqPExp3} as well as of the regular piece~\eqref{EqPBoxInvf20} show that the smooth pieces for $\sigma>0$ and for $\sigma<0$ fit together in a smooth fashion at $\sigma=0$. (See~\S\ref{SssPWGlobal} for a `resolved space' picture of this.) For real $\sigma$, we have $\wh\Box(\sigma)^{-1}f=\overline{\wh\Box(-\sigma)^{-1}\bar f}$; but note that for $\sigma<0$, and with the branch cut of $\log$ along $-i[0,\infty)$,
  \[
    \ol{\log(-\sigma)-\tfrac{i\pi}{2}} = \log(\sigma+i 0)-\tfrac{i\pi}{2},
  \]
  hence the logarithmic terms from $\pm\sigma>0$ combine to a $\sigma^2(\log(\sigma+i 0)-\frac{i\pi}{2})$ term. Absorbing the constant $-\frac{i\pi}{2}$ into the smooth parts, we thus obtain
  \begin{equation}
  \label{EqPWLo}
    \wh{\phi_0}(\sigma) \in \CIc((-1,1);\CI(X^\circ)) - c_X(f)u_{(0)}\sigma^2\log(\sigma+i 0) + \sum_\pm \cA^{2+\alpha-}(\pm[0,1);\CI(X^\circ)).
  \end{equation}
  The term $\wh{\phi_1}(\sigma)$ has an extra factor of $\sigma$, hence lies in the remainder space in~\eqref{EqPWLo}.

  Now, the smooth first term in~\eqref{EqPWLo} has rapidly vanishing (as $t_*\to\infty$, in compact subsets of $X^\circ$) inverse Fourier transform. Either of the conormal terms has Fourier transform which are bounded by $t_*^{-3-\alpha+}$ together with all their derivatives along $t_*\pa_{t_*}$ and $\pa_x$; see Lemma~\ref{LemmaPWConormal} below. The main contribution to $\phi_0$ and thus to $\phi$ comes from the logarithmic singularity.
  
  Recall then that the inverse Fourier transform of $(\sigma+i 0)^z$ (with the sign convention used to pass to the spectral family) is $\cF^{-1}\bigl((\sigma+i 0)^z\bigr)=e^{i\pi z/2}\chi_+^{-z-1}(t_*)$, where $\chi_+^z(t_*)$ is the holomorphic continuation, from $\Re z>-1$ to $z\in\C$, of $(t_*)_+^z/\Gamma(z+1)$ with $(t_*)_+=\max(t_*,0)$. For integer $z=k\geq 0$, $\chi_+^{-k-1}(t_*)=\pa_{t_*}^{k+1}\chi_+^0(t_*)=\delta^{(k)}(t_*)$ is supported at $t_*=0$; therefore, in $t_*>0$,
  \begin{equation}
  \label{EqPWFT}
    \cF^{-1}\bigl(\sigma^k\log(\sigma+i 0)\bigr) = \pa_z\cF^{-1}\bigl((\sigma+i 0)^z\bigr)|_{z=k} = -e^{i\pi k/2} \pa_z\bigl(\chi_+^z(t_*)\bigr)|_{z=-k-1}.
  \end{equation}
  To evaluate this, we note that in $t_*>0$,
  \begin{align*}
    \pa_z\chi_+^z(t_*)|_{z=-k-1} &= \pa_z\Bigl(\frac{(z+1)\cdots(z+k+1)}{\Gamma(z+k+2)}t_*^z\Bigr)|_{z=-k-1} \\
    & \qquad = \frac{(z+1)\cdots(z+k)}{\Gamma(z+k+2)}\Bigr|_{z=-k-1} t_*^{-k-1}
    = (-1)^k k! t_*^{-k-1}.
  \end{align*}
  In summary,
  \begin{equation}
  \label{EqPWi0}
    \cF^{-1}(\sigma^2\log(\sigma+i 0))=2 t_*^{-3},\qquad t_*>0,
  \end{equation}
  and this vanishes in $t_*<0$. The logarithmic term in~\eqref{EqPWLo} thus gives the $t_*^{-3}$ leading order term with the stated constant. This proves~\eqref{EqPWDecay}.
\end{proof}

\begin{rmk}
\label{RmkPWConst}
  The constant arising in~\eqref{EqPWExpPos} is \emph{necessarily} $\frac{i\pi}{2}$ by causality considerations. Indeed, in the derivation of the asymptotics above, only the imaginary component of this constant matters; let us write it as $i(\frac{\pi}{2}+c)$ for some $c\in\R$. The asymptotics of $\phi$ would then have an additional term $c_X(f)u_{(0)}\cF^{-1}(i c(\sigma_+^2 - \sigma_-^2))=-c_X(f)u_{(0)}c\pi^{-1}((t_*-i 0)^{-3}+(t_*+i 0)^{-3})$; this would be the \emph{only} contribution to $\phi$ which is not rapidly decaying as $t_*\to-\infty$. But since $\phi$ vanishes for large negative $t_*$, this is impossible unless $c=0$.
\end{rmk}

In the proof, we used the following standard result on inverse Fourier transforms of conormal distributions, whose proof we include for completeness:
\begin{lemma}
\label{LemmaPWConormal}
  Let $\beta>-1$. Let $\phi\in\sS'(\R_{t_*})$ with $\supp\hat\phi\subset[0,1)$ and $\hat\phi(\sigma)\in\cA^\beta([0,1)_\sigma)$. Then $\phi\in\CI(\R)$ is a symbol of order $-1-\beta$, i.e.\ $|(t_*\pa_{t_*})^j\phi(t_*)|\leq C_j\la t_*\ra^{-1-\beta}$ for all $j\in\N_0$.
\end{lemma}
\begin{proof}
  We work in $t_*>0$. Let $\chi\in\CIc([0,1))$ be $1$ near $0$. Then $2\pi\phi(t_*)=I+II$, where
  \[
    I = \int_0^1 e^{-i\sigma t_*}\chi(\sigma t_*)\hat\phi(\sigma)\,d\sigma, \qquad
    II = \int_0^1 e^{-i\sigma t_*}(1-\chi(\sigma t_*))\hat\phi(\sigma)\,d\sigma.
  \]
  We estimate $|I|\lesssim\int_0^{t_*^{-1}}|\sigma|^{-\beta}\,d\sigma\lesssim t_*^{-1-\beta}$. In $II$, we fix $k\in\N_0$, $k>\beta+1$, and write
  \[
    II = (i t_*)^{-k}\int_0^1 e^{-i\sigma t_*}\sigma^{-k}\cdot\sigma^k\pa_\sigma^k\bigl((1-\chi(\sigma t_*))\hat\phi(\sigma)\bigr)\,d\sigma.
  \]
  Expanding the derivative and substituting $\tilde\sigma=\sigma t_*$, we can estimate $|II|$ by the sum of two types of terms: the first arises from having all $\sigma$-derivatives fall on $\hat\phi$, giving
  \[
    t_*^{-k} t_*^{-1}\int_1^{t_*} (\tilde\sigma t_*^{-1})^{-k}(\tilde\sigma t_*^{-1})^\beta\,d\tilde\sigma \lesssim t_*^{-\beta-1};
  \]
  in the second type of terms, $j=1,\ldots,k$ derivatives fall on $\chi$, similarly giving
  \[
    t_*^{-k}\int \sigma^{-k}\cdot|(t_*\sigma)^j\chi^{(j)}(\sigma t_*)| |\sigma^{k-j}\pa_\sigma^{k-j}\hat\phi(\sigma)|\,d\sigma \lesssim t_*^{-k}t_*^{-1}\int_1^{t_*} (\tilde\sigma t_*^{-1})^{-k+\beta}\,d\tilde\sigma \lesssim t_*^{-\beta-1}.\qedhere
  \]
\end{proof}

\subsubsection{Global asymptotics}
\label{SssPWGlobal}

To cleanly describe the asymptotic behavior of $\phi$ in Theorem~\ref{ThmPW} in the full forward light cone, we pass to a compactification of the spacetime manifold in $t_*>0$.

\begin{definition}
\label{DefPWComp}
  Let $X=\ol{\R^3}$ as in Definition~\ref{DefACompact}. Compactify the nonnegative half line via $[0,\infty]_{t_*}:=([0,\infty)_{t_*}\sqcup[0,\infty)_\tau)/\sim$ where $\sim$ identifies $t_*>0$ and $\tau=t_*^{-1}$. Then the compactification of the causal future of $t_*=0$ inside of $M^\circ=(\R_{t_*}\times X^\circ,g)$ is defined by
  \[
    M_+ := \bigl[ [0,\infty]_{t_*} \times X; \{\infty\}\times\pa X \bigr],
  \]
  which contains $[0,\infty)\times X^\circ\subset M^\circ$ as an open dense subset. Thus, $M_+$ has four boundary hypersurfaces:
  \begin{enumerate}
  \item the \emph{Cauchy surface} $\Sigma=t_*^{-1}(0)\cong X$;
  \item \emph{null infinity} $\scri^+$: the lift of $[0,\infty]_{t_*}\times\pa X$;
  \item \emph{Minkowski future timelike infinity} $I^+$: the front face;
  \item \emph{spatially compact future infinity} $\cK^+$: the lift of $\tau^{-1}(0)$.
  \end{enumerate}
\end{definition}

See Figure~\ref{FigPWCpt}. Here, $\Sigma$ is an interior hypersurface, whereas the remaining three are boundary hypersurfaces `at infinity'. Correspondingly, the correct notion of regularity at $\Sigma$ is smoothness in the usual sense:

\begin{definition}
\label{DefPWSpaces}
  Denoting by $\rho_{\scri^+},\rho_{I^+},\rho_{\cK^+}\in\CI(M_+)$ boundary defining functions of the respective boundary hypersurfaces, we define the space of conormal functions
  \[
    \cA^{\alpha,\beta,\gamma}(M_+) := \rho_{\scri^+}^\alpha\rho_{I^+}^\beta\rho_{\cK^+}^\gamma\cA^{0,0,0}(M_+)
  \]
  which are smooth at $\Sigma$. That is, $\phi\in\cA^{0,0,0}(M_+)$ if and only if $\phi$ remains bounded upon application of any finite number of smooth vector fields on $M_+$ which are tangent to $\scri^+$, $I^+$, and $\cK^+$ (but not necessarily at $\Sigma$). Function spaces with partial polyhomogeneous expansions at some of the boundary hypersurfaces are denoted $\cA^{(\cE,\alpha),\beta,\gamma}(M_+)$ etc.\ as in Definition~\ref{DefAPhg}.
\end{definition}

\begin{figure}[!ht]
\centering
\includegraphics{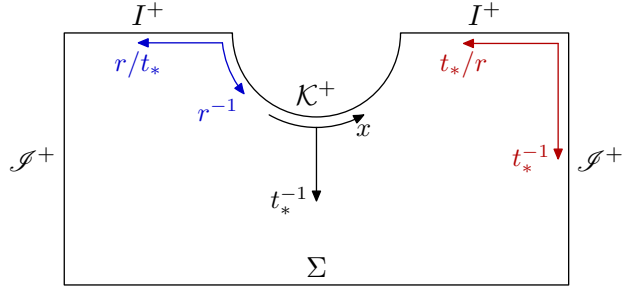}
\caption{The compactification $M_+$ from Definition~\ref{DefPWComp}, and some useful local coordinates: those in~\eqref{EqPWCompBdf1} in red, and those in~\eqref{EqPWCompBdf0} in blue. See also Figure~\ref{FigI2}.}
\label{FigPWCpt}
\end{figure}

Away from $\cK^+$,
\begin{equation}
\label{EqPWCompBdf1}
  \rho_{\scri^+,1} := \rho/\tau = \rho t_* = t_*/r,\quad
  \rho_{I^+,1} := \tau = t_*^{-1}
\end{equation}
are defining functions of $\scri^+$ and $I^+$, respectively; away from $\scri^+$ on the other hand,
\begin{equation}
\label{EqPWCompBdf0}
  \rho_{\cK^+,0} := \tau/\rho = (\rho t_*)^{-1} = r/t_*,\quad
  \rho_{I^+,0} := \rho = r^{-1}
\end{equation}
are defining functions of $\cK^+$ and $I^+$, respectively. In particular,
\begin{equation}
\label{EqPWIplus}
  I^+=[0,\infty]_{\rho_{\scri^+,1}} \times \pa X,
\end{equation}
with $I^+\cap\rho_{\scri^+,1}^{-1}(0)$ being the future boundary of $\scri^+$.

\begin{thm}
\label{ThmPWGlobal}
  Let $\alpha\in(0,1)$, $f\in\CIc((0,\infty)_{t_*};\cA^{4+\alpha}(X))$, and denote by $\phi$ the unique forward solution of $\Box\phi=f$. Then
  \[
    \phi \in \cA^{((1,0),3+\alpha),((3,0),3+\alpha-),((3,0),3+\alpha-)}(M_+).
  \]
  Setting $c_M(f)=-\frac{2\bhm}{\pi}\la f,u_{(0)}^*\ra$, the leading order terms of $\phi$ are:
  \begin{enumerate}
  \item\label{ItPWGlobalK} at $\cK^+$: $c_M(f)u_{(0)}t_*^{-3}$, with $c_M(f)$ given by~\eqref{EqPWConst};
  \item\label{ItPWGlobalI} at $I^+$: $c_M(f)u^+(\rho t_*)t_*^{-3}$, where
  \[
    u^+(v):=\frac{v(v+1)}{(v+2)^2}.
  \]
  \item\label{ItPWGlobalScri} at $\scri^+$: $\rho_{\scri^+,1}\phi_{\rm rad}(t_*,\omega)$, where\footnote{The exponent refers to decay at $\scri^+\cap I^+$, as measured by inverse powers of $t_*$.} $\phi_{\rm rad}\in\cA^{((3,0),3+\alpha-)}(\scri^+)$ satisfies
  \[
    \phi_{\rm rad}(t_*,\omega)-\tfrac14 c_M(f)t_*^{-3}\in\cA^{3+\alpha-}(\scri^+).
  \]
  \end{enumerate}
\end{thm}

The leading order terms match up, as they should: the leading order term at $I^+$ in part~\eqref{ItPWGlobalI} has asymptotics at $\scri^+$ and $\cK^+$ matching~\eqref{ItPWGlobalK} and \eqref{ItPWGlobalScri}.

\begin{rmk}
\label{RmkPWGlobalCompact}
  As long as $u_{(0)}$ is constant (which is the case here, but not in the more general context of~\S\ref{SsPV} below), we can capture the leading order behavior of $\phi$ in a more condensed form by writing $v=t_*/r$ and noting that $t_*(t_*+r)^{-1}$ is a global defining function for $\scri^+$; thus,
  \[
    \Bigl|\phi - c_M(f)\frac{t_*+r}{t_*^2(t_*+2 r)^2}\Bigr| \leq C t_*^{-2-\alpha+}(t_*+r)^{-1}.
  \]
  In terms of $\ft:=t_*+r$, the leading order term is $c_M(f)\ft/(\ft^2-r^2)^2$.
\end{rmk}

For the proof, it is convenient to glue together the resolved spaces for positive and negative frequencies, thus forming
\[
  X_{\rm res} := \bigl[ (-1,1)_\sigma \times X; \{0\}\times\pa X \bigr],
\]
which contains
\[
  X^\pm_{\rm res} := \bigl[ \pm[0,1)_\sigma\times X; \{0\}\times\pa X \bigr]
\]
as submanifolds with corners; see Figure~\ref{FigPWXres}. The point is that $X_{\rm res}$ allows us to track smoothness \emph{across} $\sigma=0$ (see the discussion leading up to~\eqref{EqPWLo}) while at the same time resolving the delicate zero energy behavior.

\begin{figure}[!ht]
\centering
\includegraphics{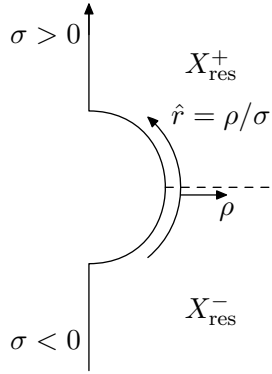}
\caption{The total resolved space $X_{\rm res}$ in which $X_{\rm res}^\pm$ are glued together along (the lift of) $\sigma=0$ (dashed line).}
\label{FigPWXres}
\end{figure}

\begin{proof}[Proof of Theorem~\usref{ThmPWGlobal}]
  At first, we shall only keep very rough track of the decay at $\scri^+$; we will recover sharp asymptotics there at the end of the proof. We revisit the proof of Theorem~\ref{ThmPW} and use the splitting
  \[
    \phi=\phi_0+\phi_1+\phi_2
  \]
  as in~\eqref{EqPWSplit}. By~\eqref{EqPWu1}, we have
  \begin{equation}
  \label{EqPWGlobalPhi2}
    \phi_2\in\cA^{1-,\infty,\infty}(M_+)
  \end{equation}

  \pfstep{Low energy contribution: regular part.} Consider next $\phi_0$, starting with the contribution from the regular part $u_{\rm reg}(\sigma)$ in the notation of Theorem~\ref{ThmPRes}. The contribution $\phi_{0,\rm reg,1}$ from the smooth first piece of $u_{\rm reg}(\sigma)$ is rapidly vanishing, that is,
  \begin{equation}
  \label{EqPWGlobalPhi0reg1}
    \phi_{0,\rm reg,1}\in\cA^{1-,\infty,\infty}(M).
  \end{equation}
  The second piece of $u_{\rm reg}(\sigma)$, $\sigma>0$, glues together with its negative frequency analogue to give an element
  \begin{align*}
    &u_{\rm reg,2} = u_{\rm reg,2,0} + \sum_\pm u_{\rm reg,2,\pm}, \\
    &\qquad u_{\rm reg,2,0}\in \cA^{\alpha-,2+\alpha-}(X_{\rm res}),\quad
 u_{\rm reg,2,\pm}\in \cA^{\alpha-,2+\alpha-,2+\alpha-}(X_{\rm res}^\pm).
  \end{align*}
  Since $u_{\rm reg,2,\pm}\in\cA^{2+\alpha-}(\pm[0,1);\cA^0(X))$, Lemma~\ref{LemmaPWConormal} implies that
  \begin{equation}
  \label{EqPWGlobalPhi0Reg2pm}
    \phi_{0,\rm reg,2,\pm}:=\cF^{-1}u_{\rm reg,2,\pm} \in \cA^{3+\alpha-}([0,1)_\tau;\cA^0(X)) \subset \cA^{0,3+\alpha-,3+\alpha-}(M_+).
  \end{equation}
  As for $\phi_{0,\rm reg,2,0}=\cF^{-1}u_{\rm reg,2,0}$, we used in the proof of Theorem~\ref{ThmPW} that $\phi_{0,\rm reg,2,0}$ vanishes rapidly in compact spatial sets due to the smoothness of $u_{\rm reg,2,0}$ is smooth across $\sigma=0$. Here, we need to refine this argument and take the degeneration at $\rho=\sigma=0$ into account. If
  \begin{equation}
  \label{EqPWGlobalPsi}
    \psi=\psi(\sigma/\rho)\in\CIc(\R),\quad \psi\equiv 1\ \text{near}\ 0,
  \end{equation}
  is a cutoff, then $(1-\psi)u_{\rm reg,2}\in\sum_\pm\cA^{\alpha-,2+\alpha-,\infty}(X_{\rm res}^\pm)$ lies a fortiori in the same space as $u_{\rm reg,2,\pm}$, hence has inverse Fourier transform contained in~\eqref{EqPWGlobalPhi0Reg2pm}. On the other hand, we can regard $\psi u_{\rm reg,2}\in\cA^{\infty,2+\alpha-}(X_{\rm res})$ as a function
  \begin{equation}
  \label{EqPWGlobalPhi0Reg2Space}
    \psi u_{\rm reg,2}(\rho,\hat r,\omega)\in\cA^{2+\alpha-}([0,1)_\rho;\CIc(\R_{\hat r}\times\pa X_\omega)),\qquad \hat r=\sigma/\rho.
  \end{equation}
  Its inverse Fourier transform is thus
  \begin{equation}
  \label{EqPWGlobalHatrFT}
  \begin{split}
    \frac{1}{2\pi}\int e^{-i\sigma t_*}\psi u_{\rm reg,2}\bigl(\rho,\tfrac{\sigma}{\rho},\omega\bigr)\,d\sigma &= \frac{\rho}{2\pi} \int e^{-i\hat r\cdot\rho t_*}\psi u_{\rm reg,2}(\rho,\hat r,\omega)\,d\hat r \\
      &= \rho\cF_{v\to\hat r}^{-1}(\psi u_{\rm reg,2})(\rho,\rho t_*,\omega).
  \end{split}
  \end{equation}
  By~\eqref{EqPWGlobalPhi0Reg2Space}, this is bounded by $\rho\rho^{2+\alpha-}\la\rho t_*\ra^{-N}$ for all $N$, and in fact lies in $\cA^{3+\alpha-,3+\alpha-,\infty}(M_+)$ (cf.\ $\rho_{\cK^+,0}$ in~\eqref{EqPWCompBdf0}). In summary,
  \begin{equation}
  \label{EqPWGlobalPhi0Reg2}
    \phi_{0,\rm reg,2}=\cF^{-1}u_{\rm reg,2} \in \cA^{0,3+\alpha-,3+\alpha-}(M_+).
  \end{equation}

  \pfstep{Low energy contribution: singular part.} The main contribution to $\phi_0$ at $I^+$ and $\cK^+$ comes from the singular part $u_{\rm sing}(\sigma)$ in Theorem~\ref{ThmPRes}. Note that $\sigma\pa_\sigma+\rho\pa_\rho-2$ annihilates the leading order term of $u_{\rm sing}(\sigma)$ at $\tface$, hence
  \[
    \tilde u_{\rm sing} := (\sigma\pa_\sigma+\rho\pa_\rho-2)u_{\rm sing} \in \sum_\pm\cA^{1-,3-,((2,1),3-)}(X_{\rm res}^\pm).
  \]
  With $\psi$ as in~\eqref{EqPWGlobalPsi}, the pieces $(1-\psi)\tilde u_{\rm sing}\in\cA^{1-,3-,\infty}(X_{\rm res}^\pm)\subset\cA^{3-}([0,1)_\sigma;\cA^0(X))$ have inverse Fourier transform in $\cA^{0,4-,4-}(X)$. The piece $\psi\tilde u_{\rm sing}\in\sum_\pm\cA^{\infty,3-,((2,1),3-)}(X_{\rm res}^\pm)$ on the other hand is the sum of two terms: the first is the logarithmic term
  \[
    \psi a(\rho,\omega)\hat r^2\log(\hat r+i 0),\qquad a\in\cA^{3-}(X);
  \]
  its inverse Fourier transform is bounded by $\rho\rho^{3-}(\rho t_*)^{-3}$ by the calculation~\eqref{EqPWGlobalHatrFT}, and in fact lies in $\cA^{4-,4-,((3,0),4-)}(M_+)$ (thus gives part of the $t_*^{-3}$ leading order term of $\phi$ at $\cK^+$). The other piece lies in $\cA^{\infty,3-,3-}(X_{\rm res}^\pm)\subset\cA^{3-}(\pm[0,1);\cA^0(X))$ for $\pm\sigma>0$, which contributes $\cA^{0,4-,4-}(M_+)$. Overall,
  \[
    \tilde\phi_{\rm sing} := \cF^{-1}\tilde u_{\rm sing} = (-t_*\pa_{t_*}+\rho\pa_\rho-3)\phi_{\rm sing} \in \cA^{0,4-,((3,0),4-)}(M_+),
  \]
  where we set $\phi_{\rm sing}=\cF^{-1}u_{\rm sing}$, with $u_{\rm sing}$ denoting the sum of the singular pieces for positive and negative frequencies.

  But $-t_*\pa_{t_*}-r\pa_r-3$ is the radial differential operator which precisely annihilates $\rho_{I^+}^3$ leading order terms at $I^+$; in fact, we have
  \[
    -t_*\pa_{t_*}+\rho\pa_\rho-3 = \rho_{I^+,1}\pa_{\rho_{I^+,1}}-3=\rho_{I^+,0}\pa_{\rho_{I^+,0}}-3,
  \]
  the second expression being valid in the coordinates $\rho_{I^+,1}=t_*^{-1}$, $\rho_{\scri^+,1}=\rho t_*$ from~\eqref{EqPWCompBdf1} away from $\cK^+$, and the third expression being valid away from $\scri^+$ in the coordinates $\rho_{\cK^+,0}$, $\rho_{I^+,0}$ from~\eqref{EqPWCompBdf0}.
  
  The conclusion is that the $\rho_{I^+}^3$ leading order term of $\phi$ at $I^+$ is equal to that of the inverse Fourier transform of the leading order term of $u_{\rm sing}$ at $\tface$, extended by (degree $-2$) homogeneity along $(\rho,\sigma)\mapsto(\lambda\rho,\lambda\sigma)$, $\lambda>0$, i.e.\ to the inverse Fourier transform of
  \[
    \sigma^2 c_X(\hat f(0))\bigl(\tilde u^{(2)}(\tfrac{\sigma}{\rho}) H(\tfrac{\sigma}{\rho}) + \ol{\tilde u^{(2)}}(-\tfrac{\sigma}{\rho})H(-\tfrac{\sigma}{\rho})\bigr)
  \]
  in $\sigma$; here, $H(x)=x_+^0$ is the Heaviside function. Thus, the two leading order terms at $\tface$ from $X_{\rm res}^\pm$ get glued together at the front face of $X_{\rm res}$, with a resulting mild logarithmic singularity at the `seam' $\sigma=0$ of the form $\hat r^2\log(\hat r+i 0)$; cf.\ Figure~\ref{FigPWXres}. Writing $\hat r=\sigma/\rho$, and factoring $\sigma^2=\rho^2\hat r^2$, this is equal to
  \begin{align}
    &c_X(\hat f(0))\rho^3 \frac{1}{2\pi}\Bigl(\int_0^\infty e^{-i\hat r \rho t_*}\hat r^2\tilde u^{(2)}(\hat r)+e^{i\hat r \rho t_*}\hat r^2\ol{\tilde u^{(2)}}(\hat r)\,d\hat r\Bigr) \nonumber\\
    &\qquad= \pi^{-1}c_X(\hat f(0))\rho^3\Re\Bigl(\int_0^\infty e^{-i\hat r v}\hat r^2\tilde u^{(2)}(\hat r)\,d\hat r\Bigr)\Bigr|_{v=\rho t_*} \nonumber\\
  \label{EqPWOscInt}
    &\qquad = -\pi^{-1}c_X(\hat f(0))(\rho t_*)^{-2}\rho^3\Re\Bigl(\int_0^\infty e^{-i\hat r v}\pa_{\hat r}^2\hat r^2\tilde u^{(2)}(\hat r)\,d\hat r\Bigr)\Bigr|_{v=\rho t_*};
  \end{align}
  here we integrated by parts twice using $i v^{-1}\pa_{\hat r}e^{-i\hat r v}=e^{-i\hat r v}$. The integral can be evaluated explicitly.\footnote{This is not surprising: in the context of works by Baskin--Vasy--Wunsch \cite{BaskinVasyWunschRadMink,BaskinVasyWunschRadMink2} and, more directly, Baskin--Marzuola \cite{BaskinMarzuolaCone}, the leading order behavior at $I^+$ is directly related to resonances of exact hyperbolic space with a cone point at the origin (i.e.\ forgetting that hyperbolic space is smooth across the origin), this being the model of Minkowski space with a line of conic points along $r=0$.} Indeed, since $\pa_{\hat r}^2\hat r^2=(\hat r\pa_{\hat r}+2)\pa_{\hat r}\hat r$, the arguments in Remark~\ref{RmkAModelExpl} show that
  \[
    (\hat r\pa_{\hat r}+2)\pa_{\hat r}\hat r\tilde u^{(2)} = (\hat r\pa_{\hat r}+2)\int_0^\infty e^{-2 t\hat r}(t-i)^{-1}\,d t = 2 \int_0^\infty e^{-2 t\hat r}(1-t\hat r)(t-i)^{-1}\,d t.
  \]
  We then regularize the integral in~\eqref{EqPWOscInt} by inserting a factor $e^{-\eps\hat r}$ and letting $\eps\searrow 0$, giving
  \[
    2\int_0^\infty e^{-i\hat r(v-i 0)}\Bigl(\int_0^\infty e^{-2 t\hat r}(1-t\hat r)(t-i)^{-1}\,d t\Bigr)d\hat r = \int_0^\infty\frac{2(t+i v)}{(2 t+i v)^2(t-i)}\,d t.
  \]
  This integral is now easily evaluated, and its real part is $\frac{2\pi(v+1)}{(v+2)^2}$. Plugging this into~\eqref{EqPWOscInt} and using that $c_X(\hat f(0))=-\half c_M(f)$, this gives the desired leading order term at $I^+$.

  Putting this together with~\eqref{EqPWGlobalPhi2}, \eqref{EqPWGlobalPhi0reg1}, \eqref{EqPWGlobalPhi0Reg2pm}, and \eqref{EqPWGlobalPhi0Reg2}, and noting that the piece $\phi_1$ has an extra order of $t_*$-decay relative to $\phi_0$, we have now shown
  \begin{equation}
  \label{EqPWPhiAlmost}
    \phi \in \cA^{0,((3,0),3+\alpha-),((3,0),3+\alpha-)}(M_+),
  \end{equation}
  with the claimed leading order terms at $I^+$ and $\cK^+$.

  \pfstep{Decay at null infinity.} The asymptotic behavior can be determined by integration along approximate characteristics as in \cite[\S5]{HintzVasyMink4}. Let us use more compact notation,
  \[
    v:=\rho_{\scri^+,1}=\rho t_*,\quad
    \tau=\rho_{I^+,1} = t_*^{-1},
  \]
  near $\scri^+\subset M_+$, so that $\pa_{t_*}=\tau(-\tau\pa_\tau+v\pa_v)$ and $\rho\pa_\rho=v\pa_v$; we work in the neighborhood $M':=[0,1)_\tau\times[0,1)_v\times\pa X$ of $\scri^+\cap I^+\subset M_+$. Then Lemma~\ref{LemmaABox} implies
  \[
    \Box = 2 v\tau^2\bigl(A + R\bigr),\quad A:=(\tau\pa_\tau-v\pa_v)(v\pa_v-1)\in\Diffb^2(M'),\ R\in v\Diffb^2(M').
  \]
  By~\eqref{EqPWPhiAlmost}, $\phi\in\cA^{0,3}(M')=v^0\tau^3\cA^{0,0}(M')$ is a solution of
  \[
    (A+R)\phi = f' :=(2 v\tau^2)^{-1}f\in\cA^{3+\alpha,\infty}(M').
  \]

  We rewrite this equation as $A\phi=f'-R\phi\in\cA^{1,3}(M')$. Integrating $\tau\pa_\tau-v\pa_v$ with initial data at $v=\half$ of class $\cA^3([0,1)_\tau)$, this gives $(v\pa_v-1)\phi\in\cA^{1,3}(M')$ and thus $\phi\in\cA^{1-,3}(M')$, the (really: logarithmic) loss at $v=0$ due to $v^1$ being an indicial solution of $v\pa_v-1$. Iterating this argument once more, we find that $A\phi\in\cA^{2-,3}(M')$, and the inversion of the regular singular operator $v\pa_v-1$ gives the $v^1$ leading order term of $\phi$ at $\scri^+$, so $\phi\in\cA^{((1,0),2-),3}(M')$.

  More precisely, note that $\phi|_{v=1/2}\in\cA^{((3,0),3+\alpha-)}([0,1)_\tau)$ has a $\tau^3$ leading order term at $\tau=0$, and $R\phi\in\cA^{1,((3,0),3+\alpha-}(M_+)$, by~\eqref{EqPWPhiAlmost}. Integration gives $\phi\in\cA^{1-,((3,0),3+\alpha-)}(M')$. Hence,
  \begin{equation}
  \label{EqPWInt}
    A\phi = f'-R\phi \in \cA^{2-,((3,0),3+\alpha-)}.
  \end{equation}
  Integrating this shows that the $v^1$ leading order term (the radiation field) $\phi_{\rm rad}$ of $\phi$ at $\scri^+$ itself has a leading order term
  \[
    \phi_{\rm rad} - c_{\rm rad}(f)\tau^3 \in \cA^{3+\alpha-}([0,1)_\tau),\quad c_{\rm rad}(f)\in\CI(\pa X);
  \]
  moreover, $c_{\rm rad}(f)$ matches the leading order term of $\phi$ at $I^+$ as one approaches $\scri^+$. That is, we have an equality of leading order terms at $v=\tau=0$ of $c_M(f)t_*^{-3}u^+(\rho t_*)=c_M(f)\tau^3 u^+(v)$ (from $I^+$) and of $c_{\rm rad}(f)v\tau^3$ (from $\scri^+$), hence $c_{\rm rad}(f)$ is a constant given by
  \[
    c_{\rm rad}(f) = c_M(f) \lim_{v\to 0}v^{-1} u^+(v) = \tfrac14 c_M(f).
  \]
  
  We finally note that we can iterate the normal operator argument at $\scri^+$ based on~\eqref{EqPWInt}, until the $v^{3+\alpha}$ decay of $f'$ prevents getting further terms in the expansion; the result is that $\phi\in\cA^{((1,0),3+\alpha),((3,0),3+\alpha)}(M')$. The proof is complete.
\end{proof}

Initial value problems can be reduced to forcing problems:
\begin{cor}
\label{CorPWIVP}
  Let $\phi_0,\phi_1\in\CIc(X^\circ)$.\footnote{It suffices to assume $\phi_0\in\cA^{4+\alpha}(X)$, and sufficient decay (depending on the precise decay of $|d t_*|^2$ as $r\to\infty$) of $\phi_1$.} Then the solution $\phi$ of the initial value problem
  \[
    \Box\phi = 0,\quad
    \phi(0,x) = \phi_0(x),\quad
    \pa_{t_*}\phi(0,x) = \phi_1(x),
  \]
  has the asymptotic behavior stated in Theorem~\usref{ThmPWGlobal} for any $\alpha<1$, with the constant $c_M(f)$ replaced by
  \[
    c_M(\phi_0,\phi_1) := \frac{2\bhm}{\pi}\big\la{-}([\Box,t_*]\phi_0)|_{t_*=0}+|d t_*|^2\phi_1,u_{(0)}^*\big\ra.
  \]
\end{cor}
\begin{proof}
  Fix $R>1$ so that $\supp\phi_0$ and $\supp\phi_1$ are contained in the ball $B(0,R)\subset X^\circ$ of radius $R$. Then there exists $\eps>0$ such that $\phi(t_*,x)=0$ for $|x|\geq 2 R$ and $t_*\in[0,3\eps]$. Fix a cutoff $\chi(t_*)$ which is identically $0$ for $t_*\leq 0$ and equal to $1$ for $t_*\geq 2\eps$. Then $\phi_+=\chi\phi$ vanishes in $t_*\leq\eps$ and satisfies
  \begin{equation}
  \label{EqKWIVPRedux}
   \Box\phi_+ = [\Box,\chi]\phi \in \CIc(\R_{t_*};\CIc(X^\circ));
  \end{equation}
  its asymptotic behavior as $t_*\to\infty$ can thus be computed using Theorem~\ref{ThmPWGlobal}. The constant $c_M([\Box,\chi]\phi)$ must be independent of $\chi$. One can thus evaluate it by taking $\chi$ to be the Heaviside function $H(t_*)$, which by Lemma~\ref{LemmaABox} gives
  \[
    [\Box,H]\phi = -2\rho Q\delta(t_*)\phi_0(x) - g^{0 0}(x)\bigl(\delta(t_*)\phi_1(x)+\delta'(t_*)\phi_0(x)\bigr),
  \]
  which pairs against $u_{(0)}^*$ to $-\la 2\rho Q\phi_0+g^{0 0}\phi_1,u_{(0)}^*\ra_{L^2(X)}$, as claimed.
\end{proof}

\subsection{Wave equations with stationary potentials}
\label{SsPV}

As a simple generalization, we briefly consider wave equations with a stationary potential $V$ decaying like $r^{-4}$, or more precisely
\begin{equation}
\label{EqPVPotential}
  V \in \rho^4\CI(X);
\end{equation}
we can allow $V$ to be complex-valued. With $\Box=\Box_g$ the wave operator of a stationary and asymptotically flat (with mass $\bhm$) metric $g$, we consider the operator
\[
  P_V := \Box_g+V.
\]
We assume that $P_V$ spectrally admissible, that is, $P_V$ has no nontrivial zero energy bound states, no nontrivial mode solutions with frequency $0\neq\sigma\in\C$, $\Im\sigma\geq 0$, and that high energy estimates hold---these assumptions are precisely those in Definition~\ref{DefASAssm} but for $P_V$ in place of $\Box$. Following the proof of Lemma~\ref{LemmaAExtState}, there exist extended zero energy (dual) states
\begin{equation}
\label{EqPVStates}
  \cA^0(X)\cap\ker\wh{P_V}(0) = \C u_{(0)},\qquad
  \cA^0(X)\cap\ker\wh{P_V}(0)^* = \C u_{(0)}^*,
\end{equation}
where $u_{(0)}$ and $u_{(0)}^*$ are uniquely determined by $u_{(0)},u_{(0)}^*\in 1+\cA^{1-}(X)$. (The two are related by $u_{(0)}^*=\ol{u_{(0)}}$, and thus equal for real-valued potentials.)

\begin{thm}
\label{ThmPV}
  For $P_V$ as above, the unique forward solution of
  \[
    P_V\phi = f \in \CIc(\R_{t_*};\cA^{4+\alpha}(X))
  \]
  satisfies the asymptotics stated in Theorems~\usref{ThmPW} and \usref{ThmPWGlobal}, with $u_{(0)}$ and $u_{(0)}^*$ given by~\eqref{EqPVStates}. (In particular, the shape $u^+$ of the leading order term at $I^+$ does not depend on $V$.)
\end{thm}
\begin{proof}
  The analysis of the resolvent $\wh{P_V}(\sigma)^{-1}$ acting on $\cA^{4+\alpha}(X)$ in Theorem~\ref{ThmPRes} goes through \emph{verbatim}. Indeed, the decay assumption~\eqref{EqPVPotential} ensures that in the zero energy operator $\rho^{-2}\wh{P_V}(0)=\rho^{-2}\wh\Box(0)+\rho^{-2}V$, the potential enters at the same level as the error term $\tilde L$ in~\eqref{EqABoxPieces0}; this term did not play any role in the arguments above.
\end{proof}

Long range potentials
\begin{equation}
\label{EqPVLongRange}
  V\in\rho^3\CI(X)
\end{equation}
can be handled as well. Denote by $V_0\in\CI(\pa X)$ the leading order term, i.e.\ $V-V_0\rho^3\in\rho^4\CI(X),\quad V_0\in\CI(\pa X)$. (We do not require $V_0$ to be spherically symmetric, which removes a requirement made in \cite{TataruDecayAsympFlat}.) Then the spherically symmetric part $\bar V_0$ of $V_0$ enters at the same level as the mass parameter $\bhm$; Theorem~\ref{ThmPV} thus remains valid upon changing the mass $\bhm$ in the definition of the constant $c_M(f)$ in~\eqref{EqPWConst} by the effective mass
\[
  \bhm(V):=\bhm+\half\bar V_0,\qquad \bar V_0:=\frac{1}{4\pi}\int_{\pa X}V_0\,|d\slg|.
\]
Indeed, $\rho^{-2}\wh{P_V}(0)(\rho+\bhm(V)\rho^2)\in\rho^3\CI(X)$; thus, for $f\in\cA^{4+\alpha}(X)$, we have $\wh{P_V}(0)^{-1}f\equiv c_{(0)}(\rho+\bhm(V)\rho^2)$ up to terms in $\rho^2\CI(X)$ with vanishing spherical average (just like $Y_{(1)}$ in Lemma~\ref{LemmaPL0inv4}---the vanishing of the spherical average of $Y_{(1)}$, i.e.\ the orthogonality to spherically symmetric functions, is all that was used in subsequent arguments). This gives the claimed correction to $u_0$ in~\eqref{EqPu0}, and the remainder of the calculation is unaffected.

As a concrete example, one can take $g$ to be the Minkowski metric $g=-d t^2+d r^2+r^2\slg$, put $t_*=t-\la r\ra$ (thus, $g$ satisfies Definition~\ref{DefAF} with $\bhm=0$), and take $V$ to be a smooth potential of the form $V(r,\omega)=r^{-3}V_0(\omega)+r^{-4}\tilde V(r^{-1},\omega)$ in $r>1$, with $V_0,\tilde V$ smooth; for instance, $V=\la r\ra^{-3}$. Thus, if $\int_{\Sph^2}V_0\,|d\slg|\neq 0$, solutions of $(\Box_g+V)\phi=f$ generically have precisely $t_*^{-3}$ decay as $t_*\to\infty$; the constant prefactor $c_M(\phi_0,\phi_1)$ of the $u_{(0)}t_*^{-3}$ leading order term can be computed as in Corollary~\ref{CorPWIVP} with $\bhm(V)=\half\bar V_0$ in place of $\bhm$. In particular, modifying $t_*$ to be equal to $t$ near $\supp\phi_0\cup\supp\phi_1$, this gives $c_M(\phi_0,\phi_1)=-\frac{\bar V_0}{\pi}\la\phi_1,u_{(0)}^*\ra$ and thus proves Theorem~\ref{ThmIV}.\footnote{As a rough check of this result (following a suggestion by Tataru), one can compute the solution $\psi(t,r,\omega)$ of $\Box_g\psi=(-D_t^2+\Delta_{\R^3})\psi=0$ with initial data $(\psi,\pa_t\psi)|_{t=0}=(0,\phi_1)$, $\phi_1\in\CIc(B(0,R_0))$, and then couple this with a radial potential $V=V(r)$ by solving the forced equation $\Box_g\phi=-V\psi$. The result is
\[
  \phi(T,0) = -\int_{\half(T-R_0)}^{\half(T-R_0)} \frac{1}{4\pi s}V(s)\int_{\Sph^2}\frac{1}{4\pi(T-s)}\Bigl(\int_{x\cdot\omega=T-2 s}\phi_1(x)\,d x\Bigr)s^2\,d s.
\]
Approximating $\frac{V(s)s^2}{4\pi s}\approx\frac{\bar V_0}{4\pi s^2}\approx\frac{\bar V_0}{\pi T^2}$ and similarly replacing $\frac{1}{T-s}$ by $\frac{2}{T}$, one obtains
\[
  \phi(T,0)\approx-\frac{2\bar V_0}{\pi T^3}\int\int_{x\cdot\omega=T-2 s}\phi_1(x)\,d x\,d s = -\frac{\bar V_0}{\pi}\la\phi_1,1\ra T^{-3}.
\]
}

\begin{rmk}
\label{RmkPVConormal}
  Similarly to Remark~\ref{RmkAConormal}, one can relax~\eqref{EqPVLongRange} to $\rho^3\CI(X)+\cA^4(X)$ without causing any changes in the argument, and to $\rho^3\CI(X)+\cA^{3+\beta}(X)$, $\beta>0$, with extra work.
\end{rmk}

\section{Asymptotics on subextremal Kerr spacetimes}
\label{SK}

We recall the definition of subextremal Kerr spacetimes in~\S\ref{SsKD} and describe the (minimal) adaptations of the setup of~\S\ref{SA} and the main theorems on wave decay in~\S\S\ref{SsKS}--\ref{SsKW}.

\subsection{Definition of the metric; asymptotics}
\label{SsKD}

The Kerr metric \cite{KerrKerr} with mass $\bhm>0$ and angular momentum $\bha\in(-\bhm,\bhm)$ (this is the \emph{subextremal range} of angular momenta) is the Ricci flat metric given by
\begin{align*}
  g_{\bhm,\bha} &= -\frac{\Delta_r}{r_\bha^2}(d t-\bha\,\sin^2\theta\,d\varphi)^2 + r_\bha^2\Bigl(\frac{d r^2}{\Delta_r}+d\theta^2\Bigr) + \frac{\sin^2\theta}{r_\bha^2}\bigl(\bha\,d t-(r^2+\bha^2)d\varphi\bigr)^2, \\
  &\qquad \Delta_r:=r^2-2\bhm r+\bha^2,\quad
  r_\bha^2:=r^2+\bha^2\cos^2\theta,
\end{align*}
in Boyer--Lindquist coordinates $t\in\R$, $r\in(r_{\bhm,\bha},\infty)$, $\theta\in(0,\pi)$, $\varphi\in(0,2\pi)$, where
\[
  r_{\bhm,\bha} := \bhm+\sqrt{\bhm^2-\bha^2};
\]
the dual metric $g_{\bhm,\bha}^{-1}$ is
\begin{equation}
\label{EqKDDual0}
  r_\bha^2 g_{\bhm,\bha}^{-1} = -\Delta_r^{-1}\bigl((r^2+\bha^2)\pa_t+\bha\pa_\varphi\bigr)^2 + \Delta_r\pa_r^2+\pa_\theta^2+\sin^{-2}\theta\bigl(\pa_\varphi+\bha\sin^2\theta\,\pa_t\bigr)^2.
\end{equation}
In order to extend $g_{\bhm,\bha}$ across $r=r_{\bhm,\bha}$ on the one hand, and place it into the setting of Definition~\ref{DefAF} for large $r$ on the other hand, we define new coordinates
\begin{equation}
\label{EqKDtstar}
\begin{split}
  t_* &:= t + \int_{3\bhm}^r \frac{r^2+\bha^2}{\Delta_r}(\chi_+(r)-\chi_\pa(r))\,d r + F(r), \\
  \varphi_* &:= \varphi+\int_{3\bhm}^r \frac{\bha}{\Delta_r}\chi_+(r)\,d r,
\end{split}
\end{equation}
where $\chi_+,\chi_\pa\in\CI(\R)$ with $\chi_+\equiv 1$ on $[0,3\bhm)$, $\chi_+\equiv 0$ on $[4\bhm,\infty)$, and $\chi_\pa\equiv 0$ on $[0,5\bhm]$, $\chi_\pa\equiv 1$ on $[6\bhm,\infty)$, and $F(r)\in\CI(\R)$ will be chosen below. The dual metric then is
\begin{equation}
\label{EqKDDual}
\begin{split}
  r_\bha^2 g_{\bhm,\bha}^{-1} &= \Delta_r\pa_r^2 + \pa_\theta^2 + \sin^{-2}\theta\bigl(\pa_{\varphi_*}+\bha\sin^2\theta\,\pa_{t_*}\bigr)^2 - \frac{1-\chi_+^2}{\Delta_r}\bigl((r^2+\bha^2)\pa_{t_*}+\bha\pa_{\varphi_*}\bigr)^2 \\
    &\quad + \Delta_r\Bigl(-\frac{r^2+\bha^2}{\Delta_r}\chi_\pa+F'\Bigr)^2\pa_{t_*}^2 + 2\bigl[ ((r^2+\bha^2)(\chi_+-\chi_\pa)+F'\Delta_r)\pa_{t_*}+\bha\chi_+\pa_{\varphi_*}\bigr]\pa_r \\
    &\quad + 2\chi_+\bigl((r^2+\bha^2)\pa_{t_*}+\bha\pa_{\varphi_*}\bigr)\Bigl(-\frac{r^2+\bha^2}{\Delta_r}\chi_\pa+F'\Bigr)\pa_{t_*}.
\end{split}
\end{equation}
For $F\equiv 0$, this can be extended analytically from $r_{\bhm,\bha}<r<3\bhm$ to $r>r_-=\bhm-\sqrt{\bhm^2-\bha^2}$. In $r\geq 6\bhm$, we want the inner product
\[
  r_\bha^2 g_{\bhm,\bha}^{-1}(d t_*,d t_*) = \bha^2\sin^2\theta - 2(r^2+\bha^2)F' + (F')^2\Delta_r
\]
to be negative; this holds for example for $F'=\frac{\bhm^2}{r^2+\bha^2}$, for which $r_\bha^2 g_{\bhm,\bha}^{-1}(d t_*,d t_*)=-2\bhm^2+\bha^2\sin^2\theta+\cO(r^{-2})$ is indeed negative for $r\geq r_0\gg 6\bhm$. Choosing $F$ in $r\leq r_0$ suitably, we can then ensure that $d t_*$ is future timelike everywhere in $r\geq\bhm$. Define the spatial manifold
\[
  X^\circ := [\bhm,\infty)_r \times \Sph^2_{\theta\varphi_*};
\]
we then state the asymptotic behavior of the metric on the closure
\begin{equation}
\label{EqKDCompact}
  X:=[\bhm,\infty]\times\Sph^2\subset\ol{\R^3}
\end{equation}
of $X^\circ$ inside the radial compactification $\ol{\R^3}$ of $\R^3$ (see Definition~\usref{DefACompact}); by a slight abuse of notation, we write
\begin{equation}
\label{EqKDBdy}
  \pa X := X \cap \pa\ol{\R^3},\quad
  \Sigma_{X,\rm f} := X \cap r^{-1}(\bhm)
\end{equation}
for the boundary of $X$ at infinity, resp.\ the artificial boundary inside the black hole. Let $\rho=r^{-1}$ denote a defining function of $\pa X$. Then:

\begin{lemma}
\label{LemmaKDAsy}
  For a suitable choice of $F\in\CI(\R_r)$, with $F'(r)=\frac{\bhm^2}{r^2+\bha^2}$ for $r\gg 1$, the dual metric $g_{\bhm,\bha}^{-1}$ in the coordinates $t_*,r,\theta,\varphi_*$ defined in~\eqref{EqKDtstar}, is a smooth, nondegenerate Lorentzian dual metric on $M^\circ=\R_{t_*}\times X^\circ$ so that $d t_*$ is everywhere future timelike, and so that $g_{\bhm,\bha}^{-1}$ satisfies the assumptions in part~\eqref{ItAFAsy} of Definition~\usref{DefAF} near $\pa X$.
\end{lemma}
\begin{proof}
  The choice of $F$ ensures $g_{\bhm,\bha}^{-1}(d t_*,d t_*)\in r_\bha^{-2}\CI(X)\subset\rho^2\CI(X)$; the asymptotics of the remaining metric coefficients can be readily verified by inspection of~\eqref{EqKDDual}.
\end{proof}

See Figure~\ref{FigI2}. We note that in $r>r_{\bhm,\bha}$, the zero energy operator can be computed directly in Boyer--Lindquist coordinates; changing to the coordinate $\varphi_0:=\varphi+\int_{3\bhm}^r \frac{\bha}{\Delta_r}\,d r$ (which is~\eqref{EqKDtstar} without the cutoff), which is regular across $r=r_{\bhm,\bha}$, we obtain
\begin{equation}
\label{EqKDZeroOpReg}
  \wh{\Box_g}(0) = -r_\bha^{-2}\pa_r\Delta_r\pa_r - r_\bha^{-2}(\sin\theta)^{-1}\pa_\theta(\sin\theta\,\pa_\theta) - 2\bha r_\bha^{-2}\pa_r\pa_{\varphi_0} - r_\bha^{-2}(\sin\theta)^{-2}\pa_{\varphi_0}^2.
\end{equation}

\emph{For the rest of this section, we fix a subextremal Kerr metric}
\begin{equation}
\label{EqKDFixed}
  g_{\bhm,\bha},\qquad \Box\equiv\Box_{g_{\bhm,\bha}},\qquad \la-,-\ra\equiv\la-,-\ra_{L^2(X;|d g_X|)},\ |d g_X|= r_\bha^2\sin\theta\,d r\,d\theta\,d\varphi_*;
\end{equation}
as in~\eqref{EqASDensity}, the density on $X$ is defined so that $|d t_*||d g_X|=|d g_{\bhm,\bha}|$.

\subsection{Spectral theory}
\label{SsKS}

For forward solutions of the wave equation, and for the inversion of the spectral family of $\Box$, the correct notion of regularity across the artificial boundary $\Sigma_{X,\rm f}$ (see~\eqref{EqKDBdy}) is smoothness in the usual sense, i.e.\ we consider \emph{extendible distributions} in the terminology of \cite[Appendix~B]{HormanderAnalysisPDE3}; for the adjoint problem on the other hand, we use \emph{supported distributions}:

\begin{definition}
\label{DefKSSpaces}
  \begin{enumerate}
  \item For $s,\ell\in\R$, the space $\Hbext^{s,\ell}(X)=\rho^\ell\Hbext^s(X)$ (of extendible distributions) is the space of restrictions of elements of $\Hb^{s,\ell}(\ol{\R^3})$ to $X$. The semiclassical space $\Hbhext^{s,\ell}(X)$ is defined analogously.
  \item The space $\Hbsupp^{s,\ell}(X)$ (of supported distributions) is the subspace of $\Hb^{s,\ell}(\ol{\R^3})$ consisting of all those elements which are supported in $X$.
  \item For $\alpha\in\R$, we define $\bar\cA^\alpha(X)=\rho^\alpha\bar\cA^0(X)$ as the space of restrictions of elements of $\cA^\alpha(\ol{\R^3})$ to $X^\circ$.
  \end{enumerate}
\end{definition}

Thus, $\Hbext^{s,\ell}(X)$ is a Hilbert space as the quotient of $\Hb^{s,\ell}(\ol{\R^3})$ by the subspace of elements with support in the closure of $\ol{\R^3}\setminus X$. Elements of $\bar\cA^\alpha(X)$ are smooth down to $\Sigma_{X,\rm f}$, and the Sobolev embedding~\eqref{EqASobolevEmb} holds for the extendible spaces.

\begin{lemma}
\label{LemmaKSAdm}
  Subextremal Kerr metrics $g_{\bhm,\bha}$ are spectrally admissible in the sense that they satisfy conditions~\eqref{ItASAssmBound}--\eqref{ItASAssmMS} in Definition~\usref{DefASAssm} on the corresponding extendible function spaces $\bar\cA^1(X)$. Moreover, they satisfy the high energy estimates~\eqref{EqASAssmHi} in condition~\eqref{ItASAssmHi} on $\bar H_{\bop,|\sigma|^{-1}}^{s,\ell}(X)$ for $\delta=1$, under the additional assumption that $s>\half$. Moreover, the estimate~\eqref{EqASAssmMed} for bounded nonzero energies holds on extendible function spaces as well for $s>\half$.
\end{lemma}
\begin{proof}
  For Schwarzschild metrics $g_{\bhm,0}$, mode stability and the absence of bound states are proved using simple integration by parts arguments; see \cite[Theorem~6.1]{HaefnerHintzVasyKerr} for detailed arguments in the function spaces used in the present paper. In the reference, it is also shown that mode stability and absence of bound states are open conditions in $\bha$, thus hold for slowly rotating Kerr black holes as well. In the general subextremal Kerr case, mode stability for nonzero $\sigma$ was proved by Whiting \cite{WhitingKerrModeStability} and Shlapentokh-Rothman \cite{ShlapentokhRothmanModeStability}; see also \cite{AnderssonMaPaganiniWhitingModeStab} for generalizations; the absence of zero energy bound states was argued for in \cite{PressTeukolskyKerrII,TeukolskySeparation}.
  
  The estimates for bounded nonzero energies as well as the high energy estimates are proved in \cite[Theorem~4.3]{HaefnerHintzVasyKerr} (note that the compact error terms in the estimate for bounded energies can be dropped due to mode stability). We recall that for bounded frequencies, these estimates rely on \begin{enumerate*} \item radial point estimates at the event horizon \cite{VasyMicroKerrdS,MelroseEuclideanSpectralTheory} (requiring the regularity to be above the threshold $\half$ in order to exclude singular behavior there, cf.\ $u_{(0)}^*\in H_\loc^{1/2-}(X^\circ)$ in Lemma~\ref{LemmaKSExtState} below, and see \cite[Footnote~3]{HaefnerHintzVasyKerr}) combined with real principal type propagation \cite{DuistermaatHormanderFIO2}; \item scattering theory near $\pa X$ in second microlocal function spaces \cite{VasyLAPLag}; \item standard elliptic estimates in $\bhm\ll r<\infty$.\end{enumerate*} For high energy estimates, one uses the semiclassical versions of these estimates in \cite{VasyMicroKerrdS,VasyLAPLag}, see also \cite{VasyZworskiScl}, as well as estimates at the normally hyperbolically trapped set of Kerr spacetimes; this structural nature of the trapped set was first noted for small $\bha$ by Wunsch--Zworski \cite{WunschZworskiNormHypResolvent}, together with the requisite high energy estimates, and proved in the full subextremal range by Dyatlov \cite{DyatlovWaveAsymptotics}. The high energy estimates lose only a logarithmic power of $h$, whereas we simply allow ourselves a loss of a full power (see also \cite{DyatlovSpectralGaps}).
\end{proof}

The extended zero energy states from Lemma~\ref{LemmaAExtState} are as follows:
\begin{lemma}
\label{LemmaKSExtState}
  For $s<\half$ and $\ell\in(-\tfrac52,-\tfrac32)$, we have
  \begin{alignat*}{2}
    \cA^0(X) \cap \ker\wh\Box(0) &= \C u_{(0)}, &\qquad u_{(0)}&=1, \\
    \Hbsupp^{s,\ell}(X) \cap \ker\wh\Box(0)^* &= \C u_{(0)}^*, &\qquad u_{(0)}^*&=H(r-r_{\bhm,\bha}).
  \end{alignat*}
\end{lemma}
\begin{proof}
  This is part of \cite[Proposition~6.2]{HaefnerHintzVasyKerr}. We can also follow the proof of Lemma~\ref{LemmaAExtState} directly for $\ker\wh\Box(0)$ (with extendible conormal function spaces instead), while in the construction of $u_{(0)}^*$ we use that $\ker\wh\Box(0)^*\colon\Hbsupp^{-\infty,\ell}(X)\to\Hbsupp^{-\infty,\ell+2}(X)$ is injective for $\ell\in(-\tfrac32,-\half)$ to prove the existence of a (unique) $u_{(0)}^*\in 1+\Hbsupp^{-\infty,-1/2-}$ in the kernel of $\wh\Box(0)$. Concretely, we certainly have $\Box 1=0$ (thus $u_{(0)}=1$), and then~\eqref{EqKDZeroOpReg} implies
  \[
    [\Box_{g_{\bhm,\bha}}(0),H(r-r_{\bhm,\bha})] = -2\bha r_\bha^{-2}\delta(r-r_{\bhm,\bha})\pa_{\varphi_0},
  \]
  since $[\pa_r\Delta_r\pa_r,H(r-r_{\bhm,\bha})]=\delta(r-r_{\bhm,\bha})\Delta_r\pa_r+\pa_r\delta(r-r_{\bhm,\bha})\Delta_r=0$ since $\Delta_r|_{r=r_{\bhm,\bha}}=0$. Thus, if $u\in\ker\Box_g$ is smooth and rotationally symmetric, then $\Box_g(H(r-r_{\bhm,\bha})u)=0$, too; we thus have $u_{(0)}^*=H(r-r_{\bhm,\bha})u_{(0)}$.
\end{proof}

\subsection{Asymptotics of waves}
\label{SsKW}

Upon using $u_{(0)}$ and $u_{(0)}^*$ from Lemma~\ref{LemmaKSExtState}, the arguments in~\S\ref{SP} now apply with only one notational change, namely we need to use spaces of extendible conormal distributions encoding smoothness down to the hypersurface $\Sigma_{X,\rm f}\subset X$ and $\Sigma_{\rm f}=\R_{t_*}\times\Sigma_{X,\rm f}\subset\R_{t_*}\times X$. The resulting statement is:

\begin{thm}
\label{ThmKW}
  Let $\bhm>0$ and $\bha\in(-\bhm,\bhm)$ be subextremal Kerr parameters, define $X=[\bhm,\infty]\times\Sph^2$ as in~\eqref{EqKDCompact}, and let $g_{\bhm,\bha}$ and $t_*$ be as in Lemma~\usref{LemmaKDAsy}. Let $\alpha\in(0,1)$, fix $f\in\CIc((0,\infty)_{t_*};\bar\cA^{4+\alpha}(X))$, and denote by $\phi$ the unique forward solution of
  \[
    \Box_{g_{\bhm,\bha}}\phi=f.
  \]
  Define the constant $c(f)$ by
  \[
    c(f) = -\frac{2\bhm}{\pi}\int_{r>r_{\bhm,\bha}}\int_0^{2\pi}\int_0^\pi f(t_*,r,\theta,\varphi_*)r_\bha^2\sin\theta\,d\theta\,d\varphi_*\,d r.
  \]
  Writing $x=(r,\theta,\varphi_*)\in\R^3$, we then have pointwise decay estimates
  \[
    \bigl|\pa_{t_*}^j\pa_x^\beta\bigl(\phi(t_*,x) - c(f)t_*^{-3}\bigr)\bigr| \leq C_{j\beta}t_*^{-3-\alpha-j+}\qquad \forall\ j\in\N_0,\ \beta\in\N_0^3,
  \]
  for $x$ restricted to any fixed compact subset of $X^\circ$. These are a special case of the \emph{global} asymptotics, valid for $t_*>1$ and \emph{all} $x=(r,\theta,\varphi_*)\in X^\circ$,
  \begin{equation}
  \label{EqKWGlobal}
    \Bigl| (t_*\pa_{t_*})^j (r\pa_r)^k \Omega^\gamma\Bigl(\phi - c(f)\frac{t_*+r}{t_*^2(t_*+2 r)^2}\Bigr) \Bigr| \leq C_{j k\gamma}t_*^{-3-\alpha+}\frac{t_*}{t_*+r}\qquad \forall\ j,k\in\N_0,\ \gamma\in\N_0^3,
  \end{equation}
  where $\Omega=(\Omega_1,\Omega_2,\Omega_3)\subset\cV(\Sph^2)$ denotes the collection of rotation vector fields. In particular, letting $r\to\infty$, the radiation field of $\phi$, defined by
  \[
    F(t_*,\omega) := \lim_{r\to\infty} r\phi(t_*,r,\omega),
  \]
  is given by $F(t_*,\omega)=\tfrac14 c(f)t_*^{-2}+\cO(t_*^{-2-\alpha+})$ and satisfies the pointwise decay estimates
  \[
    \bigl| \pa_{t_*}^j\Omega^\gamma\bigl(F - \tfrac14 c(f)t_*^{-2}\bigr) \bigr| \leq C_{j\gamma}t_*^{-2-j-\alpha+}\qquad\forall\ j\in\N_0,\ \gamma\in\N_0^3.
  \]
\end{thm}
\begin{proof}
  The constant $c(f)$ is equal to $-2\bhm\pi^{-1}\la f,u_{(0)}^*\ra$, cf.~\eqref{EqKDFixed}. The claims thus follow from Theorems~\ref{ThmPW} and~\ref{ThmPWGlobal}; we merely need to comment on the vector fields arising in the global estimate. Define the compactified spacetime manifold $M_+$ as in Definition~\usref{DefPWComp} but with $X$ as above; and define the function spaces $\bar\cA^{\alpha,\beta,\gamma}(M_+)$ as in Definition~\usref{DefPWSpaces} but with the testing vector fields unrestricted at $\Sigma_{\rm f}$. (Thus, elements of $\bar\cA^{\alpha,\beta,\gamma}(M_+)$ are smooth across $\Sigma_{\rm f}$.) Then we in fact have
  \[
    \phi \in \bar\cA^{((1,0),3+\alpha-),((3,0),3+\alpha-),((3,0),3+\alpha-)}(M_+),
  \]
  with explicit leading order terms. The significance of the vector fields in~\eqref{EqKWGlobal} is then that they are indeed b-vector fields on $M_+$. In fact, they span all of $\Vb(M_+)$ over $\CI(M_+)$, as is verified by direct calculation:
  \begin{enumerate}
  \item near the interior $(\cK^+)^\circ$, the manifold $M_+$ is a product $(\cK^+)^\circ\times[0,1)_\tau$ with $\tau=t_*^{-1}$ and $(\cK^+)^\circ=X^\circ$. Thus, b-vector fields are locally spanned by $\tau\pa_\tau=-t_*\pa_{t_*}$ and spatial derivatives (which are expressible in terms of bounded multiples of $\pa_r$ and spherical vector fields);
  \item near $\cK^+\cap I^+$ with boundary defining functions $v'=r/t_*$ and $\mu'=r^{-1}$ of $\cK^+$ and $I^+$, respectively, the vector fields $t_*\pa_{t_*}=-v'\pa_{v'}$ and $r\pa_r=v'\pa_{v'}-\mu'\pa_{\mu'}$, together with rotation vector fields, indeed span $\Vb(M_+)$ locally;
  \item near $\scri^+\cap I^+$ with $v=\rho t_*$ and $\tau=t_*^{-1}$ defining $\scri^+$ and $I^+$, respectively, the claim similarly follows from $t_*\pa_{t_*}=v\pa_v-\tau\pa_\tau$, $r\pa_r=-\rho\pa_\rho=-v\pa_v$.\qedhere
  \end{enumerate}
\end{proof}

\begin{rmk}
\label{RmkKWBL}
  In $r>r_0>2 r_{\bhm,\bha}$, we can use Boyer--Lindquist coordinates; since $t_*=t-r+\cO(\log r)$, we can replace $t_*+r$ in~\eqref{EqKWGlobal} by $t$ upon committing a lower order (in terms of powers of $t^{-1}$) error as long as in addition we stay in any fixed forward timelike cone $r<(1-\delta)t$. This gives the asymptotics
  \begin{equation}
  \label{EqKWBL}
    \Bigl|(t\pa_t)^j(r\pa_r)^k\Omega^\gamma\Bigl(\phi-c(f)\frac{t}{(t^2-r^2)^2}\Bigr)\Bigr| \lesssim t^{-3-\alpha+},\qquad r_0<r<(1-\delta)t.
  \end{equation}
\end{rmk}

This implies a corresponding result for the initial value problem precisely as in Corollary~\ref{CorPWIVP}. For initial data supported away from the horizon, the constant in the leading order term takes a particularly simple form:

\begin{cor}
\label{CorKWExpl}
  Suppose $\phi_0,\phi_1\in\CIc((r_{\bhm,\bha},\infty)\times\Sph^2)$ are initial data with compact support disjoint from the event horizon. Then the asymptotic behavior of the solution $\phi$ of the initial value problem in Boyer--Lindquist coordinates
  \[
    \Box_{g_{\bhm,\bha}}\phi=0,\quad
    \phi(t=0,x)=\phi_0(x),\quad
    \pa_t\phi(t=0,x)=\phi_1(x),
  \]
  has the asymptotic behavior given in Theorem~\usref{ThmKW} (with $\alpha<1$ arbitrary) upon replacing the constant $c(f)$ there by the constant
  \begin{equation}
  \label{EqKWExplConst}
  \begin{split}
    c(\phi_0,\phi_1) &= -\frac{2\pi}{\bhm}\iiint_{r>r_{\bhm,\bha}} \Bigl[\Bigl(\frac{(r^2+\bha^2)^2}{\Delta_r}-\bha^2\sin^2\theta\Bigr)\phi_1(r,\theta,\varphi) \\
      &\hspace{12em} + \frac{4\bhm\bha r}{\Delta_r}\phi_0(r,\theta,\varphi)\Bigr] \sin^2\theta\,d r\,d\theta\,d\varphi.
  \end{split}
  \end{equation}
  For initial data with support intersecting the event horizon, the constant is given by Corollary~\usref{CorPWIVP}.
\end{cor}
\begin{proof}
  The expression~\eqref{EqKWExplConst} comes directly from Corollary~\ref{CorPWIVP} upon \begin{enumerate*} \item choosing the time function $t_*$ in such a way that it is equal to the Boyer--Lindquist time coordinate $t$ near the support of the initial data, \item using the form~\eqref{EqKDDual0} of the dual metric $g_{\bhm,\bha}^{-1}$, and \item recalling the spatial volume density from~\eqref{EqKDFixed}.\end{enumerate*}
\end{proof}

\begin{rmk}
\label{RmkKWV}
  Under a spectral admissibility condition as discussed in~\S\ref{SsPV}, one can couple scalar waves on subextremal Kerr spacetimes to a potential $V\in\rho^4\CI(X)$, obtaining the same decay rates, though the coefficient of the $t_*^{-3}$ leading order term in compact spatial sets is modified by the potential precisely as in Theorem~\ref{ThmPV}. Long range potentials $V\in\rho^3\CI(X)$ are covered as well, with additional modifications as discussed after~\eqref{EqPVLongRange}.
\end{rmk}

\section{The full Price law on Schwarzschild spacetimes}
\label{SPF}

As another application of our methods, we prove the full Price law on the Schwarzschild spacetime which predicts pointwise $t^{-2 l-3}$ decay of linear scalar waves with fixed `angular momentum' $l\in\N_0$. Concretely, recall the space
\[
  \cY_l := \mathspan\{ Y_{l m} \colon m=-l,\ldots,l \} \subset \CI(\Sph^2)
\]
of degree $l$ spherical harmonics; $\slDelta Y=l(l+1)Y$ for $Y\in\cY_l$. Recalling the Schwarzschild metric for $\bhm>0$, given by $g_\bhm=g_{\bhm,0}$, from~\eqref{EqAMSchw}, we can introduce a time coordinate $t_*$ as in Lemma~\ref{LemmaKDAsy}. We use the notation $X$ from~\eqref{EqKDCompact} for the compactification of the spatial manifold $X^\circ=[\bhm,\infty)_r\times\Sph^2_\omega$, and $\bar\cA^\alpha(X)$ for functions conormal at $\pa X=\rho^{-1}(0)$, $\rho=r^{-1}$, with weight $\rho^\alpha$, and smooth across the artificial interior hypersurface $r=\bhm$. We say that a function $f$ on $M^\circ=\R_{t_*}\times X^\circ$ is \emph{supported in angular frequency $l\in\N_0$} if $\slDelta f=l(l+1)f$, or equivalently $f(t_*,r,-)\in\cY_l$ for all $t_*,r$; and we say that $f$ is supported in angular frequencies $\geq l$ if the $L^2(\Sph^2)$-orthogonal projection of $f(t_*,r,-)$ to $\cY_0\oplus\cdots\oplus\cY_{l-1}$ vanishes for all $t_*,r$.

\begin{thm}
\label{ThmPF}
  Let $\alpha\in(0,1)$ and $l\in\N_0$, and write $\Box\equiv\Box_{g_\bhm}$, $\bhm>0$. Let $f\in\CIc(\R_{t_*};\bar\cA^{4+l+\alpha}(X))$ and suppose $f$ is supported in angular frequencies $\geq l$. Then the unique forward solution $\phi$ of $\Box\phi=f$ obeys the pointwise decay bounds
  \[
    | \phi(t_*,r,\omega) | \leq C t_*^{-2 l-3},
  \]
  together with all derivatives along $t_*\pa_{t_*}$, $\pa_r$, and spherical vector fields, for $x=(r,\omega)$ restricted to any fixed compact subset $K\Subset X^\circ$. Moreover, this decay rate is generically sharp when $K$ has nonempty interior: it can be improved if and only if $f$ satisfies $2 l+1$ linearly independent constraints.

  In the full future causal cone $t_*\geq 0$, the pointwise decay rate of $\phi$ (and of its derivatives along any product of powers of $t_*\pa_{t_*}$, $r\pa_r$, $\Omega$) is $t_*^{-l-3}$, and the radiation field has $t_*^{-l-2}$ decay (as do its derivatives along any product of powers of $t_*\pa_{t_*}$, $\Omega$).
\end{thm}

\begin{rmk}
\label{RmkPFComplete}
  One can explicitly compute the leading order coefficient $Y$ of $\phi$ for which $\phi=Y_{(0)}t_*^{-2 l-3} + \cO(t_*^{-2 l-3-\alpha+})$ in compact subsets of $X^\circ$; it is an element in the $(2 l+1)$-dimensional space of solutions $u_{(l)}$ of $\wh\Box(0)u_{(l)}=0$ with leading order asymptotic behavior $u_{(l)}-r^l Y\in\bar\cA^{-l+1-}(X)$, $Y\in\cY_l$. For example, for $l=1$, this space is spanned by $(r-\bhm)Y_{1 m}$, $m=-1,0,1$ (see \cite[Proposition~6.2]{HaefnerHintzVasyKerr}). Moreover, one can, in principle, prove \emph{global} asymptotics similar to those in Theorem~\ref{ThmPWGlobal}, with explicit leading order term. However, we shall not pursue this here beyond the rough description given in Theorem~\ref{ThmPF}.
\end{rmk}

\begin{rmk}
\label{RmkPFGeneral}
  The proof of Theorem~\ref{ThmPF} below works verbatim for all \emph{spherically symmetric}, stationary and asymptotically flat metrics with mass $\bhm\in\R$ (see Definition~\ref{DefAF}) which are spectrally admissible (see Definition~\ref{DefASAssm}).
\end{rmk}

\begin{proof}[Proof of Theorem~\usref{ThmPF}]
  We shall assume $l\geq 1$, the case $l=0$ being a special case of Theorem~\ref{ThmKW}. In order to unburden the notation, we write $\cA^\alpha$ for $\bar\cA^\alpha$ throughout. The strategy of the proof is the same as in~\S\ref{SP}. The restriction to angular frequencies $\geq l$ increases the range of weights for which the zero energy operator is invertible, and the range of weights on which the low energy resolvent remains bounded. For clarity, we discuss this for fixed angular frequencies, as well as the inversion of the resolvent on borderline inputs, in the first part of the proof. In the second part, we explain how the iteration based on~\eqref{EqPRewrite1} works, and how the resolvent acting on $\cA^{4+l+\alpha}(X)$ inputs has a $\sigma^{2 l+2}\log(\sigma+i 0)$ singularity; this will imply the theorem, as shown in the third step. In the final step of the proof, we explain the minimal changes required for bounding the `infinite sum over angular frequencies $\geq l$'.

  \pfstep{Changes to the low frequency analysis.} Let us write $\Box_l$ for the restriction of $\Box$ to the space of functions supported in angular frequency $l$; thus, Lemma~\ref{LemmaABox} gives
  \begin{align*}
    \wh{\Box_l}(\sigma) &\in 2 i\sigma\rho\bigl(\rho\pa_\rho-1 + \rho^2\Diffb^1(X)\bigr) \\
      &\qquad + \rho^2\bigl(-(\rho\pa_\rho)^2+\rho\pa_\rho+l(l+1)+2\bhm(\rho\pa_\rho)^2+\rho^2\Diffb^2(X)\bigr) + \sigma^2\rho^2\CI(X).
  \end{align*}
  The normal operator of $\rho^{-2}\wh{\Box_l}(0)$, given by $L_0=-(\rho\pa_\rho)^2+\rho\pa_\rho+l(l+1)$ is thus a regular singular operator with indicial roots $-l$ and $l+1$, corresponding to solutions $\rho^{-l}Y_l=r^l Y_l$ and $\rho^{l+1}Y_l=r^{-l-1}Y_l$ (with $Y_l\in\cY_l$) in $\ker L_0$. This implies that the inverse $\wh{\Box_l}(0)\colon\cA^{\alpha+2}(X)\to\cA^{\alpha-}(X)$ for $\alpha\in(0,1)$ (see~\eqref{EqAZeroMapping}) is well-defined on a larger range of weights,
  \begin{equation}
  \label{EqPFGap}
    \wh{\Box_l}(0)^{-1}\colon\cA^{\alpha+2}(X) \to \cA^{\alpha-}(X),\qquad \alpha\in(-l,l+1).
  \end{equation}

  As for uniform low energy behavior, we recall that the interval in the condition $\ell-\nu\in(-\tfrac32,-\half)$ in Theorem~\ref{ThmALAP} is precisely the interval of weights (relative to $L^2(X)$) for which the zero energy operator is invertible; hence here we have
  \begin{equation}
  \label{EqPFEst}
    \|(\rho+|\sigma|)^\nu u\|_{\Hbext^{s,\ell}(X)} \leq C\|(\rho+|\sigma|)^{\nu-1}\wh{\Box_l}(\sigma)u\|_{\Hbext^{s,\ell+1}(X)}
  \end{equation}
  for $\ell<-\half$, $s>\half$, $s+\ell>-\half$ and the enlarged range $\ell-\nu\in(-\tfrac32-l,-\half+l)$. Using this with $\nu=0$ and $\ell\in(-\tfrac32-l,-\half)$ implies that $\wh{\Box_l}(\sigma)^{-1}\colon\cA^{2+\alpha}(X)\to\cA^{\alpha-}(X)$ is uniformly bounded for small $\sigma$ with $\Im\sigma\geq 0$ provided $\alpha\in(-l,1)$.

  When considering $\wh{\Box_l}(\sigma)^{-1}(\rho^{-l+2}Y_l)$, this uniform boundedness fails in the same manner as in Proposition~\ref{PropARhoSq}. Indeed, the corresponding model problem at $\tface\subset X_{\rm res}^+$ reads
  \[
    \sigma^2\wt{\Box_l}(1)\tilde u=\rho^{-l+2}Y_l.
  \]
  Writing $\tilde u=\sigma^{-l}\tilde u'$, this is
  \[
    \wt{\Box_l}(1)\tilde u'=\hat r^{l-2}Y_l,\qquad
    \wt{\Box_l}(1)=-2 i\hat r^{-1}\bigl(\hat r\pa_{\hat r}-1) + \hat r^{-2}\bigl(-(\hat r\pa_{\hat r})^2-\hat r\pa_{\hat r}+l(l+1)\bigr),
  \]
  where $\hat\rho=\rho/\sigma=\hat r^{-1}$, cf.\ Lemma~\ref{LemmaAModel}, with unique solution in $\cA^{1,-l-}(\tface)$. (At $\hat\rho=0$, the model operator is $2 i\hat\rho(\hat\rho\pa_{\hat\rho}-1)$, which produces $\hat\rho^1$ leading order terms there. At $\hat r=0$, uniqueness in the stated space is again due to the enlarged indicial gap arising already in~\eqref{EqPFGap}.) The leading order logarithmic term of $\tilde u'$ at $\hat r=0$ is $-(2 l+1)^{-1}\hat r^l(\log\hat r)Y_l$, and therefore $\tilde u$ has the claimed logarithmic singularity at $\zface=\hat r^{-1}(0)$, with leading order term
  \begin{equation}
  \label{EqPFModelSol}
    {-}(2 l+1)^{-1}\rho^{-l}(\log\hat r)Y_l.
  \end{equation}
  The conclusion is that
  \begin{equation}
  \label{EqPFBorderline}
    \wh{\Box_l}(\sigma)^{-1}(\rho^{-l+2}Y_l) \in \cA^{-l+1-,((-l,0),-l+1-),((0,1),1-)}(X_{\rm res}^+),
  \end{equation}
  which is the analogue of Proposition~\ref{PropARhoSq} but with spatial decay orders decreased by $l$.

  \pfstep{Expansion of the low energy resolvent.} We put $f_0:=\hat f(0)$ and $u_0:=\wh{\Box_l}(0)^{-1}f_0\in\cA^{l+1-}(X)$ using~\eqref{EqPFGap}. We claim that there exists $Y_{(0)}\in Y_l$ such that
  \begin{equation}
  \label{EqPFu0}
    u_0\in\rho^{l+1}Y_{(0)}+\bhm(l+1)\rho^{l+2}Y_{(0)}+\cA^{l+2+\alpha-}(X).
  \end{equation}
  As in the proof of Lemma~\ref{LemmaPL0inv4}, the leading order term arises from the pole of the inverse $\wh{L_0}(\xi)|_{\cY_l}^{-1}=(-(i\xi)^2+i\xi+l(l+1))^{-1}$ of the Mellin transformed normal operator of $\wh{\Box_l}(0)$ at $\xi=-i(l+1)$, while the subleading term arises from solving $L_0(c\rho^{l+2}Y_{(0)})=-\rho L_1(\rho^{l+1}Y_{(0)})=2\bhm(l+1)^2\rho^{l+2}$ for the constant $c$. (There are no other $\cO(\rho^{l+2})$ contributions since we are restricting to angular frequency $l$, and $\wh{L_0}(\xi)|_{\cY_l}^{-1}$ has no poles other than $\xi=-i(l+1),i l$.)

  To reduce the number of terms one needs to keep track of, let us modify $t_*$ so that it is equal to the null coordinate $t-r_*$ for large $r$; therefore, $|d t_*|^2\in\CIc(X^\circ)$. We then set
  \[
    f_1 := -\sigma^{-1}\bigl(\wh{\Box_l}(\sigma)-\wh{\Box_l}(0)\bigr)u_0 \in -2 i l\rho^{l+2}Y_{(0)} - 2 i\bhm(l+1)^2\rho^{l+3}Y_{(0)} + \cA^{l+3+\alpha-}(X);
  \]
  here and below we absorb the $\sigma\CIc(X^\circ)$ contribution from $\sigma^{-1}(\wh{\Box_l}(\sigma)-\wh{\Box_l}(0))$ arising from $g^{0 0}=|d t_*|^2$ into the remainder space upon allowing smooth $\sigma$-dependence. But then (just as in Lemma~\ref{LemmaPL0inv3}) we obtain a logarithmic term in
  \[
    u_1 := \wh{\Box_l}(0)^{-1}f_1 \in c_1 Y_{(0)}\rho^l + c'_1 Y_{(0)}\rho^{l+1}\log\rho + c''_1 Y_{(1)}\rho^{l+1}+ \cA^{l+1+\alpha-}(X),\quad Y_{(1)}\in\cY_l,
  \]
  where $c_1=i$ and $c'_1=\frac{2 i\bhm(l+1)^2}{2 l+1}$, and $c''_1\in\C$ arises as before from the $\rho^{l+1}\cY_l$ type asymptotics produced by the pole of $\wh{L_0}(\xi)|_{\cY_l}^{-1}$ at $\xi=-i(l+1)$. We now proceed iteratively in four steps.
  
  \pfsubstep{Step 1. Terms in the expansion with spatial decay.} For $k=1,\ldots,l-1$, suppose that
  \begin{equation}
  \label{EqPFIndHi}
    u_k \in c_k Y_{(0)}\rho^{l+1-k} + c'_k Y_{(0)}\rho^{l+2-k}\log\rho + c''_k Y_{(1),k}\rho^{l+2-k} + \cA^{l+2-k+\alpha-}(X)
  \end{equation}
  for some $c_k,c'_k,c''_k\in\C$ and $Y_{(1),k}\in\cY_l$; we shall see that the main term we are interested in the logarithmic one. Then
  \begin{align*}
    f_k &:= -\sigma^{-1}\bigl(\wh{\Box_l}(\sigma)-\wh{\Box_0}(\sigma)\bigr)u_k \\
      &\in d_k Y_{(0)}\rho^{l+2-k}+d'_k Y_{(0)}\rho^{l+3-k}\log\rho + d''_k Y_{(1),k}\rho^{l+3-k} + \cA^{l+3+\alpha-k-}(X).
  \end{align*}
  Using~\eqref{EqPFGap} and the absence of indicial roots in $(-l,l+1)$, the next iterate $u_{k+1}:=\wh{\Box_l}(0)^{-1}f_k$ lies in the space~\eqref{EqPFIndHi} with $k+1$ in place of $k$.\footnote{Note that the resolvent expansion is, schematically, $u_0+\sigma u_1+\dots+\sigma^k u_k+\dots$, thus~\eqref{EqPFIndHi} showcases the usual gain of regularity in $\sigma$ at the cost of loss of decay in $\rho$.} This establishes~\eqref{EqPFIndHi} all the way up to $u_l\in c_l Y_{(0)}\rho+c'_l Y_{(0)}\rho^2\log\rho+c''_l Y_{(1),l}\rho^2+\cA^{\alpha+2-}(X)$.

  \pfsubstep{Step 2. The $(l+1)$-st term.} The next one and a half iterations require special attention: the normal operator of $\sigma^{-1}\bigl(\wh{\Box_l}(\sigma)-\wh{\Box_0}(\sigma))$, namely $2 i\rho(\rho\pa_\rho-1)$, annihilates the leading order term of $u_l$, thus
  \[
    f_{l+1} := -\sigma^{-1}\bigl(\wh{\Box_l}(\sigma)-\wh{\Box_l}(0)\bigr)u_l \in d'_{l+1} Y_{(0)}\rho^3\log\rho + d''_{l+1}Y_{(1),l+1}\rho^3 + \cA^{\alpha+3-}(X)
  \]
  gains almost $2$ orders of decay relative to $u_l$, rather than just $1$. (This is a key observation giving us an extra order of regularity for the resolvent; it was already used for $l=0$ in~\eqref{EqPf2}.) Therefore, we have
  \[
    u_{l+1} := \wh{\Box_l}(0)^{-1}f_{l+1} \in c'_{l+1}Y_{(0)}\rho\log\rho + c''_{l+1}Y_{(0)}\rho + \cA^{\alpha+1-}(X),
  \]
  and subsequently $\rho\pa_\rho-1$ gets rid of the logarithm, to wit
  \[
    f_{l+2} := -\sigma^{-1}\bigl(\wh{\Box_l}(\sigma)-\wh{\Box_l}(0)\bigr)u_l \in d'_{l+2}Y_{(0)}\rho^2 + \cA^{\alpha+2-}(X).
  \]
  
  \pfsubstep{Step 3. Terms in the expansion with spatial growth.} An induction on $k=l+2,\ldots,2 l+1$ now shows that
  \begin{alignat*}{2}
    u_k &:= \wh{\Box_l}(0)^{-1}f_k &&\in c'_k Y_{(0)}\rho^{l+2-k} + \cA^{l+2-k+\alpha-}(X), \\
    f_{k+1} &:= -\sigma^{-1}\bigl(\wh{\Box_l}(\sigma)-\wh{\Box_l}(0)\bigr)u_k &&\in d'_{k+1}Y_{(0)}\rho^{l+3-k}+\cA^{l+3-k+\alpha-}(X)
  \end{alignat*}

  \pfsubstep{Step 4. The appearance of the logarithm.} Finally, we can no longer apply $\wh{\Box_l}(0)^{-1}$ to $f_{2 l+2}\in d'_{2 l+2}Y_{(0)}\rho^{-l+2}+\cA^{-l+2+\alpha-}(X)$; rather, we stop the expansion at this step and put
  \[
    u_{2 l+2}(\sigma) := \wh{\Box_l}(\sigma)^{-1}f_{2 l+2}.
  \]
  In view of~\eqref{EqPFBorderline}, this has a logarithmic singularity at $\sigma=0$.

  \pfstep{Sharp asymptotics in compact spatial sets.} Restricting to compact spatial sets, we now have established the following analogue of~\eqref{EqPWLo}:
  \begin{equation}
  \label{EqPFExp}
    \wh{\Box_l}(\sigma)^{-1}\hat f(0) \in \sum_{k=0}^{2 l+1} \sigma^k u_k - u_{(l)}\sigma^{2 l+2}\log(\sigma+i 0) + \sum_\pm \sigma^{2 l+2}\cA^{((0,0),\alpha-)}(\pm[0,1);\CI(X^\circ)).
  \end{equation}
  Here, $u_{(l)}$ can be determined by following the calculations \eqref{EqARhoSqErr} and \eqref{EqARhoSqLog} mutatis mutandis: in view of~\eqref{EqPFModelSol} for $Y_l=(2 l+1)^{-1}d'_{2 l+2}Y_{(0)}$, $u_{(l)}$ can be expressed using a cutoff $\chi_\pa(\rho)$, identically $1$ for small $\rho$, as
  \[
    u_{(l)} = \rho^{-l}Y_l\chi_\pa - \wh{\Box_l}(0)^{-1}\bigl(\wh{\Box_l}(0)(\rho^{-l}Y_l\chi_\pa)\bigr),
  \]
  where $\wh{\Box_l}(0)^{-1}\colon\cA^{-l+3}(X)\to\cA^{-l+1-}(X)$ is the inverse~\eqref{EqPFGap}. Therefore, $u_{(l)}$ is the (unique) element of $\ker\wh{\Box_l}(0)$ with asymptotics $\rho^{-l}Y_l$, that is,
  \begin{equation}
  \label{EqPFul}
    \wh{\Box_l}(0)u_{(l)}=0,\qquad u_{(l)}-(2 l+1)^{-1}d'_{2 l+2}\rho^{-l}Y_{(0)}\in\cA^{-l+1-}(X).
  \end{equation}

  In the expansion~\eqref{EqPFExp}, the $+i 0$ nature of the logarithmic singularity is forced by causality considerations as in Remark~\ref{RmkPWConst}. Moreover, the $\sigma$-regularity of the remainder term arises from the fact that $\wh{\Box_l}(\sigma)^{-1}$ acts on the conormal remainder term of $f_{2 l+2}$ by
  \[
    \wh{\Box_l}(\sigma)^{-1}\cA^{-l+2+\alpha-}(X)\subset\cA^{((0,0),\alpha-)}\bigl(\pm[0,1)_\sigma;\cA^{-l+\alpha-}(X)\bigr),
  \]
  as can be proved by direct differentiation in $\sigma$ exactly like in Proposition~\ref{PropALoSmall}, now using~\eqref{EqPFEst} for $\ell=-3/2-l+\alpha-\delta$, $\nu=0$, with $0<\delta\leq\alpha$ arbitrary.

  The term $\wh{\Box_l}(\sigma)^{-1}(\hat f(\sigma)-\hat f(0))$ has an additional order of vanishing in $\sigma$, hence can be absorbed into the error term in~\eqref{EqPFExp}. Upon taking the inverse Fourier transform, the resolvent expansion~\eqref{EqPFExp} thus implies the desired $t^{-2 l-3}$ decay of the forward solution $\phi$ of $\Box\phi=f$; and the explicit description~\eqref{EqPFul} justifies Remark~\ref{RmkPFComplete}.
  
  We argue that this decay rate is generically sharp. It follows from the argument that all constants $c'_k$ and $d'_k$ are nonzero for $k\geq 1$, hence so is $u_{(l)}$ in view of~\eqref{EqPFul}, provided only that $Y_{(0)}\neq 0$. But a pairing argument completely analogous to~\eqref{EqPL0invLot} (with the role of $u_{(0)}^*$ now played by $u_{(l)}^*(Y)\in\ker\wh{\Box_l}(0)$, $Y\in\cY_l$, with leading order behavior $\rho^{-l}Y$) produces a formula for $Y_{(0)}$: identifying $\cY_l\cong\C^{2 l+1}$ via the basis $Y_{l m}$, it reads $Y_{(0)}=c\la f,u_{(l)}^*(Y_{l m})\ra$, where $c\neq 0$ is an explicit constant. Therefore, $\phi$ has a \emph{nontrivial} $t^{-2 l-3}$ leading order term plus a $\cO(t^{-2 l-3-\alpha+})$ remainder, \emph{unless} $f$ satisfies the $2 l+1$ linearly independent constraints $\la f,u_{(l)}^*(Y_{l m})\ra=0$. Since the leading order term $u_{(l)}$ has radial dependence given by the solution of an ODE, it cannot vanish on a nonempty open subset of $X^\circ$ unless it vanishes identically.

  \pfstep{Decay in the full future causal cone.} With $M_+$ defined as in Definition~\ref{DefPWComp} (but with $X$ as in~\eqref{EqKDCompact}), we proved that the leading order term at $\cK^+$ is a constant multiple of $t_*^{-2 l-3}u_{(l)}$, which is of size $(\rho_{\cK^+}\rho_{I^+})^{2 l+3} \rho_{I^+}^{-l}=\rho_{\cK^+}^{2 l+3}\rho_{I^+}^{l+3}$, thus suggesting $\rho_{I^+}^{l+3}$ decay at $I^+$. This is confirmed by the following more precise argument. The contribution of the model solution at $\tface$ to the resolvent expansion has leading order term at $\tfrac{\sigma}{\rho}=0$ given by $\sigma^{2 l+2}\cdot\rho^{-l}\log(\tfrac{\sigma}{\rho}+i 0)=\rho^{l+3}\cdot(\tfrac{\sigma}{\rho})^{2 l+2}\log(\tfrac{\sigma}{\rho}+i 0)$ in view of~\eqref{EqPFBorderline}; via the calculation~\eqref{EqPWGlobalHatrFT}, this gives the leading order term $\rho^{l+3}\cdot f(\rho t_*)$ at $I^+$ with $|f(\rho t_*)|\lesssim\la\rho t_*\ra^{-2 l-3}$, as claimed.

  This in particular implies $t_*^{-l-3}$ decay of $\phi$ as $t_*\to\infty$ for $\rho t_*$ restricted to compact subsets of $(0,\infty)$; the claimed $t_*^{-l-2}$ decay of the radiation field, or equivalently the $t_*^{-l-3}$ decay of $\lim_{r\to\infty} (\rho t_*)^{-1}\phi$, follows from the $t_*^{-l-3}$ decay of $\phi$ towards $(I^+)^\circ$ as in the last step of the proof of Theorem~\ref{ThmPWGlobal}.

  \pfstep{Modifications for forcing supported in angular frequencies $\geq l$.} One can analyze $\wh\Box(\sigma)$ directly on the subspace of $\CIc(X^\circ)$ consisting of functions supported in angular frequencies $\geq l$ and follow the above arguments. The only modification is that $u_0$ in~\eqref{EqPFu0} attains an additional contribution $\rho^{l+2}Y'_{(0)}$ with $Y'_{(0)}\in\cY_{l+1}$; one can keep track of the effect of this term simply by using the above arguments for angular frequency equal to $l+1$. In particular, it does not contribute to the most singular $\sigma^{2 l+2}\log\sigma$ term of the resolvent.
\end{proof}

\begin{cor}
\label{CorPF}
  Let $l\in\N_0$, and consider initial data $\phi_0,\phi_1\in\CIc((2\bhm,\infty)\times\Sph^2)$ which are supported in angular frequency $l$. Then the solution $\phi$ of the initial value problem
  \begin{equation}
  \label{EqPFIVP}
    \Box\phi = 0,\quad
    \phi(t=0,x) = \phi_0(x),\quad
    \pa_t \phi(t=0,x) = \phi_1(x),
  \end{equation}
  obeys pointwise decay bounds $|\phi(t_*,r,\omega)|\leq C t_*^{-2 l-3}$ for $x=(r,\omega)$ restricted to any fixed compact subset of $X^\circ$. The same decay holds for derivatives of $\phi$ of any order along $t_*\pa_{t_*}$, $\pa_r$, and spherical vector fields. If $\phi$ is initially static, i.e.\ $\phi_1\equiv 0$, then the stronger decay
  \begin{equation}
  \label{EqPFStatic}
    |\phi(t_*,r,\omega)|\leq C t_*^{-2 l-4}\qquad (\phi_1\equiv 0)
  \end{equation}
  holds for $\phi$ and all its derivatives as above. These decay rates are generically sharp.
\end{cor}
\begin{proof}
  We only need to prove~\eqref{EqPFStatic}. We can choose $R_0<R_1$ so that $R_0<r<R_1$ on $\supp\phi_j$, $j=0,1$, and then define $t_*$ to be equal to the static time coordinate $t$ for $r\in[R_0,R_1]$; thus, Theorem~\ref{ThmPF} is applicable. The reduction of~\eqref{EqPFIVP} to a forcing problem in static coordinates reads
  \begin{equation}
  \label{EqPFForcing}
    \Box(H(t) \phi) = [\Box,H(t)]\phi = \bigl[\bigl(1-\tfrac{2\bhm}{r}\bigr)^{-1}\pa_t^2,H(t)]\phi = \bigl(1-\tfrac{2\bhm}{r}\bigr)^{-1}\bigl(\delta(t)\phi_1(x) + \delta'(t)\phi_0(x)\bigr);
  \end{equation}
  there are no spatial derivatives falling on $\phi_0,\phi_1$ since mixed time-space derivatives are absent in the Schwarzschild wave operator in static coordinates. For initially static perturbations, the right hand side is $(1-\tfrac{2\bhm}{r})^{-1}\phi_0(x)\delta'(t)$, and the point is that its Fourier transform vanishes simply at $\sigma=0$. The extra factor of $\sigma$ of the forcing on the spectral side renders the main singularity of the resolvent a multiple of $\sigma\cdot\sigma^{2 l+2}\log(\sigma+i 0)$, thus giving an extra order of time decay upon taking the inverse Fourier transform.
\end{proof}

\bibliographystyle{alpha}


\end{document}